\documentclass{article}

\usepackage{fullpage}
\usepackage[final]{hyperref}
\usepackage{nameref}
\usepackage[numbers]{natbib}
\usepackage{tabularx}
\usepackage{booktabs}
\usepackage{setspace}
\usepackage{multirow}
\usepackage{makecell}
\usepackage{graphicx}
\usepackage{delarray}
\usepackage{xcolor}
\usepackage{colortbl}
\usepackage{amssymb}
\usepackage{amsmath}
\usepackage{mathtools}
\usepackage{nicefrac}       
\usepackage{tikz}
\usepackage{xspace}
\xspaceaddexceptions{]\}}
\usepackage{algorithm}
\usepackage{algorithmic}
\usepackage{mdframed}
\usepackage{bm}
\usepackage{tablefootnote}

\newcommand{\cC}{{\mathcal C}}
\newcommand{\calC}{{\mathcal C}}
\newcommand{\cD}{{\mathcal D}}

\newtheorem{definition}{Definition}
\newtheorem{assumption}{Assumption}
\newtheorem{theorem}{Theorem}
\newtheorem{lemma}{Lemma}
\newtheorem{proposition}{Proposition}
\newtheorem{corollary}{Corollary}

\newenvironment{proof}{\noindent {\bf \em Proof: }\ignorespaces}%
{\hspace*{\fill}$\Box$\par}
{\hspace*{\fill}$\Box$\par\vspace{4mm}}
\newenvironment{proofof}[1]{\smallskip\noindent{\bf \em Proof of #1.}}%
{\hspace*{\fill}$\Box$\par}

\hypersetup{
	colorlinks=true,       
	linkcolor=blue,        
	citecolor=blue,        
	filecolor=magenta,     
	urlcolor=blue
}

\numberwithin{equation}{section}

\newcommand{\eat}[1]{}
\usepackage{todonotes}

\newcommand{\topic}[1]{\vspace{2mm}\noindent{{\bf #1:}}}

\newcommand{\blue}[1]{\textcolor{blue}{#1}}
\definecolor{bgcolor}{rgb}{0.66,0.88,1.00}

\newcommand{\E}{{\mathbb{E}}}
\newcommand{\R}{{\mathbb R}}

\newcommand{\inner}[2]{\langle #1,#2 \rangle}

\newcommand{\ns}[1]{\| #1 \|^2}
\newcommand{\nsB}[1]{\left\| #1 \right\|^2}
\newcommand{\n}[1]{\| #1 \|}

\newcommand{\hx}{\widehat{x}}

\newcommand{\bL}{\bar{L}}

\newcommand{\fgap}{{\Delta_{0}}}
\newcommand{\fgapp}{{\Delta'_{0}}}

\newcommand{\xk}{x^k}
\newcommand{\xkn}{x^{k+1}}
\newcommand{\gk}{g^k}
\newcommand{\Ek}{\E_k}
\newcommand{\fs}{f^*}
\newcommand{\sk}{\sigma_k^2}
\newcommand{\skn}{\sigma_{k+1}^2}
\newcommand{\Dk}{{\Delta^{k}}}
\newcommand{\Dkn}{{\Delta^{k+1}}}
\newcommand{\wk}{w^k}
\newcommand{\wkn}{w^{k+1}}
\newcommand{\wik}{w_i^k}
\newcommand{\wikn}{w_i^{k+1}}
\newcommand{\wjk}{w_j^k}

\newcommand{\wijk}{w_{i,j}^{k}}
\newcommand{\wijkn}{w_{i,j}^{k+1}}

\newcommand{\gi}{\widetilde{g}_i^k}
\newcommand{\Di}{\widehat{\Delta}_i^k}
\newcommand{\tsk}{\widetilde{\sigma}_{k}^2}
\newcommand{\tskn}{\widetilde{\sigma}_{k+1}^2}
\newcommand{\Ci}{\cC_i^k}
\newcommand{\hi}{h_i^k}
\newcommand{\hin}{h_i^{k+1}}
\newcommand{\summ}{\sum_{i=1}^m}
\newcommand{\EkB}[1]{\Ek\left[ #1 \right]}

\newcommand{\ski}{\sigma_{k,i}^2}
\newcommand{\skin}{\sigma_{k+1,i}^2}

\begin{document}
	
\title{\bf A Unified Analysis of Stochastic Gradient Methods \\ for Nonconvex Federated Optimization}
\author{\qquad\quad Zhize Li   \qquad\quad  Peter Richt{\'a}rik \\
	 	King Abdullah University of Science and Technology (KAUST)
}

\date{June 12, 2020}
\maketitle

\begin{abstract}
In this paper, we study the performance of a large family of SGD variants in the smooth nonconvex regime. 
To this end, we propose a generic and flexible assumption capable of accurate modeling of the second moment of the stochastic gradient. 
Our assumption is satisfied by a large number of specific variants of SGD in the literature, including SGD with arbitrary sampling, SGD with compressed gradients, and a wide variety of variance-reduced SGD methods such as SVRG and SAGA. 
We provide a single convergence analysis for all methods that satisfy the proposed unified assumption, thereby offering a unified understanding of SGD variants in the nonconvex regime instead of relying on dedicated analyses of each variant. 
Moreover, our unified analysis is accurate enough to recover or improve upon the best-known convergence results of several classical methods, and also gives new convergence results for many new methods which arise as special cases. 
In the more general distributed/federated nonconvex optimization setup, we propose two new general algorithmic frameworks differing in whether direct gradient compression (DC) or compression of gradient differences (DIANA) is used. 
We show that all methods captured by these two frameworks also satisfy our unified assumption. 
Thus, our unified convergence analysis also captures a large variety of distributed methods utilizing compressed communication. 
Finally, we also provide a unified analysis for obtaining faster linear convergence rates in this nonconvex regime under the PL condition.
\end{abstract}

\vspace{5mm}
\tableofcontents

\newpage
\section{Introduction}
\label{sec:intro}

In this paper, we develop a general framework for studying and designing  SGD-type  methods for solving {\em nonconvex distributed/federated optimization problems} \cite{khirirat2018distributed, FEDLEARN, kairouz2019advances}. Given $m$ machines/workers/devices, each having access to their own  data samples, we consider the problem
\begin{equation}\label{eq:prob-fed}
  \min_{x\in \R^d} \left\{  f(x) := \frac{1}{m}\sum \limits_{i=1}^m{f_i(x)}   \right\}
\end{equation}
in the heterogeneous (non-IID) data setting, i.e., we allow different workers to have access to different data distributions. We consider the case when the loss $f_i$ in worker $i$ is of an online/expectation form, 
\begin{align}
f_i(x) := \E_{\zeta \sim \cD_i}[f_i(x,\zeta)] , \label{prob-fed:exp}
\end{align}
and also the case when $f_i$ is of a finite-sum form,
\begin{align}
  f_i(x) := \frac{1}{n}\sum \limits_{j=1}^n{f_{i,j}(x)}, \label{prob-fed:finite}
\end{align}
where  $f(x), f_i(x), f_i(x,\zeta)$ and  $f_{i,j}(x)$ are possibly nonconvex functions. Forms \eqref{prob-fed:exp} and \eqref{prob-fed:finite} capture the population (resp.\ empirical) risk minimization problems in distributed/federated learning. 

\subsection{Single machine setting}
In particular, the single machine/node case (i.e., $m=1$) of problem \eqref{eq:prob-fed} reduces to the standard problem
\begin{equation}\label{eq:prob}
\min_{x\in \R^d}   f(x),
\end{equation}
where $f(x)$ can be the online/expectation form 
\begin{align}
f(x) := \E_{\zeta\sim \cD}[f(x,\zeta)]  \label{prob:exp}
\end{align}
or the finite-sum form
\begin{align}
 f(x) := \frac{1}{n}\sum \limits_{j=1}^n{f_j(x)}, \label{prob:finite}
\end{align}
where $f(x), f(x,\zeta)$ and  $f_j(x)$  are possibly nonconvex functions.
These forms capture the standard population/empirical risk minimization problems in machine learning.

There has been extensive research into solving the standard problem \eqref{eq:prob}--\eqref{prob:finite} and an enormous number of methods were proposed, e.g., \citep{nesterov2014introductory, nemirovski2009robust, ghadimi2013stochastic, johnson2013accelerating, defazio2014saga, nguyen2017sarah, ge2015escaping,  lin2015universal, lan2015optimal, lan2018random, zhize2019unified, allen2017katyusha, zhize2020anderson, ghadimi2016mini,zhou2018stochastic, fang2018spider,zhize2019ssrgd, pham2019proxsarah}.
Due to the increasing popularity of distributed/federated learning, the more general distributed/federated optimization problem \eqref{eq:prob-fed}--\eqref{prob-fed:finite} has attracted significant attention as well~  \citep{FEDLEARN, FL2017-AISTATS, lian2017can, lan2018random, li2018federated, localSGD-Stich, mishchenko2019distributed, DIANA2, karimireddy2019scaffold, khaled2019first, yang2019federated, li2019communication, kairouz2019advances, li2020acceleration, localSGD-AISTATS2020}. 
However, all these methods are analyzed separately, often using different approaches, intuitions, and assumptions, and separately in the $m=1$ (single node) and $m\geq 1$ case.

\subsection{Our contributions}
\label{sec:contribution}

We provide a {\em single and sharp analysis for a large family of SGD methods (Algorithm~\ref{alg:1}) for solving the nonconvex problem \eqref{eq:prob-fed}.}  Our approach offers a {\em unified understanding} of many previously proposed SGD variants, which we believe helps the community making better sense of existing methods and results. More importantly, {\em our unified approach also motivates and facilitates the design of, and offers plug-in convergence guarantees for,  many new and practically relevant SGD variants}.

\begin{algorithm}[t]
	\caption{Framework of stochastic gradient methods}
	\label{alg:1}
	\begin{algorithmic}[1]
		\REQUIRE ~
		initial point $x^0$, stepsize $\eta_k$
		\FOR {$k=0,1,2,\ldots$}
		\STATE Compute stochastic gradient $g^k$  \label{line:grad}
		\STATE $x^{k+1} = x^k - \eta_k g^k$  \label{line:update}
		\ENDFOR
	\end{algorithmic}
\end{algorithm}

While Algorithm~\ref{alg:1} has a seemingly tame structure, the complication arises due to the fact that there is a potentially infinite number of meaningful and yet sharply distinct ways in which the gradient estimator $g^k$ can be defined. The selection of an appropriate estimator  is a very active and important area of research, as it directly impacts  many aspects of the  algorithm it gives rise to, including tractability, memory footprint, per iteration cost, parallelizability, iteration complexity, communication complexity, sample complexity and generalization.

The key technical idea of our approach is the design of a {\em flexible, tractable and accurate  parametric model} capturing the behavior of the stochastic gradient. We want the model to be {\em flexible} in order to be able to  describe many existing and have the potential to describe many  variants of SGD. As we shall see, flexibility is achieved by the inclusion of a number of parameters. We want the model to be {\em tractable}, meaning that it needs to act as an assumption which can be used to perform a theoretical complexity analysis. Finally, we want the complexity results to be {\em accurate}, i.e., we want to recover best known rates for existing methods, and obtain sharp and useful rates with predictive power for new methods. Our parametric model is described in Assumption~\ref{asp:boge}, and as we argue throughout the paper and appendices, it is indeed flexible, tractable and accurate.

\begin{assumption}[Gradient estimator]\label{asp:boge}
	The gradient estimator $g^k$ in  Algorithm~\ref{alg:1} is unbiased, i.e., 
	$\E_k[g^k] = \nabla f(x^k)$,   
	and there exist non-negative constants $A_1, A_2, B_1, B_2, C_1,C_2,D_1,\rho$ and a random sequence $\{\sigma_k^2\}$ such that the following two inequalities hold
	\begin{align}
	\E_k[\ns{g^k}] & \leq 2A_1(f(x^k)-f^*)+B_1\ns{\nabla f(x^k)} + D_1 \sigma_k^2 +C_1,  \label{eq:boge1} \\
	\E_k[\sigma_{k+1}^2] & \leq (1-\rho)\sigma_k^2 + 2A_2(f(x^k)-f^*)+B_2\ns{\nabla f(x^k)} +C_2. \label{eq:boge2}
	\end{align}
\end{assumption}

{\vspace{1mm}\noindent\bf Flexibility:} 
Our model for the behavior of the stochastic gradient for nonconvex optimization, as captured by Assumption~\ref{asp:boge}, is satisfied by a large number of specific variants of SGD proposed in the literature, including SGD with arbitrary sampling \citep{qian2019svrg, gower2019sgd, gorbunov2019unified, khaled2020better}, SGD with compressed gradients \citep{alistarh2017qsgd, wen2017terngrad, bernstein2018signsgd, khirirat2018distributed, SEGA, Cnat}, and a wide variety of variance-reduced SGD methods such as SVRG \citep{johnson2013accelerating}, SAGA \citep{defazio2014saga} and their variants (e.g., \citep{kovalev2019don, reddi2016stochastic, reddi2016proximal, allen2016variance, lei2017non, li2018simple, ge2019stable, mishchenko2019distributed, DIANA2}). Specific methods vary in the parameters for which  recurrences \eqref{eq:boge1} and \eqref{eq:boge2} are satisfied. For example, SGD variants not employing  variance reduction will generally have $D_1=0$, and  recurrence \eqref{eq:boge2} will not be used (i.e., we can ignore it and set $\rho =1$, $A_2=0$, $B_2=0$ and $C_2=0$). This setting was considered in \cite{khaled2020better}, and was an inspiration for our work. If variance reduction is applied, then $D_1>0$ and typically $C_1=0$, and  recurrence \eqref{eq:boge2} describes the variance reduction process, with parameter $\rho$ describing the speed of variance reduction. If $C_2>0$, variance reduction is not perfect. If $C_2=0$ as well, then the methods will be fully variance reduced, which typically means faster convergence rate. The specific values of all the parameters depend on how the stochastic gradient $g^k$ is constructed (e.g., via minibatching, importance sampling, variance reduction, perturbation, compression).

We design several new methods, with gradient estimators that fit Assumption~\ref{asp:boge}, for solving the general nonconvex distributed/federated problem \eqref{eq:prob-fed}--\eqref{prob-fed:finite} using compressed (e.g., quantized or sparsified) gradient communication, which is of import when training deep learning models. We adopt a direct compression (DC) framework \cite{alistarh2017qsgd, khirirat2018distributed}, and a compression of gradient differences framework (DIANA)  \cite{mishchenko2019distributed, DIANA2}.  We develop several new specific methods belonging to the  DC framework (Algorithm~\ref{alg:dc}) and DIANA framework (Algorithm \ref{alg:diana}), show that they all satisfy Assumption \ref{asp:boge}, and thus are also captured by our unified analysis.

{\vspace{1mm}\noindent\bf Tractability:} 
We use our unified assumption  to prove four complexity theorems: Theorems~\ref{thm:main}, \ref{thm:main-pl-dec}, \ref{thm:dc-diff},  and \ref{thm:diana-type-diff}. Theorem~\ref{thm:main} is the main theorem, and Theorem~\ref{thm:main-pl-dec} is used to obtain sharper results under the PL condition. Theorems~\ref{thm:dc-diff}  and \ref{thm:diana-type-diff} are used in combination with the previous generic Theorems \ref{thm:main} and \ref{thm:main-pl-dec} to obtain specialized results for distributed/federated optimization utilizing either direct gradient compression (DC framework (Algorithm~\ref{alg:dc})), or compression of gradient  differences (DIANA framework (Algorithm \ref{alg:diana})), respectively. In  Tables \ref{table:1}--\ref{table:fed-pl} we visualize  how these theorems lead to corollaries which describe the detailed complexity results of various existing and new methods.

{\vspace{1mm}\noindent\bf Accuracy:} 
For all existing methods, the rates we obtain using our general analysis match the best known rates.

\vspace{-2mm}
\begin{table}[!h]
	\centering
	\caption{Selected  methods that fit our unified analysis framework for \emph{nonconvex optimization} ($m=1$, i.e., single node).} 
	\label{table:1}
	\vspace{1mm}
	\begin{tabular}{|c|c|c|c|c|c|c|}
		\hline
		Problem & Assumption & Method & Algorithm & \multicolumn{2}{c|}{Convergence result} & Recover \\
		\hline
		\eqref{eq:prob} 
		& Asp \ref{asp:lsmooth}
		&  GD 
		&  Alg  \ref{alg:gd}
		& \multirow{4}{*}{Thm \ref{thm:main}}
		& Cor \ref{cor:gd}
		& \citep{nesterov2014introductory}  \\ \cline{1-4}\cline{6-7}
		
		\eqref{eq:prob} with \eqref{prob:exp} or \eqref{prob:finite} 
		& Asp \ref{asp:lsmooth}
		&  SGD
		&  Alg  \ref{alg:sgd}
		& 
		& Cor \ref{cor:sgd}
		& \citep{ghadimi2016mini, khaled2020better} \\ \cline{1-4}\cline{6-7}
		
		\eqref{eq:prob} with \eqref{prob:finite} 
		& Asp \ref{asp:avgsmooth}
		&  L-SVRG
		&  Alg  \ref{alg:lsvrg}
		& 
		&Cor \ref{cor:lsvrg}
		& \citep{reddi2016stochastic, allen2016variance, li2018simple, qian2019svrg} \\ \cline{1-4}\cline{6-7}
		
		\eqref{eq:prob} with \eqref{prob:finite} 
		& Asp \ref{asp:avgsmooth-saga}
		&  SAGA
		&  Alg  \ref{alg:saga}
		& 
		&Cor \ref{cor:saga}
		& \citep{reddi2016proximal} \\
		\hline
	\end{tabular}
	
	\vspace{1mm}
	\centering
	\caption{Selected methods that fit our unified analysis framework for \emph{nonconvex distributed/federated optimization} ($m\geq 1$, i.e., any number of nodes).} 
	\label{table:fed}
	\vspace{1mm}
	\begin{tabular}{|c|c|c|c|c|c|c|}
		\hline
		Problem & Assumption & Method & Algorithm & \multicolumn{2}{c|}{Convergence result} & Recover \\
		\hline
		\eqref{eq:prob-fed} 
		& Asp \ref{asp:lsmooth-diana}
		&  DC-GD
		&  Alg \ref{alg:dc-gd} 
		& \multirow{4}{*}{Thm \ref{thm:main}, \ref{thm:dc-diff}}
		& Cor \ref{cor:dc-gd} 
		&  \citep{khaled2020better} \\ \cline{1-4}\cline{6-7}
		
		\eqref{eq:prob-fed} with \eqref{prob-fed:exp} or \eqref{prob-fed:finite} 
		& Asp \ref{asp:lsmooth-diana}
		&  DC-SGD
		&  Alg \ref{alg:dc-sgd} 
		& 
		& Cor \ref{cor:dc-sgd} 
		& \citep{khaled2020better, Cnat} \\ \cline{1-4}\cline{6-7}
		
		\eqref{eq:prob-fed} with \eqref{prob-fed:finite} 
		& Asp \ref{asp:lsmooth-diana}, \ref{asp:avgsmooth-fed}
		&  DC-LSVRG
		&  Alg \ref{alg:dc-lsvrg} 
		& 
		& Cor \ref{cor:dc-lsvrg} 
		& \bf{New} \\ \cline{1-4}\cline{6-7}
		
		\eqref{eq:prob-fed} with \eqref{prob-fed:finite}
		& Asp \ref{asp:lsmooth-diana}, \ref{asp:avgsmooth-saga-fed}
		&  DC-SAGA
		&  Alg \ref{alg:dc-saga} 
		& 
		& Cor \ref{cor:dc-saga} 
		& \bf{New} \\ 
		\hline
		\hline
		\eqref{eq:prob-fed} 
		& Asp \ref{asp:lsmooth-diana}
		&  DIANA-GD
		&  Alg \ref{alg:diana-gd} 
		& \multirow{4}{*}{Thm \ref{thm:main}, \ref{thm:diana-type-diff}}
		& Cor \ref{cor:diana-gd} 
		& \bf{New} \\ \cline{1-4}\cline{6-7}
		
		\eqref{eq:prob-fed} with \eqref{prob-fed:exp} or \eqref{prob-fed:finite} 
		& Asp \ref{asp:lsmooth-diana}
		&  DIANA-SGD
		&  Alg \ref{alg:diana-sgd} 
		& 
		& Cor \ref{cor:diana-sgd} 
		& \bf{New} \\ \cline{1-4}\cline{6-7}
		
		\eqref{eq:prob-fed} with \eqref{prob-fed:finite} 
		& Asp \ref{asp:lsmooth-diana}, \ref{asp:avgsmooth-fed}
		&  DIANA-LSVRG
		&  Alg \ref{alg:diana-lsvrg} 
		& 
		& Cor \ref{cor:diana-lsvrg} 
		& \bf{New}$^\dagger$ \\ \cline{1-4}\cline{6-7}
		
		\eqref{eq:prob-fed} with \eqref{prob-fed:finite}
		& Asp \ref{asp:lsmooth-diana}, \ref{asp:avgsmooth-saga-fed}
		&  DIANA-SAGA
		&  Alg \ref{alg:diana-saga} 
		& 
		& Cor \ref{cor:diana-saga} 
		& \bf{New}$^\dagger$ \\
		\hline
	\end{tabular}
	\vspace{0.5mm}
	{\begin{spacing}{0.5}\footnotesize$^\dagger$We want to mention that \citet{DIANA2} studied a weak version of DIANA-LSVRG and DIANA-SAGA with minibatch size $b=1$ (non-minibatch version).
	See Section \ref{sec:diana} for more details.
	\end{spacing}}
	
	\vspace{5mm}
	\centering
	\caption{Selected methods that fit our unified analysis framework for {\em nonconvex optimization under the PL condition} ($m=1$).} 
	\label{table:pl}
	\vspace{1mm}
	\begin{tabular}{|c|c|c|c|c|c|c|}
		\hline
		Problem & Assumption & Method & Algorithm &  \multicolumn{2}{c|}{Convergence result} & Recover \\
		\hline
		\eqref{eq:prob} 
		& Asp \ref{asp:lsmooth}, \ref{asp:pl}
		&  GD
		&  Alg  \ref{alg:gd}
		& \multirow{4}{*}{Thm \ref{thm:main-pl-dec}}
		& Cor \ref{cor:gd-pl}
		& \citep{polyak1963gradient,karimi2016linear} \\\cline{1-4}\cline{6-7}
		
		\eqref{eq:prob} with \eqref{prob:exp} or \eqref{prob:finite} 
		& Asp \ref{asp:lsmooth}, \ref{asp:pl}
		&  SGD
		&  Alg  \ref{alg:sgd}
		& 
		& Cor \ref{cor:sgd-pl}
		& \citep{khaled2020better} \\ \cline{1-4}\cline{6-7}
		
		\eqref{eq:prob} with \eqref{prob:finite} 
		& Asp \ref{asp:avgsmooth}, \ref{asp:pl}
		&  L-SVRG
		&  Alg  \ref{alg:lsvrg}
		& 
		& Cor \ref{cor:lsvrg-pl}
		& \citep{reddi2016proximal,li2018simple} \\ \cline{1-4}\cline{6-7}
		
		\eqref{eq:prob} with \eqref{prob:finite} 
		& Asp \ref{asp:avgsmooth-saga}, \ref{asp:pl}
		&  SAGA
		&  Alg  \ref{alg:saga}
		& 
		& Cor \ref{cor:saga-pl}
		& \citep{reddi2016proximal} \\
		\hline
	\end{tabular}
	
	\vspace{1mm}
	\centering
	\caption{Selected methods that fit our unified analysis framework for \emph{nonconvex distributed/federated optimization under PL condition} ($m\geq 1$).} 
	\label{table:fed-pl}
	\vspace{1mm}
	\begin{tabular}{|c|c|c|c|c|c|c|}
		\hline
		Problem & Assumption & Method & Algorithm & \multicolumn{2}{c|}{Convergence result} & Recover \\
		\hline
		\eqref{eq:prob-fed} 
		& Asp \ref{asp:lsmooth-diana}, \ref{asp:pl}
		&  DC-GD
		&  Alg \ref{alg:dc-gd}
		& \multirow{4}{*}{Thm \ref{thm:main-pl-dec}, \ref{thm:dc-diff}}
		& Cor \ref{cor:dc-gd-pl} 
		& \bf{New} \\ \cline{1-4}\cline{6-7}
		
		\eqref{eq:prob-fed} with \eqref{prob-fed:exp} or \eqref{prob-fed:finite} 
		& Asp \ref{asp:lsmooth-diana}, \ref{asp:pl}
		&  DC-SGD
		&  Alg \ref{alg:dc-sgd}
		& 
		& Cor \ref{cor:dc-sgd-pl} 
		& \bf{New}\\ \cline{1-4}\cline{6-7}
		
		\eqref{eq:prob-fed} with \eqref{prob-fed:finite} 
		& Asp \ref{asp:lsmooth-diana}, \ref{asp:avgsmooth-fed}, \ref{asp:pl}
		&  DC-LSVRG
		&  Alg \ref{alg:dc-lsvrg}
		& 
		& Cor \ref{cor:dc-lsvrg-pl} 
		& \bf{New} \\ \cline{1-4}\cline{6-7}
		
		\eqref{eq:prob-fed} with \eqref{prob-fed:finite}
		& Asp \ref{asp:lsmooth-diana}, \ref{asp:avgsmooth-saga-fed}, \ref{asp:pl}
		&  DC-SAGA
		&  Alg \ref{alg:dc-saga}
		& 
		& Cor \ref{cor:dc-saga-pl} 
		& \bf{New} \\
		\hline
		\hline
		\eqref{eq:prob-fed} 
		& Asp \ref{asp:lsmooth-diana}, \ref{asp:pl}
		&  DIANA-GD
		&  Alg \ref{alg:diana-gd}
		& \multirow{4}{*}{Thm \ref{thm:main-pl-dec}, \ref{thm:diana-type-diff}}
		& Cor \ref{cor:diana-gd-pl} 
		& \bf{New} \\ \cline{1-4}\cline{6-7}
		
		\eqref{eq:prob-fed} with \eqref{prob-fed:exp} or \eqref{prob-fed:finite} 
		& Asp \ref{asp:lsmooth-diana}, \ref{asp:pl}
		&  DIANA-SGD
		&  Alg \ref{alg:diana-sgd}
		& 
		& Cor \ref{cor:diana-sgd-pl} 
		& \bf{New} \\ \cline{1-4}\cline{6-7}
		
		\eqref{eq:prob-fed} with \eqref{prob-fed:finite} 
		& Asp \ref{asp:lsmooth-diana}, \ref{asp:avgsmooth-fed}, \ref{asp:pl}
		&  DIANA-LSVRG
		&  Alg \ref{alg:diana-lsvrg}
		& 
		& Cor \ref{cor:diana-lsvrg-pl} 
		& \bf{New} \\ \cline{1-4}\cline{6-7}
		
		\eqref{eq:prob-fed} with \eqref{prob-fed:finite}
		& Asp \ref{asp:lsmooth-diana}, \ref{asp:avgsmooth-saga-fed}, \ref{asp:pl}
		&  DIANA-SAGA
		&  Alg \ref{alg:diana-saga}
		&  
		& Cor \ref{cor:diana-saga-pl} 
		& \bf{New} \\
		\hline
	\end{tabular}
\end{table}

\section{Notation and Assumptions}
\label{sec:pre}

We now introduce the notation and assumptions that we will use throughout the rest of the paper.

\subsection{Notation}
Let $\fgap := f(x^0) - f^*$, where $f^* := \min_{x\in \R^d} f(x)$. 
Let $[n]$ denote the set $\{1,2,\cdots,n\}$ and $\n{\cdot}$ denote the Euclidean norm of a vector.
Let $\inner{u}{v}$ denote the standard Euclidean inner product of two vectors $u$ and $v$.
We use $O(\cdot)$ notation to hide  absolute constants.
For notational convenience, we consider the online form \eqref{prob-fed:exp} or \eqref{prob:exp} as the finite-sum form \eqref{prob-fed:finite} or \eqref{prob:finite} by letting $f_{i,j}(x) := f_i(x, \zeta_j)$ or $f_i(x) := f(x, \zeta_i)$ and thinking of $n$ as infinity (infinite data samples). By $\E[\cdot]$ we denote mathematical expectation.

\subsection{Assumptions}
In order to prove convergence results, one usually needs one or more of the following  standard smoothness assumptions for function $f$, depending on the setting (see e.g., \citep{nesterov2014introductory, ghadimi2016mini, lei2017non, reddi2016stochastic, allen2016variance, li2018simple, fang2018spider, pham2019proxsarah, khaled2020better}).

\begin{assumption}[$L$-smoothness]\label{asp:lsmooth}
	A function $f:\R^d\to \R$ is $L$-smooth if
	\begin{align}\label{eq:lsmooth}
	\n{\nabla f(x) - \nabla f(y)}  \leq L \n{x-y}, \quad \forall x,y \in \R^d.
	\end{align}
\end{assumption}

If one desires to obtain a refined analysis of SGD-type methods applied to  finite-sum problems \eqref{prob:finite}, the $L$-smoothness assumption can be replaced by  average $L$-smoothness, defined next.
\begin{assumption}[Average $L$-smoothness]\label{asp:avgsmooth}
	A function $f(x) :=\frac{1}{n}\sum_{i=1}^{n}f_i(x)$ is average $L$-smooth if 
	\begin{align}\label{eq:avgsmooth}
		\E[\ns{\nabla f_i(x) - \nabla f_i(y)}]\leq \frac{1}{n}\sum \limits_{i=1}^{n} L_i^2\ns{x-y}\leq L^2 \ns{x-y}, \quad \forall x,y \in \R^d.
	\end{align}
\end{assumption}

Note that we slightly change the form of Assumption \ref{asp:avgsmooth} for SAGA-type methods as follows:
\begin{assumption}[Average $L$-smoothness]\label{asp:avgsmooth-saga}
	A function $f(x) :=\frac{1}{n}\sum_{i=1}^{n}f_i(x)$ is average $L$-smooth if 
	\begin{align}\label{eq:avgsmooth-saga}
		\E[\ns{\nabla f_i(x) - \nabla f_i(y_i)}] \leq L^2\frac{1}{n}\sum \limits_{i=1}^{n} \ns{x-y_i}, \quad \forall x, \{y_i\}_{i\in [n]} \in \R^d.
	\end{align}
\end{assumption}

We now present smoothness assumptions suitable for the more general nonconvex federated problems, i.e., \eqref{eq:prob-fed}--\eqref{prob-fed:finite}.

\begin{assumption}[$L$-smoothness]\label{asp:lsmooth-diana}
	For each work $i\in [m]$, the function $f_i(x)$  is $L_i$-smooth if 
	\begin{align}\label{eq:lsmooth-diana}
	\n{\nabla f_i(x) - \nabla f_i(y)}  \leq L_i \n{x-y}, \quad \forall x,y \in \R^d.
	\end{align}
\end{assumption}
Moreover, we define $L^2:=\frac{1}{m}\sum_{i=1}^{m}L_i^2$. Note that Assumption \ref{asp:lsmooth-diana} reduces to Assumption \ref{asp:lsmooth} in single node case (i.e., $m=1$). Similarly, for nonconvex federated finite-sum problems, i.e., \eqref{prob-fed:finite}, we also need the average $L$-smoothness assumption (Assumption \ref{asp:avgsmooth}) for each worker $i$.
\begin{assumption}[Average $\bar{L}$-smoothness]\label{asp:avgsmooth-fed}
	A function $f_i(x) :=\frac{1}{n}\sum_{j=1}^{n}f_{i,j}(x)$  is average $\bar{L}$-smooth if 
	\begin{align}\label{eq:avgsmooth-fed}
		\E[\ns{\nabla f_{i,j}(x) - \nabla f_{i,j}(y)}] \leq \frac{1}{n}\sum \limits_{j=1}^{n} L_{i,j}^2\ns{x-y}\leq \bar{L}^2 \ns{x-y}, \quad \forall x,y \in \R^d.
	\end{align}
\end{assumption}

Similarly, we slightly change the form of Assumption \ref{asp:avgsmooth-fed} for SAGA-type methods as follows:
\begin{assumption}[Average $\bar{L}$-smoothness]\label{asp:avgsmooth-saga-fed}
	A function $f_i(x) :=\frac{1}{n}\sum_{j=1}^{n}f_{i,j}(x)$  is average $\bar{L}$-smooth if 
	\begin{align}\label{eq:avgsmooth-saga-fed}
		\E[\ns{\nabla f_{i,j}(x) - \nabla f_{i,j}(y_{i,j})}] \leq \bar{L}^2\frac{1}{n}\sum \limits_{j=1}^{n} \ns{x-y_{i,j}}, \quad \forall x, \{y_{i,j}\}_{j\in [n]} \in \R^d.
	\end{align}
\end{assumption}

Moreover, we also provide a unified analysis for nonconvex (federated) optimization problems \eqref{eq:prob-fed}--\eqref{prob:finite} under the Polyak-\L{}ojasiewicz (PL) condition \citep{polyak1963gradient}.
\begin{assumption}[PL condition] \label{asp:pl}
	A function $f$ satisfies the PL condition if 
	\begin{equation}\label{eq:pl}
	\exists \mu>0, ~\mathrm{such~that} ~\ns{\nabla f(x)} \geq 2\mu (f(x)-f^*),~ \forall x\in \R^d,
	\end{equation}
	where $f^*:=\min_{x\in \R^d} f(x)$ denotes the optimal function value.
\end{assumption}
It is worth noting that the PL condition does not imply convexity of $f$. For example, $f(x) = x^2 + 3\sin^2 x$ is a nonconvex function but it satisfies PL condition with $\mu=1/32$.
\citet{karimi2016linear} showed that PL condition is weaker than many conditions, e.g., strong convexity (SC), weak strong convexity (WSC), and error bound (EB). 
Moreover, if $f$ is convex, PL condition is equivalent to the error bound (EB) and quadratic growth (QG) condition \citep{luo1993error,anitescu2000degenerate}.

\section{Main Unified Theorems and Simple Special Cases}

In this section we first provide our main unified complexity results (Section~\ref{sec:bui98fg8s_09uf}), and subsequently enumerate a few special cases to showcase the flexibility of our unified approach in accurately describing specific SGD methods (Section~\ref{sec:nonconvex}).

\subsection{Main unified theorems}\label{sec:bui98fg8s_09uf}

We first state Theorem~\ref{thm:main}, which covers a large family of SGD methods (Algorithm~\ref{alg:1}) under the general parametric assumption (Assumption~\ref{asp:boge}). The theorem says that SGD converges at the rate $O(\cdot\frac{1}{\epsilon^2})$ or $O(\cdot\frac{1}{\epsilon^4})$, depending in the value of the parameters.

\begin{theorem}[Main theorem]\label{thm:main}
	Suppose that Assumptions \ref{asp:boge} and \ref{asp:lsmooth} hold. Use the fixed stepsize
	$$  \eta_k \equiv \eta = \min\left\{ \frac{1}{LB_1+LD_1B_2\rho^{-1}},~ \sqrt{\frac{\ln 2}{(LA_1 + LD_1A_2\rho^{-1})K}},~ \frac{\epsilon^2}{2L(C_1+D_1C_2\rho^{-1})} \right\}$$ and let
	$\fgapp:= f(x^0) - f^* + 2^{-1}L\eta^2D_1\rho^{-1} \sigma_0^2$. Then
	the number of iterations performed by Algorithm~\ref{alg:1} to find an $\epsilon$-solution, i.e., a point $\hx$ such that $$\E[\n{\nabla f(\hx)}] \leq \epsilon,$$ can be bounded by
	\begin{align}
	 K = \frac{8\fgapp L}{\epsilon^2} \max \left\{ B_1+D_1B_2\rho^{-1},~ \frac{12\fgapp (A_1+D_1 A_2\rho^{-1})}{\epsilon^2},~ \frac{2(C_1+D_1C_2\rho^{-1})}{\epsilon^2}\right\}. \label{eq:main-ssuhud}
	\end{align}
\end{theorem}

We now state Theorem~\ref{thm:main-pl-dec}, which covers  a large family of SGD methods (Algorithm~\ref{alg:1}) if in addition  the  PL condition (Assumption~\ref{asp:pl}) is satisfied. Note that under the PL condition, one can obtain a faster linear convergence $O(\cdot\log \frac{1}{\epsilon})$ (Theorem~\ref{thm:main-pl-dec}) rather than the sublinear convergence of Theorem~\ref{thm:main}.

\begin{theorem}[Main theorem under PL condition]\label{thm:main-pl-dec}
	Suppose that Assumptions \ref{asp:boge}, \ref{asp:lsmooth} and \ref{asp:pl} hold. Set the stepsize as
	$$ \eta_{k} = \begin{cases}
	\eta   & \text {if~~}  k\leq \frac{K}{2}\\
	\frac{2\eta}{2+(k-\frac{K}{2})\mu\eta} &\text{if~~}  k>\frac{K}{2}
	\end{cases},
	\text{~where~~}
	\eta \leq \frac{1}{LB_1+2LD_1B_2\rho^{-1} + (L A_1 + 2LD_1A_2 \rho^{-1})\mu^{-1}}$$ 
	and let $\fgapp:= f(x^0) - f^* + L\eta^2D_1\rho^{-1} \sigma_0^2$ and $\kappa := \frac{L}{\mu}$. Then the number of iterations performed by Algorithm~\ref{alg:1} to find an $\epsilon$-solution, i.e., a point $x^K$ such that $$\E[f(x^K)-f^*] \leq \epsilon,$$ can be bounded by
	\begin{align} 
	K = \max \left\{2\left(B_1+2D_1B_2\rho^{-1} +(L A_1 + 2LD_1A_2 \rho^{-1})\mu^{-1}\right)\kappa\log \frac{2\fgapp}{\epsilon},~ \frac{10(C_1+2D_1C_2\rho^{-1})\kappa}{\mu \epsilon}\right\}. \label{eq:main-pl-k}
	\end{align}
\end{theorem}

In the following sections and appendices, we show that many specific methods, existing and new, satisfy our unified Assumption \ref{asp:boge} and can thus  be captured by our unified analysis (i.e., Theorems~\ref{thm:main} and \ref{thm:main-pl-dec}). We can thus plug their corresponding parameters (i.e., specific values for $A_1, A_2, B_1, B_2, C_1,C_2,D_1,\rho$) into our unified Theorems~\ref{thm:main}  and \ref{thm:main-pl-dec} to obtain  detailed convergence rates for these methods. See Tables~\ref{table:1} and \ref{table:pl} for an overview.

In particular, we give two general frameworks, i.e., DC framework (Algorithm \ref{alg:dc}) and DIANA framework (Algorithm \ref{alg:diana}) for the general nonconvex federated problems \eqref{eq:prob-fed}--\eqref{prob-fed:finite}.
We provide Theorems \ref{thm:dc-diff} and \ref{thm:diana-type-diff}  showing that optimization methods belonging to the DC and DIANA frameworks also satisfy our unified Assumption~\ref{asp:boge}, and can thus also be captured by our unified analysis to obtain detailed convergence rates in the nonconvex distributed regime. See Tables~\ref{table:fed} and \ref{table:fed-pl} for an overview.

\subsection{Simple special cases ($m=1$)}
\label{sec:nonconvex}

As a case study, we first focus on the single node case (i.e., $m=1$), i.e., on the standard problem \eqref{eq:prob} with online form \eqref{prob:exp} or finite-sum form \eqref{prob:finite}, i.e., 
\[ \min_{x\in \R^d}   f(x),\]
where
\[ f(x) := \E_{\zeta\sim \cD}[f(x,\zeta)],  
\text{~~~~~or~~~~~}  f(x) := \frac{1}{n}\sum \limits_{i=1}^n{f_i(x)}.\]

In the following, we prove that some classical methods such as GD (Section~\ref{sec:bi8f9g8f}), SGD (Section~\ref{sec:98g9s8gf89gs}), L-SVRG (Section~\ref{sec:bo89gfs09f}) and SAGA (Section~\ref{sec:bi98gjhf_9u9f}) satisfy our unified Assumption \ref{asp:boge}, and thus can be captured by our unified convergence analysis, i.e., Theorems \ref{thm:main} and \ref{thm:main-pl-dec}. 

\subsubsection{GD method}\label{sec:bi8f9g8f}

The vanilla gradient descent (GD) method is formalized as Algorithm~\ref{alg:gd}.

\begin{algorithm}[h]
	\caption{GD}
	\label{alg:gd}
	\begin{algorithmic}[1]
		\STATE  In Line \ref{line:grad} of Algorithm \ref{alg:1}: $g^k=\nabla f(x^k)$ \label{line:grad-gd}
	\end{algorithmic}
\end{algorithm}

We now show that the gradient estimator used in  GD, i.e., the true/full gradient, satisfies  Assumption~\ref{asp:boge}.
\begin{lemma}[GD estimator satisfies Assumption \ref{asp:boge}]\label{lem:gd}
	The gradient estimator $g^k=\nabla f(x^k)$  satisfies the unified  Assumption \ref{asp:boge} with parameters
	$$A_1=C_1=D_1=0,\quad B_1=1,\quad \sigma_k^2 \equiv 0,\quad \rho=1,\quad A_2=B_2=C_2=0.$$
\end{lemma}

\subsubsection{SGD method}
\label{sec:98g9s8gf89gs}

Many (but not all) variants of stochastic gradient descent (SGD) can be written in the form of Algorithm~\ref{alg:sgd}.

\begin{algorithm}[h]
	\caption{SGD \citep{khaled2020better}}
	\label{alg:sgd}
	\begin{algorithmic}[1]
		\STATE  In Line \ref{line:grad} of Algorithm \ref{alg:1}: 
		\STATE $g^k=\begin{cases}
		\nabla f_i(x^k), &\text{Standard SGD} \\
		\cC(\nabla f(x^k)), & \text{Compressed gradient, e.g., quantized gradient, sparse gradient, etc} \\
		\nabla f(x^k) +\xi,  &  \text{Noisy gradient} \\
		\sum_{i\in I} v_i \nabla f_i(x^k), &\text{minibatch SGD, SGD with importance or arbitrary sampling}\\
		\ldots,  &\text{Combinations (e.g., minibatch compressed SGD, noisy SGD) and beyond}
		\end{cases}
		$\label{line:grad-sgd}
	\end{algorithmic}
\end{algorithm}

\citet{khaled2020better}  showed that the stochastic gradient estimators used in many variants of SGD for nonconvex smooth problems, including variants performing minibatching, importance sampling, gradient compression and their combinations, satisfy the following generalized expected smoothness (ES) assumption.

\begin{assumption}[ES: Expected Smoothness  \citep{khaled2020better}]\label{asp:es}
	The gradient estimator $g(x)$ is unbiased
	$\E[g(x)] = \nabla f(x)$
	and there exist non-negative constants $A, B, C$ such that 
	\begin{align}
	\E[\ns{g(x)}] \leq 2A(f(x)-f^*)+B\ns{\nabla f(x)} + C,  \quad   \forall x\in \R^d.
	\end{align}        
\end{assumption}	

Our work can be seen as a further and substantial generalization of their approach, one that allows for many more methods  to be captured by a single assumption and analysis.  In particular, our unified Assumption \ref{asp:boge}  can capture the behavior of the gradient estimator constructed by variance reduced methods while ES assumption cannot. The next lemma says that, indeed, our unified Assumption \ref{asp:boge} recovers the ES assumption in a special case.

\begin{lemma}[SGD estimator satisfies Assumption \ref{asp:boge}]\label{lem:sgd}
	Any gradient estimator $g^k$  satisfying the Expected Smoothness Assumption~\ref{asp:es}, i.e.,
	$$ \E_k[\ns{g^k}] \leq 2A(f(x^k)-f^*)+B\ns{\nabla f(x^k)} + C,$$ 
	(used in Line \ref{line:grad-sgd} of Algorithm \ref{alg:sgd})
	 satisfies the unified Assumption \ref{asp:boge} with  parameters 
	$$A_1=A,\quad B_1=B,\quad C_1= C,\quad D_1=0,\quad \sigma_k^2 \equiv 0,\quad \rho=1,\quad A_2=B_2=C_2=0.$$
\end{lemma}

\subsubsection{L-SVRG method}\label{sec:bo89gfs09f}

The  loopless (minibatch) SVRG method (L-SVRG) developed by  \citet{Hofmann2015}  and rediscovered by \citet{kovalev2019don}  is formalized as Algorithm~\ref{alg:lsvrg}. 
\citet{kovalev2019don} only studied L-SVRG  in the strongly convex setting with $b=1$ (no minibatch).  Recently, \citet{qian2019svrg} studied L-SVRG in the nonconvex case.
\vspace{-2mm}
\begin{algorithm}[!h]
	\caption{Loopless SVRG (L-SVRG) \citep{Hofmann2015, kovalev2019don} }
	\label{alg:lsvrg}
	\begin{algorithmic}[1]
		\REQUIRE ~
		initial point $x^0=w^0$, stepsize $\eta_k$, minibatch size $b$, probability $p\in (0,1]$
		\FOR {$k=0,1,2,\ldots$}
		\STATE $g^k = \frac{1}{b} \sum \limits_{i\in I_b} (\nabla f_i(x^k)- \nabla f_i(w^k)) +\nabla f(w^k)$ \qquad ($I_b$ denotes random minibatch with $|I_b|=b$)
		\label{line:lsvrg}
		\STATE $x^{k+1} = x^k - \eta_k g^k$  \label{line:update-lsvrg}
		\STATE $w^{k+1} = \begin{cases}
		x^k, &\text{with probability } p\\
		w^k, &\text{with probability } 1-p
		\end{cases}$ \label{line:w_prob}
		\ENDFOR
	\end{algorithmic}
\end{algorithm}

We now show that the gradient estimator used in  L-SVRG satisfies  Assumption~\ref{asp:boge}.
\begin{lemma}[L-SVRG estimator satisfies Assumption \ref{asp:boge}]\label{lem:lsvrg}
	Suppose that Assumption \ref{asp:avgsmooth} holds. The L-SVRG gradient estimator $$g^k=\frac{1}{b} \sum_{i\in I_b} (\nabla f_i(x^k)- \nabla f_i(w^k)) +\nabla f(w^k)$$ (see Line \ref{line:lsvrg} in Algorithm \ref{alg:lsvrg})
	satisfies the unified Assumption~\ref{asp:boge} with parameters
	$$A_1=A_2=C_1=C_2=0,$$
	$$B_1=1, \quad D_1=\frac{L^2}{b}, \quad \sigma_k^2= \ns{x^k-w^k}, \quad  \rho=\frac{p}{2}+\frac{p^2}{2}-\frac{\eta^2L^2}{b},\quad  B_2=\frac{2\eta^2}{p}-\eta^2.$$
\end{lemma}

As we discussed in Section \ref{sec:contribution}, for variance-reduced methods, the gradient estimator usually satisfies the unified Assumption \ref{asp:boge} with $D_1>0$ and $C_1=0$.
The recurrence \eqref{eq:boge2} describes the variance reduction process, with parameter $\rho$ describing the speed of variance reduction. 
Here $C_2$ is also $0$,  thus L-SVRG is fully variance reduced, which typically means faster convergence rate. This can be see from our main Theorems \ref{thm:main} and \ref{thm:main-pl-dec}. Indeed, i) for Theorem \ref{thm:main}, note that because $A_1=A_2=C_1=C_2=0$, the two $O(\frac{1}{\epsilon^2})$ terms appearing in the max in \eqref{eq:main-ssuhud} disappear, and hence the  rate of L-SVRG is $O(\frac{1}{\epsilon^2})$ as opposed to the generic rate $O(\frac{1}{\epsilon^4})$ of less refined (i.e., not variance reduced) SGD variants.
ii) for Theorem \ref{thm:main-pl-dec}, note that because $C_1=C_2=0$, the last term $O(\frac{\kappa}{\mu \epsilon})$ in \eqref{eq:main-pl-k} disappears, and hence the rate of L-SVRG is linear $O(\kappa\log \frac{1}{\epsilon})$ as opposed to the worse sublinear rate $O(\frac{\kappa}{\mu \epsilon})$ of less refined (i.e., not variance reduced) SGD variants.

\subsubsection{SAGA method}\label{sec:bi98gjhf_9u9f}

The  (minibatch) SAGA method developed by  \citet{defazio2014saga}  (in convex setting) is formalized as Algorithm~\ref{alg:saga}.
Here we analyze it and obtain convergence rates in nonconvex setting.

\begin{algorithm}[!h]
	\caption{SAGA \cite{defazio2014saga}}
	\label{alg:saga}
	\begin{algorithmic}[1]
		\REQUIRE ~
		initial point $x^0, \{w_i^0\}_{i=1}^n$, stepsize $\eta_k$, minibatch size $b$
		\FOR {$k=0,1,2,\ldots$}
		\STATE $g^k = \frac{1}{b} \sum \limits_{i\in I_b} (\nabla f_i(x^k)- \nabla f_i(w_i^k)) +\frac{1}{n}\sum \limits_{j=1}^{n}\nabla f_j(w_j^k)$ ~~~($I_b$ denotes random minibatch with $|I_b|=b$) \label{line:saga}
		\STATE $x^{k+1} = x^k - \eta_k g^k$  \label{line:update-saga}
		\STATE $w_i^{k+1} = \begin{cases}
		x^k, & \text{for~~}  i\in I_b\\
		w_i^k, &\text{for~~} i\notin I_b
		\end{cases}$ \label{line:wi}
		\ENDFOR
	\end{algorithmic}
\end{algorithm}

We now show that the gradient estimator used in SAGA satisfies  Assumption~\ref{asp:boge}.

\begin{lemma}[SAGA estimator satisfies Assumption \ref{asp:boge}]\label{lem:saga}
	Suppose that Assumption \ref{asp:avgsmooth-saga} holds. The SAGA gradient estimator $$g^k= \frac{1}{b} \sum_{i\in I_b} (\nabla f_i(x^k)- \nabla f_i(w_i^k)) +\frac{1}{n}\sum_{j=1}^{n}\nabla f_j(w_j^k)$$ (see Line \ref{line:saga} in Algorithm \ref{alg:saga})
	satisfies the unified Assumption~\ref{asp:boge} with parameters $$A_1=A_2=C_1=C_2=0,$$
	$$B_1=1,\quad D_1=\frac{L^2}{b},\quad \sigma_k^2=\frac{1}{n}\sum_{i=1}^{n} \ns{x^k-w_i^k},\quad  \rho=\frac{b}{2n}+\frac{b^2}{2n^2}-\frac{\eta^2L^2}{b},\quad B_2=\frac{2\eta^2n}{b}-\eta^2.$$
\end{lemma}
{\bf Remark:}  We obtain the detailed convergence rates for these methods by plugging their corresponding specific values of $A_1, A_2, B_1, B_2, C_1,C_2,D_1,\rho$ (see Lemmas \ref{lem:gd}--\ref{lem:saga}) into our unified Theorems~\ref{thm:main}  and \ref{thm:main-pl-dec}.  See Tables~\ref{table:1} and \ref{table:pl} for an overview.

\section{General Nonconvex Federated Optimization Problems}
\label{sec:fedrated}

In this section, we consider the more general nonconvex distributed/federated problem \eqref{eq:prob-fed} with online form \eqref{prob-fed:exp} or finite-sum form \eqref{prob-fed:finite}, i.e.,
\[
  \min_{x\in \R^d} \bigg\{  f(x) := \frac{1}{m}\sum \limits_{i=1}^m{f_i(x)}  \bigg\}, 
  \]
  where
\[ f_i(x) := \E_{\zeta \sim \cD_i}[f_i(x,\zeta)],  
\text{~~~~or~~~~}  f_i(x) := \frac{1}{n}\sum \limits_{j=1}^n{f_{i,j}(x)}.
\]
Here we allow different machines/workers to have different data distributions, i.e., we consider the non-IID (heterogeneous) data setting. Note that in distributed/federated problems, the bottleneck usually is the communication cost among workers, which motivates the study of methods which employ {\em compressed} communication. 

In the following, we provide two general algorithmic frameworks differing in whether direct gradient compression (DC) or compression of gradient differences (DIANA) is used. Previous approaches mostly focus on strongly convex or convex problems for specific instances of SGD. Ours is the first unified analysis covering many variants of SGD in a single theorem, covering the nonconvex regime. In fact, many specific  SGD methods arising as special cases of our general approach have not been analyzed before.

\subsection{Compression operators}

Compressed communication is modeled by the application of a (randomized) compression operator to the communicated messages, as described next.   

\begin{definition}[Compression operator] \label{def_compression} 
	A randomized map $\calC: \R^d\mapsto \R^d$ is an $\omega$-compression operator  if  
	\begin{equation}\label{eq:compress-main}
	\E[\calC(x)]=x,  \qquad \E[\ns{\calC(x)-x}] \leq \omega\ns{x}, \qquad \forall x\in \R^d.
	\end{equation}
	In particular, no compression ($\calC(x)\equiv x$) implies $\omega=0$.
\end{definition}
Note that \eqref{eq:compress-main} holds for many practical compression methods, e.g., random sparsification \citep{stich2018sparsified}, quantization \citep{alistarh2017qsgd}, natural compression \citep{Cnat}. We are not going to focus on any specific compression operator; instead, we will analyze our methods for any compression operator captured by the above definition.

\subsection{DC framework for nonconvex federated optimization}

In the direct compression (DC) framework studied in this section, each machine $i\in [m]$ computes its local stochastic gradient $\widetilde{g}_i^k$,  subsequently applies to it a compression operator $\cC_i^k$, and communicates the compressed vector to a server, or to all other nodes (see Algorithm \ref{alg:dc}).

\vspace{-1mm}
\begin{algorithm}[!h]
	\caption{DC framework of stochastic gradient methods for nonconvex federated optimization}
	\label{alg:dc}
	\begin{algorithmic}[1]
		\REQUIRE ~
		initial point $x^0\in \R^d$,  stepsizes $\eta_k$
		\FOR {$k=0,1,2,\ldots$}
		\STATE {\bf{for all machines $i= 1,2,\ldots,m$ do in parallel}}
		\STATE \quad Compute local stochastic gradient $\widetilde{g}_i^k$ \label{line:localgrad-dc}
		\STATE \quad \blue{Compress local gradient $\cC_i^k(\widetilde{g}_i^k)$} and send it to the server \label{line:comgrad-dc}
		\STATE {\bf{end for}}
		\STATE Aggregate received compressed gradient information:
		$g^k = \frac{1}{m}\sum \limits_{i=1}^m \cC_i^k(\widetilde{g}_i^k)$ \label{line:gk-dc}
		\STATE $x^{k+1} = x^k - \eta_k g^k$  
		\ENDFOR
	\end{algorithmic}
\end{algorithm}

Our main theoretical result describing the convergence properties of Algorithm~\ref{alg:dc} is stated next.

\begin{theorem}[DC framework]\label{thm:dc-diff}
	If the local stochastic gradient $\widetilde{g}_i^k$ (see Line \ref{line:localgrad-dc} of Algorithm~\ref{alg:dc}) satisfies  the recursions
	\begin{align}
	\E_k[\ns{\widetilde{g}_i^k}] & \leq 2A_{1,i}(f_i(x^k)-f_i^*)+B_{1,i}\ns{\nabla f_i(x^k)} + {D_{1,i} \sigma_{k,i}^2} +C_{1,i}, \label{eq:gi1-dc-diff-8}\\
	\E_k[\sigma_{k+1,i}^2] & \leq {(1-\rho_i)\sigma_{k,i}^2 + 2A_{2,i}(f(x^k)-f^*)+B_{2,i}\ns{\nabla f(x^k)} + D_{2,i}\E_k[\ns{g^k}] +C_{2,i}}, \label{eq:gi2-dc-diff-8}
	\end{align}
	then $g^k$ (see Line \ref{line:gk-dc} of Algorithm \ref{alg:dc}) satisfies the unified Assumption \ref{asp:boge}
	with
	\begin{align*}
	&  A_1 =\frac{(1+\omega)A}{m}, \quad
	B_1 =1, \quad 
	D_1 =\frac{1+\omega}{m}, \quad
	\sk =  \frac{1}{m}\summ  D_{1,i} \ski , \quad
	C_1  =  \frac{(1+\omega)C}{m}, \\
	&  \rho =\min_i\rho_i-\tau, \quad
	A_2 =D_A+ \tau A, \quad
	B_2 =D_B+D_D, \quad 
	C_2  = D_C+  \tau C, 
	\end{align*}
	where 
	$A:=\max_i  (A_{1,i}+B_{1,i}L_i-L_i/(1+\omega))$,
	$C := \frac{1}{m} \summ C_{1,i} + 2A\Delta_f^*$,
	$\Delta_f^*:=f^*-\frac{1}{m}\summ f_i^*$,
	$\tau:=\frac{(1+\omega)D_D}{m}$,
	$D_A:=\frac{1}{m} \summ  D_{1,i}A_{2,i}$,
	$D_B:=\frac{1}{m} \summ  D_{1,i} B_{2,i}$,
	$D_D:=\frac{1}{m} \summ  D_{1,i} D_{2,i}$,
	and 
	$D_C:=\frac{1}{m} \summ  D_{1,i}C_{2,i}$.
\end{theorem}

The above result means that, provided that  the local gradient estimators $\widetilde{g}_i^k$ used in the DC framework (Algorithm~\ref{alg:dc}) satisfy recursions \eqref{eq:gi1-dc-diff-8}--\eqref{eq:gi2-dc-diff-8}, the global gradient estimator $g^k$ satisfies our unified Assumption~\ref{asp:boge}, and hence our main convergence results, Theorem~\ref{thm:main} and Theorem~\ref{thm:main-pl-dec} can be applied. 

To showcase the generality and expressive power of our DC framework, we describe several particular ways in which such local estimators can be generated, each leading to a particular instance DC-type method. 
In particular, we  describe methods DC-GD (Algorithm~\ref{alg:dc-gd}), DC-SGD (Algorithm \ref{alg:dc-sgd}), DC-LSVRG (Algorithm \ref{alg:dc-lsvrg}) and DC-SAGA (Algorithm \ref{alg:dc-saga}). 
In each case we show that recursions \eqref{eq:gi1-dc-diff-8} and  \eqref{eq:gi2-dc-diff-8} are satisfied, and in doing so, we obtain complexity results in the nonconvex case with or without the PL condition. 
Details can be found in the appendix.
See the first part of Table~\ref{table:fed} and the first part of Table~\ref{table:fed-pl} for a summary of these particular methods. 
Note that all of these results are new, with the exception of DC-GD and DC-SGD in the nonconvex non-PL case.

\subsection{DIANA framework for nonconvex federated optimization}
\label{sec:diana}

We now highlight an inherent issue of the DC framework (Algorithm \ref{alg:dc}), which will serve as a motivation for the proposed DIANA framework described here. Considering any stationary point $\hx$ such that $\nabla f(\hx) =\sum_{i=1}^{m} \nabla f_i(\hx) =0$,  the aggregated compressed gradient (even if the full gradient is used locally, i.e., $\widetilde{g}_i^k=\nabla f_i(x^k)$), is {\em not} equal to zero $0$, i.e., $g(\hx) = \frac{1}{m}\sum_{i=1}^m \cC_i(\nabla f_i(\hx)) \neq 0$. This effect slows down convergence of the methods in DC framework.
To address this issue, we use the DIANA framework 
to compress the {\em gradient differences} instead (see Line \ref{line:comgrad-diana} of Algorithm \ref{alg:diana}). 

\begin{algorithm}[h]
	\caption{DIANA framework of stochastic gradient methods for nonconvex federated optimization}
	\label{alg:diana}
	\begin{algorithmic}[1]
		\REQUIRE ~
		initial point $x^0$, $\{h_i^0\}_{i=1}^m$, $h^0=\frac{1}{m}\sum_{i=1}^{m}h_i^0$, stepsize parameters $\eta_k, \alpha_k$
		\FOR {$k=0,1,2,\ldots$}
		\STATE {\bf{for all machines $i= 1,2,\ldots,m$ do in parallel}}
		\STATE \quad Compute local stochastic gradient $\widetilde{g}_i^k$ \label{line:localgrad}
		\STATE \quad \blue{Compress shifted local gradient $\widehat{\Delta}_i^k= \cC_i^k(\widetilde{g}_i^k- h_i^k)$} and send $\widehat{\Delta}_i^k$ to the server \label{line:comgrad-diana}
		\STATE \quad Update local shift $h_i^{k+1}=h_i^k+\alpha_k \cC_i^k(\widetilde{g}_i^k - h_i^k)$
		\STATE {\bf{end for}}
		\STATE Aggregate received compressed gradient information:
		$g^k = h^k + \frac{1}{m}\sum \limits_{i=1}^m \widehat{\Delta}_i^k$ \label{line:gk}
		\STATE $x^{k+1} = x^k - \eta_k g^k$ 
		\STATE $h^{k+1} = h^k + \alpha_k  \frac{1}{m}\sum \limits_{i=1}^m  \widehat{\Delta}_i^k$
		\ENDFOR
	\end{algorithmic}
\end{algorithm}

The DIANA framework was previously studied by \citet{mishchenko2019distributed, DIANA2,DFPMCI2019} for a few specific instances of SGD and mainly in strongly convex or convex problems. Ours is the first unified analysis covering many variants of SGD in a single theorem, covering the nonconvex regime.

Concretely, \citet{DFPMCI2019} analyze the strongly convex or convex problem of finding a fixed point with strong assumptions. Here we analyze the more general nonconvex setting with a more general unified assumption.
\citet{mishchenko2019distributed} analyze their methods for the ternary compression operator only, and their result in the nonconvex setting makes very strong assumptions on the local gradient estimators $\widetilde{g}_i^k$. In particular, they assume that $\widetilde{g}_i^k$ is unbiased, and that there exists $\sigma_i^2\geq 0$ such that $\E_k[\ns{\widetilde{g}_i^k}] \leq \ns{\nabla f_i(x^k)} +\sigma_i^2$. 
Note that  this  implies that $B_{1,i} = 1$ and $C_{1,i} = \sigma_i^2$ in recursion \eqref{eq:gi1-dc-diff-8}.
This corresponds to a (rather restrictive) special case of recursions \eqref{eq:gi1-dc-diff-8}--\eqref{eq:gi2-dc-diff-8} with the rest of the parameters rendered ``inactive'':
$A_{1,i} = D_{1,i}=A_{2,i} = B_{2,i} = C_{2,i} = D_{2,i} =0$, and $\rho_i=1$. 
Hence, when compared to \citep{mishchenko2019distributed}, our results for the DIANA framework can be seen as a generalization to arbitrary compression operators described by Definition~\ref{def_compression}, to a much more general class of local gradient estimators (including more methods), and to the PL setting. 

Further, \citet{DIANA2} lift some of the deficiencies of \citep{mishchenko2019distributed}. In particular, they consider general compression operators, and consider the finite-sum setting 
in  which they use two particular variance reduced estimators $\widetilde{g}_i^k$, i.e., L-SVRG and SAGA with minibatch size $b=1$ (non-minibatch version).
Our work can be seen as a generalization of their work to the potentially infinite family of local gradient estimators described by recursions \eqref{eq:gi1-dc-diff-8}--\eqref{eq:gi2-dc-diff-8}, and further to the PL setting. Note that our framework allows for the analysis of more general variants of DIANA, such as DIANA-LSVRG or DIANA-SAGA with additional additive noise and with minibatch $b\geq 1$ (note that one can enjoy a linear speedup by adopting minibatch in parallel), which can be helpful in some situations.

Our main result for DIANA framework is described in the following Theorem \ref{thm:diana-type-diff}.
\begin{theorem}[DIANA framework]\label{thm:diana-type-diff}
	Suppose that the local stochastic gradient $\widetilde{g}_i^k$ (see Line~\ref{line:localgrad} of Algorithm~\ref{alg:diana}) satisfies \eqref{eq:gi1-dc-diff-8}--\eqref{eq:gi2-dc-diff-8}, same as in the DC framework.
	Then $g^k$ (see Line \ref{line:gk} of Algorithm \ref{alg:diana}) satisfies the unified Assumption \ref{asp:boge}
	with 
	\begin{align*}
	&  A_1 =\frac{(1+\omega)A}{m}, ~
	B_1 =1, ~
	D_1 =\frac{1+\omega}{m}, ~
	\sk =  \frac{1}{m} \summ  D_{1,i} \ski  + \frac{\omega}{(1+\omega)m}
	\summ \ns{\nabla f_i(\xk) -\hi}, \\
	&  C_1  =  \frac{(1+\omega)C}{m}, \quad
	\rho =\min \left\{\min_i\rho_i-\tau,~ 2\alpha -(1-\alpha)\beta^{-1} -\alpha^2-\tau \right\}, \\
	&  A_2 =D_A+ \tau A, \quad
	B_2 =D_B+B, \quad 
	C_2  = D_C +  \tau C,
	\end{align*}
	where 
	$A:=\max_i  (A_{1,i}+(B_{1,i}-1)L_i)$, 
	$B:=\frac{\omega(1+\beta)L^2\eta^2}{1+\omega} + D_D$, 
	$C := \frac{1}{m} \summ C_{1,i} + 2A\Delta_f^*$,
	$\Delta_f^*:=f^*-\frac{1}{m}\summ f_i^*$,
	$\tau:=\alpha^2 \omega + \frac{(1+\omega)B}{m}$,
	$D_A:=\frac{1}{m} \summ  D_{1,i}A_{2,i}$,
	$D_B:=\frac{1}{m} \summ  D_{1,i} B_{2,i}$,
	$D_D:=\frac{1}{m} \summ  D_{1,i} D_{2,i}$,
	$D_C:=\frac{1}{m} \summ  D_{1,i}C_{2,i}$,
	and $\forall \beta>0$.
\end{theorem}

Similarly, the above result means that, provided that  the local gradient estimators $\widetilde{g}_i^k$ used in the DIANA framework (Algorithm~\ref{alg:diana}) satisfy recursions \eqref{eq:gi1-dc-diff-8}--\eqref{eq:gi2-dc-diff-8}, the global gradient estimator $g^k$ satisfies our unified Assumption~\ref{asp:boge}, and hence our main convergence results, Theorem~\ref{thm:main} and Theorem~\ref{thm:main-pl-dec} can be applied. 

To showcase the generality and expressive power of our DIANA framework, we describe several particular ways in which such local estimators can be generated, each leading to a particular instance of a DIANA-type method. 
In particular, we  describe methods DIANA-GD (Algorithm~\ref{alg:diana-gd}), DIANA-SGD (Algorithm~\ref{alg:diana-sgd}), DIANA-LSVRG (Algorithm~\ref{alg:diana-lsvrg}) and DIANA-SAGA (Algorithm~\ref{alg:diana-saga}). 
In each case we show that recursions \eqref{eq:gi1-dc-diff-8} and  \eqref{eq:gi2-dc-diff-8} are satisfied, and in doing so, we obtain complexity results in the nonconvex case with or without the PL condition. 
Details can be found in the appendix.
See the second part of Table~\ref{table:fed} and the second part of Table~\ref{table:fed-pl} for a summary of these particular methods.
Note that all of these results are new, with the exception of a weak version of DIANA-LSVRG and DIANA-SAGA studied by \citet{DIANA2} (as discussed above) in the nonconvex non-PL case.

\newpage

\bibliographystyle{plainnat}
\bibliography{bibunified}

\newpage
\appendix
	

\section{Missing Proofs for Unified Main Theorems}
\label{sec:proofofmain}

In this section, we provide the detailed proofs for our unified main theorems with/without PL condition.

\subsection{Proof of unified Theorem \ref{thm:main}}
We first restate our unified Theorem \ref{thm:main} here and then provide the detailed proof.

\begingroup
\def\thetheorem{\ref{thm:main}}
\begin{theorem}[Main theorem]
	Suppose that Assumptions \ref{asp:boge} and \ref{asp:lsmooth} hold. Let stepsize 
	$$\eta_k \equiv \eta \leq \min\left\{ \frac{1}{LB_1+LD_1B_2\rho^{-1}},~ \sqrt{\frac{\ln 2}{(LA_1 + LD_1A_2\rho^{-1})K}},~ \frac{\epsilon^2}{2L(C_1+D_1C_2\rho^{-1})} \right\},$$ then the number of iterations performed by Algorithm \ref{alg:1} to find an $\epsilon$-solution, i.e. a point $\hx$ such that $\E[\n{\nabla f(\hx)}] \leq \epsilon$, can be bounded by
	\begin{align} 
	K = \frac{8\fgapp L}{\epsilon^2} \max \left\{ B_1+D_1B_2\rho^{-1},~ \frac{12\fgapp (A_1+D_1 A_2\rho^{-1})}{\epsilon^2},~ \frac{2(C_1+D_1C_2\rho^{-1})}{\epsilon^2}\right\}, \label{eq:k-main}
	\end{align}
	where $\fgapp:= f(x^0) - f^* + 2^{-1}L\eta^2D_1\rho^{-1} \sigma_0^2$.
\end{theorem}
\addtocounter{theorem}{-1}
\endgroup

\begin{proof}
	First, we obtain the relation between $f(\xkn)$ and $f(\xk)$:
	\begin{align}
	f(\xkn) 
	&\leq f(\xk) + \inner{\nabla f(\xk)}{\xkn-\xk} + \frac{L}{2}\ns{\xkn-\xk} \label{eq:use_smooth}\\
	&=f(\xk) -\eta  \inner{\nabla f(\xk)}{\gk} + \frac{L\eta^2}{2}\ns{\gk}, \label{eq:plug_update}
	\end{align}
	where \eqref{eq:use_smooth} uses $L$-smoothness of $f$ (see \eqref{eq:lsmooth}), and \eqref{eq:plug_update} follows from the update step $\xkn = \xk - \eta \gk$ (Line \ref{line:update} of Algorithm \ref{alg:1}). Now, we take expectation for \eqref{eq:plug_update} conditional on the past, i.e., $x^{0:k}$, denoted as $\Ek$:
	\begin{align}
	\Ek[f(\xkn)] 
	& \leq f(\xk) -\eta  \ns{\nabla f(\xk)} + \frac{L\eta^2}{2}\Ek[\ns{\gk}] \notag \\
	\Ek[f(\xkn) - \fs] 
	& \leq f(\xk) -\fs -\eta  \ns{\nabla f(\xk)} + \frac{L\eta^2}{2}\Ek[\ns{\gk}] \notag \\
	& \leq (1+L\eta^2 A_1) (f(\xk) -\fs) - \left(\eta-\frac{L\eta^2 B_1}{2}\right)  \ns{\nabla f(\xk)} + \frac{L\eta^2 D_1}{2} \sk + \frac{L\eta^2}{2}C_1 \notag
	\end{align}
	where these inequalities hold by using our unified Assumption \ref{asp:boge}. Then according to \eqref{eq:boge2} in Assumption \ref{asp:boge}, we have, for $\forall \alpha>0$
	\begin{align}
	\Ek[f(\xkn) - \fs + \alpha\skn] 
	& \leq (1+L\eta^2 A_1 +2\alpha A_2) (f(\xk) -\fs) 
	+ \left(\frac{L\eta^2 D_1}{2} + \alpha (1-\rho)\right) \sk
	\notag \\
	& \qquad \qquad  - \left(\eta-\frac{L\eta^2 B_1}{2}-\alpha B_2\right)  \ns{\nabla f(\xk)} + \frac{L\eta^2}{2}C_1 + \alpha C_2. \notag
	\end{align}
	Now, we take expectation again and define 
	\begin{align}
	\Dkn &:=f(\xkn) - \fs + \alpha\skn  \label{eq:d_def}\\
	\beta &:= 1+L\eta^2 A_1 +2\alpha A_2   \label{eq:beta_def}\\
	\eta' &: = \eta-\frac{L\eta^2 B_1}{2}-\alpha B_2  \label{eq:eta_def}\\
	C &: = \frac{L\eta^2}{2}C_1 + \alpha C_2   \label{eq:c_def}
	\end{align}
	to obtain
	\begin{align}
	\E[\Dkn] 
	& \leq \beta \E[f(\xk) -\fs] 
	+ \left(\frac{L\eta^2 D_1}{2} + \alpha (1-\rho)\right) \E[\sk]
	- \eta' \E[\ns{\nabla f(\xk)}] + C \notag \\
	& = \beta \E\left[f(\xk) -\fs +
	\left(\frac{\frac{L\eta^2 D_1}{2} + \alpha (1-\rho)}{\beta}\right) \sk\right]
	- \eta' \E[\ns{\nabla f(\xk)}] + C \notag \\
	& \leq  \beta\E[\Dk]  -  \eta' \E[\ns{\nabla f(\xk)}]  + C, \label{eq:alphabeta} 
	\end{align}
	where the last inequality holds by setting 
	\begin{equation}\label{eq:alpha_def}
	\alpha =\frac{L\eta^2D_1}{2\rho},
	\end{equation}
	the reason is as follows: 
	\begin{align}
	\frac{L\eta^2 D_1}{2} + \alpha (1-\rho) 
	&\leq \alpha \beta \overset{\eqref{eq:beta_def}}{=} \alpha(1+L\eta^2 A_1 +2\alpha A_2),
	\notag \\
	\frac{L\eta^2 D_1}{2}  &\leq \alpha(\rho+L\eta^2 A_1 +2\alpha A_2). \notag
	\end{align}
	Now we rewrite \eqref{eq:alphabeta} as 
	\begin{align}
	\beta^{K-1-k} \eta' \E[\ns{\nabla f(\xk)}]
	&\leq  \beta^{K-1-k}  (\beta \E[\Dk] - \E[\Dkn]) + \beta^{K-1-k}  C \notag\\
	&=\beta^{K-k}  \E[\Dk] - \beta^{K-(k+1)} \E[\Dkn] + \beta^{K-1-k}  C, 
	~~ \forall 0\leq k \leq K-1.  \label{eq:key}
	\end{align}
	Summing up \eqref{eq:key} for $0\leq k \leq K-1$ to get
	\begin{align}
	\sum_{k=0}^{K-1}\beta^{K-1-k} \eta' \E[\ns{\nabla f(\xk)}]
	& \leq \beta^{K}  \E[\Delta^0] 
	+\sum_{k=0}^{K-1} \beta^{K-1-k}  C   \notag\\
	\eta' \E[\ns{\nabla f(\hx)}] 
	&\leq	\frac{\beta^K}{\sum_{k=0}^{K-1}\beta^{K-1-k}}	\Delta^0 + C, \label{eq:result}
	\end{align}
	where $\hx$ randomly chosen from $\{\xk\}_{k=0}^{K-1}$ with probability $p_k=\frac{\beta^{K-1-k}}{\sum_{k=0}^{K-1}\beta^{K-1-k}}$ for $x^k$.
	Now, we compute each parameter in \eqref{eq:result}:
	\begin{align}
	\eta' &\overset{\eqref{eq:eta_def}}{=}  \eta-\frac{L\eta^2 B_1}{2}-\alpha B_2 
	\notag\\
	&\overset{\eqref{eq:alpha_def}}{=} \eta-\frac{L\eta^2 B_1}{2}-\frac{L\eta^2D_1B_2}{2\rho}  \notag\\
	&~~ = \eta- \frac{\eta}{2}(L\eta B_1+ L\eta D_1B_2\rho^{-1}) \notag\\
	&~~ \geq \frac{\eta}{2}, \label{eq:para1}
	\end{align}
	where the last inequality holds by setting 
	\begin{equation}\label{eq:eta1}
	\eta \leq \frac{1}{LB_1+LD_1B_2\rho^{-1}}.
	\end{equation}
	For the right-hand side, we have 
	\begin{align}
	\frac{\beta^K}{\sum_{k=0}^{K-1}\beta^{K-1-k}} &\overset{\eqref{eq:beta_def}}{=} 
	\frac{(1+L\eta^2 A_1 +2\alpha A_2)^K}{\sum_{k=0}^{K-1}(1+L\eta^2 A_1 +2\alpha A_2)^{K-1-k}}
	\notag\\
	&~~ \leq  \frac{(1+L\eta^2 A_1 +2\alpha A_2)^K}{K}  \notag\\
	&~~ \leq \frac{\exp\left((L\eta^2 A_1 +2\alpha A_2)K\right)}{K}  \notag\\
	&\overset{\eqref{eq:alpha_def}}{=} 
	\frac{\exp\left((L\eta^2 A_1 +L\eta^2D_1 A_2 \rho^{-1})K\right)}{K} \notag\\
	&~~\leq \frac{2}{K},  \label{eq:para2}
	\end{align} 
	where the last inequality holds by setting 
	\begin{equation}\label{eq:eta2}
	\eta \leq \sqrt{\frac{\ln 2}{(LA_1 + LD_1A_2\rho^{-1})K}}.
	\end{equation}
	For the last one $C$, we have 
	\begin{align}
	C &\overset{\eqref{eq:c_def}}{=} 
	\frac{L\eta^2}{2}C_1 + \alpha C_2  \notag\\
	&\overset{\eqref{eq:alpha_def}}{=} 
	\frac{L\eta^2}{2}C_1 + \frac{L \eta^2D_1}{2\rho}C_2  \notag\\
	&~~ = \frac{\eta^2}{2} (LC_1+LD_1C_2\rho^{-1}).  \label{eq:para3}
	\end{align} 
	Now, we plug \eqref{eq:para1}, \eqref{eq:para2} and \eqref{eq:para3} into \eqref{eq:result}: 
	\begin{align}
	\frac{\eta}{2}\E[\ns{\nabla f(\hx)}] 
	&\leq \frac{2}{K} \Delta^0 + \frac{\eta^2}{2} (LC_1+LD_1C_2\rho^{-1})
	\notag \\
	\E[\ns{\nabla f(\hx)}]  
	&\leq \frac{4}{K\eta} \Delta^0 + \eta(LC_1+LD_1C_2\rho^{-1})  \notag\\   
	&\leq \frac{\epsilon^2}{2} + \frac{\epsilon^2}{2} = \epsilon^2.    \label{eq:finalresult}
	\end{align}
	Thus, we find an $\epsilon$-solution, i.e. a point $\hx$ such that 
	$$\E[\n{\nabla f(\hx)}] \leq \sqrt{\E[\ns{\nabla f(\hx)}]} \leq \epsilon.$$
	Note that \eqref{eq:finalresult} holds by setting \begin{equation}\label{eq:eta3}
	\eta \leq \frac{\epsilon^2}{2L(C_1+D_1C_2\rho^{-1})}
	\end{equation}
	and 
	\begin{equation}
	K = \frac{8\Delta^0}{\epsilon^2}\frac{1}{\eta}=\frac{8\fgapp L}{\epsilon^2} \max \left\{ B_1+D_1B_2\rho^{-1}, \frac{12\fgapp (A_1+D_1 A_2\rho^{-1})}{\epsilon^2}, \frac{2(C_1+D_1C_2\rho^{-1})}{\epsilon^2}\right\},
	\end{equation}
	Note that according to \eqref{eq:eta1}, \eqref{eq:eta2}, \eqref{eq:eta3}, we know that $\eta$ needs to satisfy
	$$\eta \leq \min\left\{ \frac{1}{LB_1+LD_1B_2\rho^{-1}}, \sqrt{\frac{\ln 2}{(LA_1 + LD_1A_2\rho^{-1})K}}, \frac{\epsilon^2}{2L(C_1+D_1C_2\rho^{-1})} \right\}.$$
	Also, $\Delta^0 \overset{\eqref{eq:d_def}}{=} f(x^0) - \fs + \alpha \sigma_0^2 \overset{\eqref{eq:alpha_def}}{=}  f(x^0) - f^* + 2^{-1}L\eta^2D_1\rho^{-1} \sigma_0^2 \overset{\text{def}}{=} \fgapp$.
\end{proof}

\subsection{Proofs for unified theorems under PL condition}
\label{sec:proofofpl}
In this section, we provide the detailed proofs for our unified Theorem \ref{thm:main-pl} (constant stepsize) and Theorem \ref{thm:main-pl-dec} (descreasing stepsize) under PL condition.
Note that under the PL condition, one can obtain a faster linear convergence $O(\cdot\log \frac{1}{\epsilon})$ (see \eqref{eq:k-main-pl} in Theorem \ref{thm:main-pl}) rather than the sublinear convergence $O(\cdot\frac{1}{\epsilon^2})$ (see \eqref{eq:k-main} in Theorem \ref{thm:main}).
Besides, one can get a tighter convergence rate if one uses decreasing stepsize $\eta_{k}$ in $x^{k+1} = x^k - \eta_k g^k$  (Line \ref{line:update} of Algorithm \ref{alg:1}) under PL condition. See Theorem \ref{thm:main-pl} (constant stepsize) and Theorem \ref{thm:main-pl-dec} (descreasing stepsize) for a comparison.

\subsubsection{Proof of unified Theorem \ref{thm:main-pl} under PL condition (constant stepsize)}
In this section, we first provide the unified Theorem \ref{thm:main-pl} under PL condition with constant stepsize and then provide the detailed proof.

\begin{theorem}[Main theorem under PL condition with constant stepsize]\label{thm:main-pl}
	Suppose that Assumptions \ref{asp:boge}, \ref{asp:lsmooth} and \ref{asp:pl} hold. Let stepsize 
	$$\eta_k \equiv \eta \leq \min\left\{ \frac{1}{LB_1+2LD_1B_2\rho^{-1} + (L A_1 + 2LD_1A_2 \rho^{-1})\mu^{-1}},~~ \frac{\mu\epsilon}{LC_1+2LD_1C_2\rho^{-1}}  \right\},$$ 
	then the number of iterations performed by Algorithm \ref{alg:1} to find an $\epsilon$-solution, i.e. a point $x^K$ such that $\E[f(x^K)-f^*] \leq \epsilon$, can be bounded by
	\begin{align}
	K = \max \left\{ B_1+2D_1B_2\rho^{-1} + (L A_1 + 2LD_1A_2 \rho^{-1})\mu^{-1},~~ \frac{C_1+2D_1C_2\rho^{-1}}{\mu \epsilon}\right\} \kappa \log \frac{2\fgapp}{\epsilon},
	\label{eq:k-main-pl}
	\end{align}
	where $\fgapp:= f(x^0) - f^* + L\eta^2D_1\rho^{-1} \sigma_0^2$ and $\kappa := L/\mu$.
\end{theorem}

\begin{proof}
	Similar to the proof of Theorem \ref{thm:main}, we obtain the relation between $f(\xkn)$ and $f(\xk)$:
	\begin{align}
	f(\xkn) 
	&\leq f(\xk) + \inner{\nabla f(\xk)}{\xkn-\xk} + \frac{L}{2}\ns{\xkn-\xk} \label{eq:use_smooth-pl}\\
	&=f(\xk) -\eta  \inner{\nabla f(\xk)}{\gk} + \frac{L\eta^2}{2}\ns{\gk}, \label{eq:plug_update-pl}
	\end{align}
	where \eqref{eq:use_smooth-pl} uses $L$-smoothness of $f$ (see \eqref{eq:lsmooth}), and \eqref{eq:plug_update-pl} follows from the update step $\xkn = \xk - \eta \gk$ (Line \ref{line:update} of Algorithm \ref{alg:1}). Now, we take expectation for \eqref{eq:plug_update-pl} conditional on the past, i.e., $x^{0:k}$, denoted as $\Ek$:
	\begin{align}
	\Ek[f(\xkn)] 
	& \leq f(\xk) -\eta  \ns{\nabla f(\xk)} + \frac{L\eta^2}{2}\Ek[\ns{\gk}] \notag \\
	\Ek[f(\xkn) - \fs] 
	& \leq f(\xk) -\fs -\eta  \ns{\nabla f(\xk)} + \frac{L\eta^2}{2}\Ek[\ns{\gk}] \notag \\
	& \leq (1+L\eta^2 A_1) (f(\xk) -\fs) - \left(\eta-\frac{L\eta^2 B_1}{2}\right)  \ns{\nabla f(\xk)} \notag \\
	& \qquad \qquad + \frac{L\eta^2 D_1}{2} \sk + \frac{L\eta^2}{2}C_1 \notag
	\end{align}
	where these inequalities hold by using our unified Assumption \ref{asp:boge}. Then according to \eqref{eq:boge2} in Assumption \ref{asp:boge}, we have, for $\forall \alpha>0$
	\begin{align}
	\Ek[f(\xkn) - \fs + \alpha\skn] 
	& \leq (1+L\eta^2 A_1 +2\alpha A_2) (f(\xk) -\fs) 
	+ \left(\frac{L\eta^2 D_1}{2} + \alpha (1-\rho)\right) \sk
	\notag \\
	& \qquad \qquad  - \left(\eta-\frac{L\eta^2 B_1}{2}-\alpha B_2\right)  \ns{\nabla f(\xk)} + \frac{L\eta^2}{2}C_1 + \alpha C_2. \label{eq:last}
	\end{align}
	Then, we apply PL condition \eqref{eq:pl} to \eqref{eq:last}:
	\begin{align}
	\Ek[f(\xkn) - \fs + \alpha\skn] 
	& \leq (1+L\eta^2 A_1 +2\alpha A_2) (f(\xk) -\fs) 
	+ \left(\frac{L\eta^2 D_1}{2} + \alpha (1-\rho)\right) \sk
	\notag \\
	& \qquad - 2\mu\left(\eta-\frac{L\eta^2 B_1}{2}-\alpha B_2\right) (f(\xk) -\fs)  + \frac{L\eta^2}{2}C_1 + \alpha C_2.
	\end{align}
	Now, we take expectation again and define 
	\begin{align}
	\Dkn &:=f(\xkn) - \fs + \alpha\skn  \label{eq:d_def-pl}\\
	\beta &:= L\eta^2 A_1 +2\alpha A_2   \label{eq:beta_def-pl}\\
	\eta' &: = \eta-\frac{L\eta^2 B_1}{2}-\alpha B_2  \label{eq:eta_def-pl}\\
	C &: = \frac{L}{2}C_1 + \frac{\alpha}{\eta^2} C_2   \label{eq:c_def-pl}
	\end{align}
	to obtain
	\begin{align}
	\E[\Dkn] 
	& \leq (1-2\mu \eta' + \beta) \E[f(\xk) -\fs] 
	+ \left(\frac{L\eta^2 D_1}{2} + \alpha (1-\rho)\right) \E[\sk]
	+ C\eta^2 \notag \\
	& =(1-2\mu\eta' +\beta) \E\left[f(\xk) -\fs +
	\left(\frac{\frac{L\eta^2 D_1}{2} + \alpha (1-\rho)}{1-2\mu\eta' +\beta}\right) \sk\right]
	+ C\eta^2 \notag \\
	&\leq (1-\mu\eta) \E[\Dk] + C\eta^2, \label{eq:etap-pl} 
	\end{align}
	where the last inequality \eqref{eq:etap-pl} holds by setting
	\begin{equation}\label{eq:alpha_def-pl}
	\alpha =\frac{L\eta^2D_1}{\rho}, ~~~ 
	\eta \leq  \frac{1}{LB_1+2LD_1B_2\rho^{-1} + (L A_1 + 2LD_1A_2 \rho^{-1})\mu^{-1}},
	\end{equation}
	since
	\begin{align}
	1-2\mu \eta' + \beta
	&\overset{\eqref{eq:eta_def-pl}}{=} 1-2\mu(\eta-\frac{L\eta^2 B_1}{2}-\alpha B_2) + \beta \notag\\
	&\overset{\eqref{eq:beta_def-pl}}{=} 1-2\mu (\eta-\frac{L\eta^2 B_1}{2}-\alpha B_2) + L\eta^2 A_1 +2\alpha A_2 \notag\\
	&\overset{\eqref{eq:alpha_def-pl}}{=}   1-2\mu ( \eta-\frac{L\eta^2 B_1}{2}-\frac{L\eta^2D_1B_2}{\rho}) + L\eta^2 A_1 + \frac{2L\eta^2D_1}{\rho}A_2  \notag\\
	&~~= 1- 2\mu\eta \left(1-\frac{1}{2}\left(L\eta B_1+ 2L\eta D_1B_2\rho^{-1} +(L \eta A_1 + 2L\eta D_1A_2 \rho^{-1})\mu^{-1}\right) \right)  
	\notag \\
	&\overset{\eqref{eq:alpha_def-pl}}{\leq}  1-\mu\eta. \label{eq:beta-pl}
	\end{align}
	
	Telescoping \eqref{eq:etap-pl} for $0\leq k \leq K-1$, we have
	\begin{align}
	\E[\Delta^K]
	& \leq (1-\mu\eta)^K \E[\Delta^0] + C\eta^2 \sum_{k=0}^{K-1} (1-\mu\eta)^k 
	\notag\\
	&\leq  (1-\mu\eta)^K \Delta^0 +  \frac{C\eta}{\mu}   \notag\\
	&\leq \frac{\epsilon}{2} + \frac{\epsilon}{2} = \epsilon.    \label{eq:finalresult-pl}
	\end{align}
	Thus, we find an $\epsilon$-solution, i.e. a point $x^K$ such that 
	$$\E[f(x^K) - \fs] \leq E[\Delta^K] \leq \epsilon.$$
	Note that \eqref{eq:finalresult-pl} holds by setting 
	\begin{equation}\label{eq:eta3-pl}
	\eta \leq \frac{\mu\epsilon}{2C} 
	\overset{\eqref{eq:c_def-pl}}{=} \frac{\mu\epsilon}{LC_1 + \frac{2\alpha}{\eta^2}C_2}
	\overset{\eqref{eq:alpha_def-pl}}{=} \frac{\mu\epsilon}{LC_1+2LD_1C_2\rho^{-1}}
	\end{equation}
	and 
	\begin{align*}
	K &= \frac{1}{\mu\eta}\log\frac{2\Delta^0}{\epsilon} \notag\\
	&= \max \left\{ B_1+2D_1B_2\rho^{-1} + (L A_1 + 2LD_1A_2 \rho^{-1})\mu^{-1},~ \frac{C_1+2D_1C_2\rho^{-1}}{\mu \epsilon}\right\} \frac{L}{\mu}\log \frac{2\Delta^0}{\epsilon}.
	\end{align*}
	Note that according to \eqref{eq:alpha_def-pl} and \eqref{eq:eta3-pl}, we know that the step size $\eta$ needs to satisfy
	$$\eta \leq \min\left\{  \frac{1}{LB_1+2LD_1B_2\rho^{-1} + (L A_1 + 2LD_1A_2 \rho^{-1})\mu^{-1}}, \frac{\mu\epsilon}{LC_1+2LD_1C_2\rho^{-1}} \right\}.$$
	Also, $\Delta^0 \overset{\eqref{eq:d_def-pl}}{=} f(x^0) - \fs + \alpha \sigma_0^2 \overset{\eqref{eq:alpha_def-pl}}{=}  f(x^0) - f^* + L\eta^2D_1\rho^{-1} \sigma_0^2 \overset{\text{def}}{=} \fgapp$.
\end{proof}

\subsubsection{Proof of unified Theorem \ref{thm:main-pl-dec} under PL condition (decreasing stepsize)}
We first restate our unified Theorem \ref{thm:main-pl-dec} under PL condition with decreasing stepsize and then provide the detailed proof.
The decreasing stepsize $\eta_{k}$ in $x^{k+1} = x^k - \eta_k g^k$  (Line \ref{line:update} of Algorithm \ref{alg:1}) leads to a tighter convergence result (see \eqref{eq:k-main-pl} in Theorem \ref{thm:main-pl} and  \eqref{eq:k-main-pl-dec} in Theorem \ref{thm:main-pl-dec} for a comparison).

\begingroup
\def\thetheorem{\ref{thm:main-pl-dec}}
\begin{theorem}[Main theorem under PL condition with decreasing stepsize]
	Suppose that Assumptions \ref{asp:boge} , \ref{asp:lsmooth} and \ref{asp:pl} hold. Let stepsize 
	$$\eta_{k} = \begin{cases}
	\eta   & \text {if~~}  k\leq \frac{K}{2}\\
	\frac{2\eta}{2+(k-\frac{K}{2})\mu\eta} &\text{if~~}  k>\frac{K}{2}
	\end{cases},
	\text{~where~~}
	\eta \leq \frac{1}{LB_1+2LD_1B_2\rho^{-1} + (L A_1 + 2LD_1A_2 \rho^{-1})\mu^{-1}},$$ 
	then the number of iterations performed by Algorithm \ref{alg:1} to find an $\epsilon$-solution, i.e. a point $x^K$ such that $\E[f(x^K)-f^*] \leq \epsilon$, can be bounded by
	\begin{align}
	K = \max \left\{2\left(B_1+2D_1B_2\rho^{-1} +(L A_1 + 2LD_1A_2 \rho^{-1})\mu^{-1}\right)\kappa\log \frac{2\fgapp}{\epsilon}, \frac{10(C_1+2D_1C_2\rho^{-1})\kappa}{\mu \epsilon}\right\}, \label{eq:k-main-pl-dec}
	\end{align}
	where $\fgapp:= f(x^0) - f^* + L\eta^2D_1\rho^{-1} \sigma_0^2$ and $\kappa := L/\mu$.
\end{theorem}
\addtocounter{theorem}{-1}
\endgroup

\begin{proof}
	Similar to the proof of Theorem \ref{thm:main-pl}, we obtain the relation between $f(\xkn)$ and $f(\xk)$:
	\begin{align}
	f(\xkn) 
	&\leq f(\xk) + \inner{\nabla f(\xk)}{\xkn-\xk} + \frac{L}{2}\ns{\xkn-\xk} \label{eq:use_smooth-pl2}\\
	&=f(\xk) -\eta_k  \inner{\nabla f(\xk)}{\gk} + \frac{L\eta_k^2}{2}\ns{\gk}, \label{eq:plug_update-pl2}
	\end{align}
	where \eqref{eq:use_smooth-pl2} uses $L$-smoothness of $f$ (see \eqref{eq:lsmooth}), and \eqref{eq:plug_update-pl2} follows from the update step $\xkn = \xk - \eta_k \gk$ (Line \ref{line:update} of Algorithm \ref{alg:1}). Now, we take expectation for \eqref{eq:plug_update-pl2} conditional on the past, i.e., $x^{0:k}$, denoted as $\Ek$:
	\begin{align}
	\Ek[f(\xkn)] 
	& \leq f(\xk) -\eta_k  \ns{\nabla f(\xk)} + \frac{L\eta_k^2}{2}\Ek[\ns{\gk}] \notag \\
	\Ek[f(\xkn) - \fs] 
	& \leq f(\xk) -\fs -\eta_k  \ns{\nabla f(\xk)} + \frac{L\eta_k^2}{2}\Ek[\ns{\gk}] \notag \\
	& \leq (1+L\eta_k^2 A_1) (f(\xk) -\fs) - \left(\eta_k-\frac{L\eta_k^2 B_1}{2}\right)  \ns{\nabla f(\xk)} \notag \\
	& \qquad \qquad + \frac{L\eta_k^2 D_1}{2} \sk + \frac{L\eta_k^2}{2}C_1 \notag
	\end{align}
	where these inequalities hold by using our unified Assumption \ref{asp:boge}. Then according to \eqref{eq:boge2} in Assumption \ref{asp:boge}, we have, for $\forall \alpha>0$
	\begin{align}
	\Ek[f(\xkn) - \fs + \alpha\skn] 
	& \leq (1+L\eta_k^2 A_1 +2\alpha A_2) (f(\xk) -\fs) 
	+ \left(\frac{L\eta_k^2 D_1}{2} + \alpha (1-\rho)\right) \sk
	\notag \\
	& \qquad \qquad  - \left(\eta-\frac{L\eta_k^2 B_1}{2}-\alpha B_2\right)  \ns{\nabla f(\xk)} + \frac{L\eta_k^2}{2}C_1 + \alpha C_2. \label{eq:last2}
	\end{align}
	Then, we apply PL condition \eqref{eq:pl} to \eqref{eq:last2}:
	\begin{align}
	\Ek[f(\xkn) - \fs + \alpha\skn] 
	& \leq (1+L\eta_k^2 A_1 +2\alpha A_2) (f(\xk) -\fs) 
	+ \left(\frac{L\eta_k^2 D_1}{2} + \alpha (1-\rho)\right) \sk
	\notag \\
	& \qquad - 2\mu\left(\eta_k-\frac{L\eta_k^2 B_1}{2}-\alpha B_2\right) (f(\xk) -\fs)  + \frac{L\eta_k^2}{2}C_1 + \alpha C_2.
	\end{align}
	Now, we take expectation again and define 
	\begin{align}
	\Dkn &:=f(\xkn) - \fs + \alpha\skn  \label{eq:d_def-pl2}\\
	\beta_k &:= L\eta_k^2 A_1 +2\alpha A_2   \label{eq:beta_def-pl2}\\
	\eta'_k &: = \eta_k-\frac{L\eta_k^2 B_1}{2}-\alpha B_2  \label{eq:eta_def-pl2}\\
	C_k &: = \frac{L}{2}C_1 + \frac{\alpha}{\eta_k^2} C_2   \label{eq:c_def-pl2}
	\end{align}
	to obtain
	\begin{align}
	\E[\Dkn] 
	& \leq (1-2\mu \eta'_k + \beta_k) \E[f(\xk) -\fs] 
	+ \left(\frac{L\eta_k^2 D_1}{2} + \alpha (1-\rho)\right) \E[\sk]
	+ C_k\eta_k^2 \notag \\
	& =(1-2\mu\eta'_k +\beta_k) \E\left[f(\xk) -\fs +
	\left(\frac{\frac{L\eta_k^2 D_1}{2} + \alpha (1-\rho)}{1-2\mu\eta'_k +\beta_k}\right) \sk\right]
	+ C_k\eta_k^2 \notag \\
	&\leq (1-\mu\eta_k) \E[\Dk] + C_k\eta_k^2, \label{eq:etap-pl2} 
	\end{align}
	where the last inequality \eqref{eq:etap-pl2} holds by setting
	\begin{equation}\label{eq:alpha_def-pl2}
	\alpha =\frac{L\eta_k^2D_1}{\rho}, ~~  
	\eta_k \leq  \frac{1}{LB_1+2LD_1B_2\rho^{-1} + (L A_1 + 2LD_1A_2 \rho^{-1})\mu^{-1}}
	\end{equation}
	since
	\begin{align}
	1-2\mu \eta'_k + \beta_k
	&\overset{\eqref{eq:eta_def-pl2}}{=} 1-2\mu(\eta_k-\frac{L\eta_k^2 B_1}{2}-\alpha B_2) + \beta \notag\\
	&\overset{\eqref{eq:beta_def-pl2}}{=} 1-2\mu (\eta_k-\frac{L\eta_k^2 B_1}{2}-\alpha B_2) + L\eta_k^2 A_1 +2\alpha A_2 \notag\\
	&\overset{\eqref{eq:alpha_def-pl2}}{=}   1-2\mu ( \eta_k-\frac{L\eta_k^2 B_1}{2}-\frac{L\eta_k^2D_1B_2}{\rho}) + L\eta_k^2 A_1 + \frac{2L\eta_k^2D_1}{\rho}A_2  \notag\\
	&~~= 1- 2\mu\eta_k \left(1-\frac{1}{2}\left(L\eta_k B_1+ 2L\eta_k D_1B_2\rho^{-1} +(L \eta_k A_1 + 2L\eta_k D_1A_2 \rho^{-1})\mu^{-1}\right) \right)  \notag \\
	&\overset{\eqref{eq:alpha_def-pl2}}{\leq}  1-\mu\eta_k. \label{eq:beta-pl2}
	\end{align}
	Note that $C_k \overset{\eqref{eq:c_def-pl2}}{=}  \frac{L}{2}C_1 + \frac{\alpha}{\eta_k^2} C_2
	= \frac{L}{2}C_1+\frac{LD_1}{\rho}C_2$, thus we use $C:=\frac{L}{2}C_1+\frac{LD_1}{\rho}C_2$ to denote $C_k$ in \eqref{eq:etap-pl2}
	\begin{align}\label{eq:key-pl2}
	\E[\Dkn] 
	&\leq (1-\mu\eta_k) \E[\Dk] + C\eta_k^2.
	\end{align}
	To prove the convergence result with \eqref{eq:key-pl2}, we provide a key proposition as follows:
	\begin{proposition}\label{lem:keyrel}
		For any sequence $\{M_k\}_{k=0}^{K}$ satisfying 
		\begin{align}\label{eq:relation}
		M_{k}  &\leq (1-a b_k) M_{k-1} + cb_k^2, ~~\forall k \in [K]
		\end{align}
		where $a,c\geq 0$. 
		If we set $\{b_k\}$ as 
		\begin{align}\label{eq:setb}
		b_{k} = \begin{cases}
		b   & \text {if~~}  k\leq \frac{K}{2}\\
		\frac{2b}{2+(k-\frac{K}{2})ab} &\text {if~~}  k>\frac{K}{2}
		\end{cases},
		\end{align}
		then we have
		\begin{align}\label{eq:relationK}
		M_{K}  &\leq (1-a b)^{\frac{K}{2}} M_0 + \frac{10c}{a^2K}.
		\end{align}
	\end{proposition}
	\begin{proofof}{Proposition \ref{lem:keyrel}}
		Telescoping \eqref{eq:relation} for the first part $0\leq k \leq K/2$, we have 
		\begin{align}
		M_{\frac{K}{2}}  &\leq (1-a b)^\frac{K}{2} M_{0} + cb^2\sum_{k=0}^{K/2-1}(1-ab)^k \notag\\
		&\leq   (1-a b)^\frac{K}{2} M_{0} + \frac{cb}{a}. \label{eq:relation2}
		\end{align}
		Now we consider the second part $\frac{K}{2}<k<K$, we have
		\begin{align}
		M_{k}  
		&~~\leq (1-a b_k) M_{k-1} + cb_k^2 \notag\\
		&\overset{\eqref{eq:setb}}{=} (1-\frac{2ab}{2+(k-\frac{K}{2})ab}) M_{k-1} + \frac{4cb^2}{(2+(k-\frac{K}{2})ab)^2} \notag\\
		\left(2+(k-\frac{K}{2})ab\right)^2 M_k 
		&~\leq~ \left(2+(k-\frac{K}{2})ab\right)\left(2+(k-\frac{K}{2})ab-2ab\right)M_{k-1} +4cb^2 \notag\\
		&~\leq~ \left(2+(k-1-\frac{K}{2})ab\right)^2M_{k-1} +4cb^2. \label{eq:relation3}
		\end{align}
		Telescoping \eqref{eq:relation3} for the second part $\frac{K}{2}<k<K$, we have 
		\begin{align}
		\left(2+(K-\frac{K}{2})ab\right)^2 M_K 
		&\leq 4M_{\frac{K}{2}} +2cb^2K  \notag\\
		&\leq 4 (1-a b)^\frac{K}{2} M_{0} + \frac{4cb}{a} +  2cb^2K \\
		M_K  &\leq (1-a b)^\frac{K}{2} M_{0} + \frac{2c}{a^2 K} +  \frac{8c}{a^2K} \notag\\
		&=(1-a b)^\frac{K}{2} M_{0} + \frac{10c}{a^2 K}. \notag
		\end{align}
	\end{proofof}
	
	Now, we apply Lemma \ref{lem:keyrel} to equation \eqref{eq:key-pl2} and set the stepsize $\eta_k$ as 
	\begin{align}\label{eq:seteta}
	\eta_{k} = \begin{cases}
	\eta  & \text {if~~}  k\leq \frac{K}{2}\\
	\frac{2\eta}{2+(k-\frac{K}{2})\mu \eta} &\text {if~~}  k>\frac{K}{2}
	\end{cases},
	\end{align}
	then we obtain
	\begin{align}
	\E[\Delta^K]
	& \leq (1-\mu\eta)^\frac{K}{2} \E[\Delta^0] + \frac{10C}{\mu^2 K}
	\notag\\
	&\leq \frac{\epsilon}{2} + \frac{\epsilon}{2} = \epsilon.    \label{eq:finalresult-pl2}
	\end{align}
	Thus, we find an $\epsilon$-solution, i.e. a point $x^K$ such that 
	$$\E[f(x^K) - \fs] \leq E[\Delta^K] \leq \epsilon.$$
	Note that \eqref{eq:finalresult-pl2} holds by choosing 
	\begin{align}
	K &~= \max \left\{ \frac{2}{\mu\eta}\log\frac{2\Delta^0}{\epsilon},~  \frac{20C}{\mu^2\epsilon}\right\} \notag\\
	&\overset{\eqref{eq:alpha_def-pl2} }{=} \max \left\{\frac{2L( B_1+2D_1B_2\rho^{-1} +(L A_1 + 2LD_1A_2 \rho^{-1})\mu^{-1})}{\mu}\log \frac{2\fgapp}{\epsilon}, \frac{10L(C_1+2D_1C_2\rho^{-1})}{\mu^2 \epsilon}\right\}.
	\end{align}
	Note that according to \eqref{eq:alpha_def-pl2} and \eqref{eq:seteta}, we know that the step size $\eta$ needs to satisfy
	$$\eta_{k} = \begin{cases}
	\eta   & \text {if~~}  k\leq \frac{K}{2}\\
	\frac{2\eta}{2+(k-\frac{K}{2})\mu\eta} &\text {if~~}  k>\frac{K}{2}
	\end{cases},$$
	where
	$$\eta \leq \frac{1}{LB_1+2LD_1B_2\rho^{-1} + (L A_1 + 2LD_1A_2 \rho^{-1})\mu^{-1}}.$$ 
	Also, $\Delta^0 \overset{\eqref{eq:d_def-pl}}{=} f(x^0) - \fs + \alpha \sigma_0^2 \overset{\eqref{eq:alpha_def-pl}}{=}  f(x^0) - f^* + L\eta^2D_1\rho^{-1} \sigma_0^2 \overset{\text{def}}{=} \fgapp$.
\end{proof}

\newpage
\section{Convergence Results and Proofs for Nonconvex Optimization ($m=1$)}
\label{sec:specialcases}

In this section, we provide the detailed convergence rates and proofs for some specific methods (see Section \ref{sec:nonconvex}) in the single machine case (i.e., $m=1$) of nonconvex federated problem \eqref{eq:prob-fed} which reduces to the standard nonconvex problem \eqref{eq:prob} with online form \ref{prob:exp} or finite-sum form \eqref{prob:finite}, i.e., 
\begin{equation*}
\min_{x\in \R^d}   f(x), \text{~~where~}  f(x) := \E_{\zeta\sim \cD}[f(x,\zeta)],  
\text{~~~or~~~}  f(x) := \frac{1}{n}\sum_{i=1}^n{f_i(x)}.
\end{equation*}

In the following, we prove that some specific methods, i.e., GD, SGD, L-SVRG and SAGA satisfy our unified Assumption \ref{asp:boge} and thus can be captured by our unified analysis.
Then, we plug their corresponding parameters (i.e., specific values for $A_1, A_2, B_1, B_2, C_1,C_2,D_1,\rho$) into our unified Theorem \ref{thm:main} to obtain the detailed convergence rates for these methods.

\subsection{GD method}
\label{sec:gd}

We first restate our Lemma \ref{lem:gd} for GD method (Algorithm \ref{alg:gd}) and provide its proof.
Then we plug its corresponding parameters (i.e., specific values for $A_1, A_2, B_1, B_2, C_1,C_2,D_1,\rho$) into our unified Theorem \ref{thm:main} to obtain the detailed convergence rate.

\begingroup
\def\thelemma{\ref{lem:gd}}
\begin{lemma}[GD]
	Let the gradient estimator $g^k=\nabla f(x^k)$ (see Line \ref{line:grad-gd} of Algorithm \ref{alg:gd}), then 
	$g^k$ satisfies the unified Assumption \ref{asp:boge} with 
	$$A_1=C_1=D_1=0,~ B_1=1,~ \sigma_k^2 \equiv 0,~ \rho=1,~ A_2=B_2=C_2=0.$$
\end{lemma}
\addtocounter{lemma}{-1}
\endgroup

\begin{proofof}{Lemma \ref{lem:gd}}
	If gradient estimator $\gk = \nabla f(\xk)$, it is easy to see that 
	$$\Ek[\gk] = \nabla f(\xk)$$ and 
	$$\Ek[\ns{\gk}] \leq \ns{ \nabla f(\xk)}$$
	since there is no randomness in the algorithm.
	Thus, $g^k$ satisfies the unified Assumption \ref{asp:boge} with 
	$$A_1=C_1=D_1=0,~ B_1=1,~ \sigma_k^2 \equiv 0,~ \rho=1,~ A_2=B_2=C_2=0.$$
\end{proofof}

\begin{corollary}[GD]\label{cor:gd}
	Suppose that Assumption \ref{asp:lsmooth} holds. Let stepsize $\eta\leq \frac{1}{L}$, then
	the number of iterations performed by GD (Algorithm \ref{alg:gd}) to find an $\epsilon$-solution of nonconvex problem \eqref{eq:prob}, i.e. a point $\hx$ such that $\E[\n{\nabla f(\hx)}] \leq \epsilon$, can be bounded by
	$$
	K = \frac{8\fgap L}{\epsilon^2}.
	$$
\end{corollary}

\begin{proofof}{Corollary \ref{cor:gd}}
	According to our unified Theorem \ref{thm:main}, if the stepsize is chosen as 
	$$\eta \leq \min\left\{ \frac{1}{LB_1+LD_1B_2\rho^{-1}}, \sqrt{\frac{\ln 2}{(LA_1 + LD_1A_2\rho^{-1})K}}, \frac{\epsilon^2}{2L(C_1+D_1C_2\rho^{-1})} \right\} 
	= \frac{1}{L}$$
	since $B_1=1, A_1=C_1=D_1=0$ according to Lemma \ref{lem:gd},
	then the number of iterations performed by GD (Algorithm \ref{alg:gd}) to find an $\epsilon$-solution of problem \eqref{eq:prob} can be bounded by
	$$
	K = \frac{8\fgapp L}{\epsilon^2} \max \left\{ B_1+D_1B_2\rho^{-1}, \frac{12\fgapp (A_1+D_1 A_2\rho^{-1})}{\epsilon^2}, \frac{2(C_1+D_1C_2\rho^{-1})}{\epsilon^2}\right\} 
	= \frac{8\fgap L}{\epsilon^2}
	$$
	since $B_1=1, A_1=C_1=D_1=0, \sigma_0^2 = 0$, and 
	$$\fgapp:= f(x^0) - f^* + 2^{-1}L\eta^2D_1\rho^{-1} \sigma_0^2 = f(x^0) - f^* =\fgap.$$
\end{proofof}

\subsection{SGD method}
\label{sec:sgd}

We first restate our Lemma \ref{lem:sgd} for SGD method (Algorithm \ref{alg:sgd}) and provide its proof.
Then we plug its corresponding parameters (i.e., specific values for $A_1, A_2, B_1, B_2, C_1,C_2,D_1,\rho$) into our unified Theorem \ref{thm:main} to obtain the detailed convergence rate.

\begingroup
\def\thelemma{\ref{lem:sgd}}
\begin{lemma}[SGD]
	Let the gradient estimator $g^k$ (see Line \ref{line:grad-sgd} of Algorithm \ref{alg:sgd}) satisfy Assumption \ref{asp:es}, i.e.,
	$ \E_k[\ns{g^k}] \leq 2A(f(x^k)-f^*)+B\ns{\nabla f(x^k)} + C,$
	then $g^k$ satisfies the unified Assumption \ref{asp:boge} with  
	$$A_1=A,~ B_1=B,~ C_1= C,~ D_1=0,~ \sigma_k^2 \equiv 0,~ \rho=1,~ A_2=B_2=C_2=0.$$
\end{lemma}
\addtocounter{lemma}{-1}
\endgroup

\begin{proofof}{Lemma \ref{lem:sgd}}
	Suppose that the gradient estimator $g^k$ satisfies Assumption \ref{asp:es}, i.e., 
	$$\Ek[\gk] = \nabla f(\xk)$$
	and 
	$$ \Ek[\ns{\gk}] \leq 2A(f(\xk)-f^*)+B\ns{\nabla f(\xk)} + C,$$
	Then it is easy to see that $g^k$ satisfies Assumption \ref{asp:boge} with 
	$$A_1=A,~ B_1=B,~ C_1= C,~ D_1=0,~ \sigma_k^2 \equiv 0,~ \rho=1,~ A_2=B_2=C_2=0.$$
\end{proofof}

\begin{corollary}[SGD]\label{cor:sgd}
	Suppose that Assumption \ref{asp:lsmooth} holds and the gradient estimator $g^k$ in Algorithm \ref{alg:sgd} satisfies Assumption \ref{asp:es}.
	Let stepsize $\eta \leq \min\{ \frac{1}{LB}, \sqrt{\frac{\ln 2}{LAK}}, \frac{\epsilon^2}{2LC} \}$, then
	the number of iterations performed by SGD (Algorithm \ref{alg:sgd}) for finding an $\epsilon$-solution of nonconvex problem \eqref{eq:prob} with \eqref{prob:exp} or \eqref{prob:finite}, i.e. a point $\hx$ such that $\E[\n{\nabla f(\hx)}] \leq \epsilon$, can be bounded by 
	$$
	K = \frac{8\fgap L}{\epsilon^2} \max \left\{B, \frac{12\fgap A}{\epsilon^2}, \frac{2C}{\epsilon^2}\right\}.
	$$
	Note that it recovers the recent result for SGD given by \citep{khaled2020better}.
\end{corollary}

\begin{proofof}{Corollary \ref{cor:sgd}}
	According to our unified Theorem \ref{thm:main}, if the stepsize is chosen as 
	\begin{align}
	\eta &\leq \min\left\{ \frac{1}{LB_1+LD_1B_2\rho^{-1}}, \sqrt{\frac{\ln 2}{(LA_1 + LD_1A_2\rho^{-1})K}}, \frac{\epsilon^2}{2L(C_1+D_1C_2\rho^{-1})} \right\}  \notag\\
	&= \min\left\{ \frac{1}{LB}, \sqrt{\frac{\ln 2}{LAK}}, \frac{\epsilon^2}{2LC} \right\}
	\end{align}
	since $A_1=A, B_1=B, C_1= C, D_1=0 $ according to Lemma \ref{lem:sgd},
	then the number of iterations performed by SGD (Algorithm \ref{alg:sgd}) to find an $\epsilon$-solution of problem \eqref{eq:prob} with \eqref{prob:exp} or \eqref{prob:finite} can be bounded by
	\begin{align}
	K &= \frac{8\fgapp L}{\epsilon^2} \max \left\{ B_1+D_1B_2\rho^{-1}, \frac{12\fgapp (A_1+D_1 A_2\rho^{-1})}{\epsilon^2}, \frac{2(C_1+D_1C_2\rho^{-1})}{\epsilon^2}\right\}  \notag\\
	&= \frac{8\fgap L}{\epsilon^2} \max \left\{B, \frac{12\fgap A}{\epsilon^2}, \frac{2C}{\epsilon^2}\right\}
	\end{align}
	since $A_1=A, B_1=B, C_1= C, D_1=0, \sigma_0^2 = 0$, and 
	$$\fgapp:= f(x^0) - f^* + 2^{-1}L\eta^2D_1\rho^{-1} \sigma_0^2 = f(x^0) - f^* =\fgap.$$
\end{proofof}

\subsection{L-SVRG method}
\label{sec:lsvrg}

We first restate our Lemma \ref{lem:lsvrg} for L-SVRG method (Algorithm \ref{alg:lsvrg}) and provide its proof.
Then we plug its corresponding parameters (i.e., specific values for $A_1, A_2, B_1, B_2, C_1,C_2,D_1,\rho$) into our unified Theorem \ref{thm:main} to obtain the detailed convergence rate.

\begingroup
\def\thelemma{\ref{lem:lsvrg}}
\begin{lemma}[L-SVRG]
	Suppose that Assumption \ref{asp:avgsmooth} holds. The gradient estimator $g^k=\frac{1}{b} \sum_{i\in I_b} (\nabla f_i(x^k)- \nabla f_i(w^k)) +\nabla f(w^k)$ (see Line \ref{line:lsvrg} in Algorithm \ref{alg:lsvrg})
	satisfies the unified Assumption~\ref{asp:boge} with 
	$$A_1=A_2=C_1=C_2=0,$$
	$$B_1=1,~ D_1=\frac{L^2}{b},~ \sigma_k^2= \ns{x^k-w^k},~  \rho=\frac{p}{2}+\frac{p^2}{2}-\frac{\eta^2L^2}{b},~  B_2=\frac{2\eta^2}{p}-\eta^2.$$
\end{lemma}
\addtocounter{lemma}{-1}
\endgroup

\begin{proofof}{Lemma \ref{lem:lsvrg}}
	If gradient estimator $g^k=\frac{1}{b} \sum_{i\in I_b} (\nabla f_i(x^k)- \nabla f_i(w^k)) +\nabla f(w^k)$ (see Line \ref{line:lsvrg} in Algorithm \ref{alg:lsvrg}), we show the following equations:
	\begin{align}
	\Ek[\gk] &= \Ek\left[\frac{1}{b} \sum_{i\in I_b} (\nabla f_i(\xk)- \nabla f_i(\wk)) +\nabla f(\wk)\right] \notag\\
	&= \nabla f(\xk) - \nabla f(\wk) + \nabla f(\wk) = \nabla f(\xk) \label{eq:svrg0}
	\end{align}
	and 
	\begin{align}
	\Ek[\ns{\gk}] & = \Ek\left[\ns{\gk-\nabla f(\xk)}\right] + \ns{\nabla f(\xk) } \notag\\
	&= \Ek\left[\nsB{\frac{1}{b} \sum_{i\in I_b} (\nabla f_i(\xk)- \nabla f_i(\wk)) +\nabla f(\wk) - \nabla f(\xk)} \right]  + \ns{\nabla f(\xk) }  \notag\\
	&= \frac{1}{b^2} \Ek\left[\nsB{\sum_{i\in I_b} \left((\nabla f_i(\xk)- \nabla f_i(\wk)) -(\nabla f(\xk) - \nabla f(\wk)) \right)} \right]  + \ns{\nabla f(\xk) }  \notag\\
	&= \frac{1}{b} \Ek\left[\nsB{(\nabla f_i(\xk)- \nabla f_i(\wk)) -(\nabla f(\xk) - \nabla f(\wk))} \right]  + \ns{\nabla f(\xk) }  \notag\\
	&\leq \frac{1}{b} \Ek\left[\nsB{\nabla f_i(\xk)- \nabla f_i(\wk)} \right]  + \ns{\nabla f(\xk) }  \label{eq:var}\\
	& \leq  \frac{L^2}{b} \ns{\xk-\wk} + \ns{\nabla f(\xk)}, \label{eq:svrg1}
	\end{align}
	where \eqref{eq:var} uses the fact $\E[\ns{x-\E[x]}]\leq \E[\ns{x}]$, and the last inequality uses Assumption \ref{asp:avgsmooth} (i.e., \eqref{eq:avgsmooth}).
	Now, we define $\sk:=\ns{\xk-\wk}$ and obtain
	\begin{align}
	&\Ek[\skn] \notag\\
	& := \Ek[\ns{\xkn-\wkn}]  \notag\\
	&= p\Ek[\ns{\xkn-\xk}] + (1-p)\Ek[\ns{\xkn-\wk}]  \label{eq:useprob}\\
	&=p\eta^2\Ek[\ns{\gk}] + (1-p)\Ek[\ns{\xk-\eta \gk-\wk}]   \label{eq:useupdate} \\
	&=p\eta^2\Ek[\ns{\gk}] + (1-p)\Ek[\ns{\xk-\wk} + \ns{\eta \gk} -2\inner{\xk-\wk}{\eta\gk}]  \notag\\
	&= p\eta^2\Ek[\ns{\gk}] + (1-p)\ns{\xk-\wk} + (1-p)\eta^2\Ek[\ns{\gk}] -2(1-p)\inner{\xk-\wk}{\eta \nabla f(\xk)} \notag \\
	&\leq \eta^2\Ek[\ns{\gk}] + (1-p)\ns{\xk-\wk} +(1-p)\beta\ns{\xk-\wk} + \frac{1-p}{\beta}\ns{\eta \nabla f(\xk)} \label{eq:useyoung} \\
	&= \eta^2\Ek[\ns{\gk}] + (1-p)(1+\beta)\ns{\xk-\wk}  
	+ \frac{(1-p)\eta^2}{\beta}\ns{\nabla f(\xk)} \notag
	\end{align}
	\begin{align}
	&\leq \eta^2 \left(\frac{L^2}{b} \ns{\xk-\wk} + \ns{\nabla f(\xk)} \right)
	+(1-p)(1+\beta)\ns{\xk-\wk}  
	+ \frac{(1-p)\eta^2}{\beta}\ns{\nabla f(\xk)}  \label{eq:usesvrg1}\\
	&=  \left((1-p)(1+\beta) +\frac{\eta^2L^2}{b}\right)\ns{x^k-w^k} + \left( \frac{(1-p)\eta^2}{\beta} + \eta^2\right)\ns{\nabla f(x^k)}  \notag\\
	&=  \left(1-\frac{p}{2}-\frac{p^2}{2}+\frac{\eta^2L^2}{b}\right)\ns{x^k-w^k} + \left(\frac{2\eta^2}{p}-\eta^2\right)\ns{\nabla f(x^k)}, \label{eq:svrg2}
	\end{align}
	where \eqref{eq:useprob} uses Line \ref{line:w_prob} of Algorithm \ref{alg:lsvrg}, \eqref{eq:useupdate} uses Line \ref{line:update-lsvrg} of Algorithm \ref{alg:lsvrg}, \eqref{eq:useyoung} uses Young's inequality for $\forall \beta>0$, \eqref{eq:usesvrg1} uses \eqref{eq:svrg1}, and \eqref{eq:svrg2} holds by setting $\beta=p/2$.
	
	Now, according to \eqref{eq:svrg0}, \eqref{eq:svrg1} and \eqref{eq:svrg2}, we know $g^k$ satisfies the unified Assumption \ref{asp:boge} with $$A_1=A_2=C_1=C_2=0,$$
	$$B_1=1,~ D_1=\frac{L^2}{b},~ \sigma_k^2= \ns{x^k-w^k},~  \rho=\frac{p}{2}+\frac{p^2}{2}-\frac{\eta^2L^2}{b},~  B_2=\frac{2\eta^2}{p}-\eta^2.$$
\end{proofof}

\begin{corollary}[L-SVRG]\label{cor:lsvrg}
	Suppose that Assumption \ref{asp:avgsmooth} holds.
	Let stepsize $\eta \leq \frac{1}{L(1+2b^{-1/3}p^{-2/3})}$, 
	then the number of iterations performed by L-SVRG (Algorithm \ref{alg:lsvrg}) for finding an $\epsilon$-solution of nonconvex problem \eqref{eq:prob} with \eqref{prob:finite}, i.e. a point $\hx$ such that $\E[\n{\nabla f(\hx)}] \leq \epsilon$, can be bounded by 
	$$
	K = \frac{8\fgap L}{\epsilon^2} \left(1+ \frac{2}{b^{1/3}p^{2/3}}\right).
	$$
	In particular, we have
	\begin{enumerate}
		\item let minibatch size $b=1$ and probability $p=\frac{1}{n}$, then the number of iterations $K = \frac{24\fgap  Ln^{2/3}}{\epsilon^2}$.
		\item let minibatch size $b=n^{2/3}$ and probability $p=\frac{1}{n^{1/3}}$, then the number of iterations $K =\frac{24\fgap  L}{\epsilon^2}$, but each iteration costs $n^{2/3}$ due to minibatch size $b=n^{2/3}$.
	\end{enumerate}
\end{corollary}

\topic{Remark} i) The analysis and results for L-SVRG in this nonconvex case are new. Previous work studied the standard SVRG form \citep{johnson2013accelerating} not this simpler loopless version.
ii) L-SVRG enjoys a linear speedup for parallel computation, i.e., if one can parallel compute $b$ minibatch stochastic gradients, then one can finding an $\epsilon$-solution within $O(\frac{n^{2/3}}{\epsilon^2 b})$ steps by letting $p=\frac{b}{n}$ while a single node (minibatch $b=1$) needs $O(\frac{n^{2/3}}{\epsilon^2})$ steps. In the best case, one can achieve $b=n^{2/3}$ times acceleration via parallel computation.

\begin{proofof}{Corollary \ref{cor:lsvrg}}
	According to our unified Theorem \ref{thm:main}, the stepsize should be chosen as 
	\begin{align}
	\eta &\leq \min\left\{ \frac{1}{LB_1+LD_1B_2\rho^{-1}}, \sqrt{\frac{\ln 2}{(LA_1 + LD_1A_2\rho^{-1})K}}, \frac{\epsilon^2}{2L(C_1+D_1C_2\rho^{-1})} \right\}  \notag\\
	&= \frac{1}{LB_1+LD_1B_2\rho^{-1}}
	\end{align}
	since $A_1=A_2=C_1=C_2=0$ according to Lemma \ref{lem:lsvrg}.
	Then the number of iterations performed by L-SVRG (Algorithm \ref{alg:lsvrg}) to find an $\epsilon$-solution of problem \eqref{eq:prob} with \eqref{prob:finite} can be bounded by
	\begin{align}
	K &= \frac{8\fgapp L}{\epsilon^2} \max \left\{ B_1+D_1B_2\rho^{-1}, \frac{12\fgapp (A_1+D_1 A_2\rho^{-1})}{\epsilon^2}, \frac{2(C_1+D_1C_2\rho^{-1})}{\epsilon^2}\right\}  \notag\\
	&= \frac{8\fgap L}{\epsilon^2} (B_1+D_1B_2\rho^{-1})
	\end{align}
	since $A_1=A_2=C_1=C_2=0, \sigma_0^2 = \ns{x^0-w^0}=0$, and $\fgapp:= f(x^0) - f^* + 2^{-1}L\eta^2D_1\rho^{-1} \sigma_0^2 = f(x^0) - f^* =\fgap$.
	
	Now, the remaining thing is to upper bound the term $B_1+D_1B_2\rho^{-1}$,
	\begin{align}
	B_1+D_1B_2\rho^{-1} 
	&= 1+ \frac{L^2}{b}\left(\frac{2\eta^2}{p}-\eta^2\right) \left(\frac{p}{2}+\frac{p^2}{2}-\frac{\eta^2L^2}{b}\right)^{-1} \label{eq:svrgpara}\\
	& \leq 1+ \frac{L^2}{b}\left(\frac{2\eta^2}{p}\right) \left(\frac{p}{4}\right)^{-1}      \label{eq:pb} \\
	& = 1+ \frac{8L^2\eta^2}{bp^2}  \notag\\
	&\leq 1+ \frac{2}{b^{1/3}p^{2/3}}, \label{eq:plugeta}  
	\end{align}
	where \eqref{eq:svrgpara} follows from $B_1=1,  D_1=\frac{L^2}{b}, B_2=\frac{2\eta^2}{p}-\eta^2, \rho=\frac{p}{2}+\frac{p^2}{2}-\frac{\eta^2L^2}{b}$ in Lemma \ref{lem:lsvrg}, 
	the last inequality \eqref{eq:plugeta} holds by setting 
	$\eta \leq \frac{1}{L(1+\frac{2}{b^{1/3}p^{2/3}})} 
	\leq \frac{1}{LB_1+LD_1B_2\rho^{-1}}$, 
	and \eqref{eq:pb} is due to the fact $\frac{\eta^2L^2}{b}\leq \frac{p}{4}$.
	
	In sum, let stepsize 
	$$	\eta \leq \frac{1}{L(1+\frac{2}{b^{1/3}p^{2/3}})},$$ 
	then the number of iterations performed by L-SVRG (Algorithm \ref{alg:lsvrg}) to find an $\epsilon$-solution of problem \eqref{eq:prob} with \eqref{prob:finite} can be bounded by
	\begin{align}
	K &= \frac{8\fgap L}{\epsilon^2}\left(1+ \frac{2}{b^{1/3}p^{2/3}}\right).
	\end{align}
	In particular, we have
	\begin{enumerate}
		\item let minibatch size $b=1$ and probability $p=\frac{1}{n}$, then the number of iterations $K = \frac{24\fgap  Ln^{2/3}}{\epsilon^2}$.
		\item let minibatch size $b=n^{2/3}$and probability $p=\frac{1}{n^{1/3}}$, then the number of iterations $K =\frac{24\fgap  L}{\epsilon^2}$, but each iteration costs $n^{2/3}$ due to minibatch size $b=n^{2/3}$.
	\end{enumerate}
\end{proofof}

\subsection{SAGA method}
\label{sec:saga}

We first restate our Lemma \ref{lem:saga} for SAGA method (Algorithm \ref{alg:saga}) and provide its proof.
Then we plug its corresponding parameters (i.e., specific values for $A_1, A_2, B_1, B_2, C_1,C_2,D_1,\rho$) into our unified Theorem \ref{thm:main} to obtain the detailed convergence rate.

\begingroup
\def\thelemma{\ref{lem:saga}}
\begin{lemma}[SAGA]
	Suppose that Assumption \ref{asp:avgsmooth-saga} holds. The gradient estimator $g^k= \frac{1}{b} \sum_{i\in I_b} (\nabla f_i(x^k)- \nabla f_i(w_i^k)) +\frac{1}{n}\sum_{j=1}^{n}\nabla f_j(w_j^k)$ (see Line \ref{line:saga} in Algorithm \ref{alg:saga})
	satisfies the unified Assumption~\ref{asp:boge} with $$A_1=A_2=C_1=C_2=0,$$
	$$B_1=1,~ D_1=\frac{L^2}{b},~ \sigma_k^2=\frac{1}{n}\sum_{i=1}^{n} \ns{x^k-w_i^k},~  \rho=\frac{b}{2n}+\frac{b^2}{2n^2}-\frac{\eta^2L^2}{b},~ B_2=\frac{2\eta^2n}{b}-\eta^2.$$
\end{lemma}
\addtocounter{lemma}{-1}
\endgroup

\begin{proofof}{Lemma \ref{lem:saga}}
	If gradient estimator $g^k=\frac{1}{b} \sum_{i\in I_b} (\nabla f_i(\xk)- \nabla f_i(\wik)) +\frac{1}{n}\sum_{j=1}^{n}\nabla f_j(\wjk)$ (see Line \ref{line:saga} in Algorithm \ref{alg:saga}), we show the following equations:
	\begin{align}
	\Ek[\gk] &= \Ek\left[\frac{1}{b} \sum_{i\in I_b} (\nabla f_i(\xk)- \nabla f_i(\wik)) +\frac{1}{n}\sum_{j=1}^{n}\nabla f_j(\wjk)\right] \notag\\
	&= \nabla f(\xk) - \frac{1}{n}\sum_{i=1}^{n}\nabla f_i(\wik)+ \frac{1}{n}\sum_{j=1}^{n}\nabla f_j(\wjk) \notag\\
	&= \nabla f(\xk) \label{eq:saga0}
	\end{align}
	and 
	\begin{align}
	&\Ek[\ns{\gk}] \notag\\
	& = \Ek\left[\ns{\gk-\nabla f(\xk)}\right] + \ns{\nabla f(\xk) } \notag\\
	&= \Ek\left[\nsB{\frac{1}{b} \sum_{i\in I_b} (\nabla f_i(\xk)- \nabla f_i(\wik)) +\frac{1}{n}\sum_{j=1}^{n}\nabla f_j(\wjk) - \nabla f(\xk)} \right]  + \ns{\nabla f(\xk) }  \notag \\
	&= \frac{1}{b^2} \Ek\left[\nsB{\sum_{i\in I_b} \left(\left(\nabla f_i(\xk)- \nabla f_i(\wik)\right) -\left(\frac{1}{n}\sum_{j=1}^{n}\nabla f_j(\xk) - \frac{1}{n}\sum_{j=1}^{n}\nabla f_j(\wjk)\right) \right)} \right]  \notag\\
	&\qquad  + \ns{\nabla f(\xk) }  \notag\\
	&= \frac{1}{b^2} \Ek\left[\sum_{i\in I_b}\nsB{\left(\nabla f_i(\xk)- \nabla f_i(\wik)\right) -\left(\frac{1}{n}\sum_{j=1}^{n}\nabla f_j(\xk) - \frac{1}{n}\sum_{j=1}^{n}\nabla f_j(\wjk)\right)} \right]  \notag\\
	&\qquad  + \ns{\nabla f(\xk) }  \notag\\
	&\leq \frac{1}{b} \Ek\left[\nsB{\nabla f_i(\xk)- \nabla f_i(\wik)} \right]  + \ns{\nabla f(\xk) }  \label{eq:var1}\\
	& \leq  \frac{L^2}{b} \frac{1}{n}\sum_{i=1}^{n} \ns{\xk-\wik} + \ns{\nabla f(\xk)}, \label{eq:saga1}
	\end{align}
	where \eqref{eq:var1} uses the fact $\E[\ns{x-\E[x]}]\leq \E[\ns{x}]$, and the last inequality uses Assumption \ref{asp:avgsmooth-saga} (i.e., \eqref{eq:avgsmooth-saga}).
	Now, we denote $\sk:=\frac{1}{n}\sum_{i=1}^{n} \ns{\xk-\wik}$ and obtain
	\begin{align}
	\Ek[\skn] & := \EkB{\frac{1}{n}\sum_{i=1}^{n} \ns{\xkn- \wikn}}  \notag\\
	&= \EkB{\frac{1}{n}\sum_{i=1}^{n} \frac{b}{n}\ns{\xkn-\xk} 
		+ \frac{1}{n}\sum_{i=1}^{n} \left(1-\frac{b}{n}\right)\ns{\xkn-\wik} }  \label{eq:usewi}\\
	&= \frac{b}{n}\Ek\eta^2\ns{\gk} 
	+ \left(1-\frac{b}{n}\right)\EkB{\frac{1}{n}\sum_{i=1}^{n} \ns{\xk-\eta\gk-\wik} }  \label{eq:useupdate1} \\
	&= \frac{b\eta^2}{n}\Ek\ns{\gk} 
	+ \left(1-\frac{b}{n}\right)\EkB{\frac{1}{n}\sum_{i=1}^{n} \left(\ns{\xk-\wik} + \ns{\eta \gk} -2\inner{\xk-\wik}{\eta\gk}\right)}  \notag
	\end{align}
	\begin{align}
	&= \eta^2\Ek\ns{\gk} 
	+ \left(1-\frac{b}{n}\right)\frac{1}{n}\sum_{i=1}^{n} \ns{\xk-\wik} 
	+ 2\left(1-\frac{b}{n}\right)\frac{1}{n}\sum_{i=1}^{n}\inner{\xk-\wik}{\eta \nabla f(\xk)}  \notag\\
	&\leq \eta^2\Ek\ns{\gk} 
	+ \left(1-\frac{b}{n}\right)\frac{1}{n}\sum_{i=1}^{n} \ns{\xk-\wik} \notag\\
	&\qquad \qquad 
	+ \left(1-\frac{b}{n}\right)\frac{1}{n}\sum_{i=1}^{n}\left(\beta\ns{\xk-\wik} +\frac{\eta^2}{\beta}\ns{\nabla f(\xk)}\right) \label{eq:useyoung1} \\
	&= \eta^2\Ek\ns{\gk} 
	+ \left(1-\frac{b}{n}\right)\left(1+\beta\right)\frac{1}{n}\sum_{i=1}^{n} \ns{\xk-\wik} 
	+ \left(1-\frac{b}{n}\right)\frac{\eta^2}{\beta}\ns{\nabla f(\xk)} \notag\\
	&\leq \eta^2 \left(\frac{L^2}{b} \frac{1}{n}\sum_{i=1}^{n} \ns{\xk-\wik} + \ns{\nabla f(\xk)}\right)  \notag\\
	&\qquad \qquad 
	+ \left(1-\frac{b}{n}\right)\left(1+\beta\right)\frac{1}{n}\sum_{i=1}^{n} \ns{\xk-\wik} 
	+ \left(1-\frac{b}{n}\right)\frac{\eta^2}{\beta}\ns{\nabla f(\xk)} \label{eq:usesaga1}\\
	&=  \left(\left(1-\frac{b}{n}\right)\left(1+\beta\right) +\frac{\eta^2L^2}{b}\right)\frac{1}{n}\sum_{i=1}^{n} \ns{\xk-\wik} + \left(\left(1-\frac{b}{n}\right)\frac{\eta^2}{\beta}+ \eta^2\right)\ns{\nabla f(x^k)}  \notag\\
	&=  \left(1-\frac{b}{2n}-\frac{b^2}{2n^2}+\frac{\eta^2L^2}{b}\right)\frac{1}{n}\sum_{i=1}^{n} \ns{\xk-\wik}  + \left(\frac{2\eta^2n}{b}-\eta^2\right)\ns{\nabla f(x^k)}, \label{eq:saga2}
	\end{align}
	where \eqref{eq:usewi} uses Line \ref{line:wi} of Algorithm \ref{alg:saga}, \eqref{eq:useupdate1} uses Line \ref{line:update-saga} of Algorithm \ref{alg:saga}, \eqref{eq:useyoung1} uses Young's inequality for $\forall \beta>0$, \eqref{eq:usesaga1} uses \eqref{eq:saga1}, and \eqref{eq:saga2} holds by setting $\beta=b/2n$.
	
	Now, according to \eqref{eq:saga0}, \eqref{eq:saga1} and \eqref{eq:saga2}, we know $g^k$ satisfies the unified Assumption \ref{asp:boge} with $$A_1=A_2=C_1=C_2=0,$$
	$$B_1=1,~ D_1=\frac{L^2}{b},~ \sigma_k^2=\frac{1}{n}\sum_{i=1}^{n} \ns{x^k-w_i^k},~  \rho=\frac{b}{2n}+\frac{b^2}{2n^2}-\frac{\eta^2L^2}{b},~ B_2=\frac{2\eta^2n}{b}-\eta^2.$$
\end{proofof}

\begin{corollary}[SAGA]\label{cor:saga}
	Suppose that Assumption \ref{asp:avgsmooth-saga} holds.
	Let stepsize $\eta \leq \frac{1}{L(1+2n^{2/3}b^{-1})}$, 
	then the number of iterations performed by SAGA (Algorithm \ref{alg:saga}) for finding an $\epsilon$-solution of nonconvex problem \eqref{eq:prob} with \eqref{prob:finite}, i.e. a point $\hx$ such that $\E[\n{\nabla f(\hx)}] \leq \epsilon$, can be bounded by 
	$$
	K = \frac{8\fgap L}{\epsilon^2} \left(1+ \frac{2n^{2/3}}{b}\right).
	$$
	In particular, we have
	\begin{enumerate}
		\item let minibatch size $b=1$, then the number of iterations $K = \frac{24\fgap  Ln^{2/3}}{\epsilon^2}$.
		\item let minibatch size $b=n^{2/3}$, then the number of iterations $K =\frac{24\fgap  L}{\epsilon^2}$, but each iteration costs $n^{2/3}$ due to minibatch size $b=n^{2/3}$.
	\end{enumerate}
\end{corollary}

\topic{Remark} Similar to L-SVRG, SAGA also enjoys a linear speedup for parallel computation, i.e., if one can parallel compute $b$ minibatch stochastic gradients, then one can finding an $\epsilon$-solution within $O(\frac{n^{2/3}}{\epsilon^2 b})$ steps while a single node (minibatch $b=1$) needs $O(\frac{n^{2/3}}{\epsilon^2})$ steps. In the best case, one can achieve $b=n^{2/3}$ times acceleration via parallel computation.

\begin{proofof}{Corollary \ref{cor:saga}}
	According to our unified Theorem \ref{thm:main}, the stepsize should be chosen as 
	\begin{align}
	\eta &\leq \min\left\{ \frac{1}{LB_1+LD_1B_2\rho^{-1}}, \sqrt{\frac{\ln 2}{(LA_1 + LD_1A_2\rho^{-1})K}}, \frac{\epsilon^2}{2L(C_1+D_1C_2\rho^{-1})} \right\}  \notag\\
	&= \frac{1}{LB_1+LD_1B_2\rho^{-1}}
	\end{align}
	since $A_1=A_2=C_1=C_2=0$ according to Lemma \ref{lem:saga}.
	Then the number of iterations performed by SAGA (Algorithm \ref{alg:saga}) to find an $\epsilon$-solution of problem \eqref{eq:prob} with \eqref{prob:finite} can be bounded by
	\begin{align}
	K &= \frac{8\fgapp L}{\epsilon^2} \max \left\{ B_1+D_1B_2\rho^{-1}, \frac{12\fgapp (A_1+D_1 A_2\rho^{-1})}{\epsilon^2}, \frac{2(C_1+D_1C_2\rho^{-1})}{\epsilon^2}\right\}  \notag\\
	&= \frac{8\fgap L}{\epsilon^2} (B_1+D_1B_2\rho^{-1})
	\end{align}
	since $A_1=A_2=C_1=C_2=0, \sigma_0^2 = \ns{x^0-w^0}=0$, and $\fgapp:= f(x^0) - f^* + 2^{-1}L\eta^2D_1\rho^{-1} \sigma_0^2 = f(x^0) - f^* =\fgap$.
	
	Now, the remaining thing is to upper bound the term $B_1+D_1B_2\rho^{-1}$,
	\begin{align}
	B_1+D_1B_2\rho^{-1} 
	&= 1+ \frac{L^2}{b}\left(\frac{2\eta^2n}{b}-\eta^2\right) \left(\frac{b}{2n}+\frac{b^2}{2n^2}-\frac{\eta^2L^2}{b}\right)^{-1} \label{eq:sagapara}\\
	& \leq 1+ \frac{L^2}{b}\left(\frac{2\eta^2n}{b}\right) \left(\frac{b}{4n}\right)^{-1}      \label{eq:pb1} \\
	& = 1+ \frac{8n^2L^2\eta^2}{b^3}  \notag\\
	&\leq 1+ \frac{2n^{2/3}}{b}, \label{eq:plugeta1}  
	\end{align}
	where \eqref{eq:sagapara} follows from $B_1=1,  D_1=\frac{L^2}{b},
	B_2=\frac{2\eta^2n}{b}-\eta^2,
	\rho=\frac{b}{2n}+\frac{b^2}{2n^2}-\frac{\eta^2L^2}{b}$ in Lemma \ref{lem:saga}, 
	the last inequality \eqref{eq:plugeta1} holds by setting 
	$\eta \leq \frac{1}{L(1+\frac{2}{b^{1/3}p^{2/3}})} 
	\leq \frac{1}{LB_1+LD_1B_2\rho^{-1}}$, 
	and \eqref{eq:pb1} is due to the fact $\frac{\eta^2L^2}{b}\leq \frac{b}{4n}$.
	
	In sum, let stepsize 
	$$	\eta \leq \frac{1}{L(1+\frac{2n^{2/3}}{b})},$$ 
	then the number of iterations performed by SAGA (Algorithm \ref{alg:saga}) to find an $\epsilon$-solution of problem \eqref{eq:prob} with \eqref{prob:finite} can be bounded by
	\begin{align}
	K &= \frac{8\fgap L}{\epsilon^2}\left(1+ \frac{2n^{2/3}}{b}\right).
	\end{align}
	In particular, we have
	\begin{enumerate}
		\item let minibatch size $b=1$, then the number of iterations $K = \frac{24\fgap  Ln^{2/3}}{\epsilon^2}$.
		\item let minibatch size $b=n^{2/3}$, then the number of iterations $K =\frac{24\fgap  L}{\epsilon^2}$, but each iteration costs $n^{2/3}$ due to minibatch size $b=n^{2/3}$.
	\end{enumerate}
\end{proofof}

\newpage
\section{Convergence Results and Proofs for Nonconvex Federated Optimization}
\label{sec:federated-app}

In this section, we provide the detailed convergence rates and proofs for the more general nonconvex distributed/federated problem \eqref{eq:prob-fed} with online form \eqref{prob-fed:exp} or finite-sum form \eqref{prob-fed:finite}, i.e.,
\begin{equation*}
\min_{x\in \R^d} \bigg\{  f(x) := \frac{1}{m}\sum_{i=1}^m{f_i(x)}  \bigg\}, 
\text{~~where~}  f_i(x) := \E_{\zeta \sim \cD_i}[f_i(x,\zeta)],  
\text{~~or~~}  f_i(x) := \frac{1}{n}\sum_{j=1}^n{f_{i,j}(x)}.
\end{equation*}
Here we allow that different machine/worker $i\in [m]$ can have different data distribution $\cD_i$, i.e., non-IID data (heterogeneous data) setting.

Note that in this general distributed/federated problem, the bottleneck usually is the communication cost among all workers. Thus we focus on the compressed gradient methods. 
Here we recall the definition of  compression operator.
\begingroup
\def\thedefinition{\ref{def_compression}}
\begin{definition}[Compression operator]
	A randomized map $\calC: \R^d\mapsto \R^d$ is an $\omega$-compression operator  if  
	\begin{equation}\label{eq:compress}
	\E[\calC(x)]=x,  \qquad \E[\ns{\calC(x)-x}] \leq \omega\ns{x}, \qquad \forall x\in \R^d.
	\end{equation}
	In particular, no compression ($\calC(x)\equiv x$) implies $\omega=0$.
\end{definition}
\addtocounter{definition}{-1}
\endgroup

In the following, we prove that several (new) methods belonging to the proposed general DC framework (Algorithm \ref{alg:dc}) and DIANA framework (Algorithm \ref{alg:diana}) for solving this general distributed/federated problem also satisfy the unified Assumption \ref{asp:boge}
and thus can also be captured by our unified analysis.
Then, we plug their corresponding parameters into our unified Theorem \ref{thm:main} to obtain the detailed convergence rates for these methods. 

\subsection{DC framework for nonconvex federated optimization}
\label{sec:dc-app}

We first prove a general Theorem for DC framework (Algorithm \ref{alg:dc})  which shows that several (new) methods belonging to the general DC framework satisfy Assumption \ref{asp:boge} and thus can be captured by our unified analysis.
Then, we plug their corresponding parameters into our unified Theorem \ref{thm:main} to obtain the detailed convergence rates for these methods. 

Before proving the Theorem \ref{thm:dc-diff}, we first provide a simple version as in Theorem \ref{thm:dc} where all workers share the same variance term $\tsk$ (see \eqref{eq:gi1-dc}).
If the parallel workers use GD, SGD or L-SVRG for computing their local stochastic gradient $\widetilde{g}_i^k$ (see Line \ref{line:localgrad-dc} of Algorithm \ref{alg:dc}), then they indeed share the same variance term $\tsk$, i.e., Theorem \ref{thm:dc} includes these settings.
However, if the parallel workers use SAGA-type methods for computing their local stochastic gradient $\widetilde{g}_i^k$, then the variance term $\ski$ (see \eqref{eq:gi1-dc-diff}) is different for different worker $i$, i.e., the more general Theorem \ref{thm:dc-diff} includes this SAGA setting while Theorem \ref{thm:dc} does not.

\begin{theorem}[DC framework with same variance for all workers]\label{thm:dc} 
	Suppose that the local stochastic gradient $\widetilde{g}_i^k$ (see Line \ref{line:localgrad-dc} of Algorithm \ref{alg:dc}) satisfies 
	\begin{align}
	\E_k[\ns{\widetilde{g}_i^k}] & \leq 2A_{1,i}(f_i(x^k)-f_i^*)+B_{1,i}\ns{\nabla f_i(x^k)} + \blue{D_{1}' \widetilde{\sigma}_{k}^2} +C_{1,i},  \label{eq:gi1-dc}\\
	\E_k[\widetilde{\sigma}_{k+1}^2] & \leq \blue{(1-\rho')\widetilde{\sigma}_{k}^2 + 2A_{2}'(f(x^k)-f^*)+B_{2}'\ns{\nabla f(x^k)} + D_2'\E_k[\ns{g^k}] +C_2'}, \label{eq:gi2-dc}
	\end{align}
	then $g^k$ (see Line \ref{line:gk-dc} of Algorithm \ref{alg:dc}) satisfies the unified Assumption \ref{asp:boge}, i.e., 
	\begin{align}
	\E_k[\ns{g^k}] & \leq 2A_1(f(x^k)-f^*)+B_1\ns{\nabla f(x^k)} + D_1 \sigma_k^2 +C_1, \label{eq:gk1-dc}\\
	\E_k[\sigma_{k+1}^2] & \leq (1-\rho)\sigma_k^2 + 2A_2(f(x^k)-f^*)+B_2\ns{\nabla f(x^k)} +C_2, \label{eq:gk2-dc}
	\end{align}
	with parameters
	\begin{align*}
	&A_1 =\frac{(1+\omega)A}{m}, \qquad
	B_1 =1, \qquad 
	C_1  =  \frac{(1+\omega)C}{m}, \\ 
	&D_1 =\frac{1+\omega}{m}, \qquad
	\sk =  D_{1}'\tsk, \qquad
	\rho =\rho'-\tau, \\
	&A_2 =D_{1}'A_{2}'+ \tau A, \qquad
	B_2 =D_{1}'B_{2}'+D_{1}'D_2', \qquad 
	C_2  = D_{1}'C_2' +  \tau C,
	\end{align*}
	where 
	$A:=\max_i  (A_{1,i}+B_{1,i}L_i-L_i/(1+\omega))$,
	$C := \frac{1}{m} \summ C_{1,i} + 2A\Delta_f^*$,
	$\Delta_f^*:=f^*-\frac{1}{m}\summ f_i^*$,
	and $\tau:= \frac{(1+\omega)D_{1}'D_2'}{m}$.
\end{theorem}

Before providing the proof for Theorem \ref{thm:dc}, we recall the more general Theorem \ref{thm:dc-diff} here for better comparison. Then we provide the detailed proofs for Theorems \ref{thm:dc} and \ref{thm:dc-diff}.

\begingroup
\def\thetheorem{\ref{thm:dc-diff}}
\begin{theorem}[DC framework with different variance for different worker]
	Suppose that the local stochastic gradient $\widetilde{g}_i^k$ (see Line \ref{line:localgrad-dc} of Algorithm \ref{alg:dc}) satisfies 
	\begin{align}
	\E_k[\ns{\widetilde{g}_i^k}] & \leq 2A_{1,i}(f_i(x^k)-f_i^*)+B_{1,i}\ns{\nabla f_i(x^k)} + \blue{D_{1,i} \sigma_{k,i}^2} +C_{1,i}, \label{eq:gi1-dc-diff}\\
	\E_k[\sigma_{k+1,i}^2] & \leq \blue{(1-\rho_i)\sigma_{k,i}^2 + 2A_{2,i}(f(x^k)-f^*)+B_{2,i}\ns{\nabla f(x^k)} + D_{2,i}\E_k[\ns{g^k}] +C_{2,i}}, \label{eq:gi2-dc-diff}
	\end{align}
	then $g^k$ (see Line \ref{line:gk-dc} of Algorithm \ref{alg:dc}) satisfies the unified Assumption \ref{asp:boge}
	i.e., 
	\begin{align}
	\E_k[\ns{g^k}] & \leq 2A_1(f(x^k)-f^*)+B_1\ns{\nabla f(x^k)} + D_1 \sigma_k^2 +C_1, \label{eq:gk1-dc-diff}\\
	\E_k[\sigma_{k+1}^2] & \leq (1-\rho)\sigma_k^2 + 2A_2(f(x^k)-f^*)+B_2\ns{\nabla f(x^k)} +C_2, \label{eq:gk2-dc-diff}
	\end{align}
	with parameters
	\begin{align*}
	&A_1 =\frac{(1+\omega)A}{m}, \qquad
	B_1 =1, \qquad 
	C_1  =  \frac{(1+\omega)C}{m}, \\ 
	&D_1 =\frac{1+\omega}{m}, \qquad
	\sk =  \frac{1}{m}\summ  D_{1,i} \ski , \qquad
	\rho =\min_i\rho_i-\tau, \\
	&A_2 =D_A+ \tau A, \qquad
	B_2 =D_B+D_D, \qquad 
	C_2  = D_C+  \tau C, 
	\end{align*}
	where 
	$A:=\max_i  (A_{1,i}+B_{1,i}L_i-L_i/(1+\omega))$,
	$C := \frac{1}{m} \summ C_{1,i} + 2A\Delta_f^*$,
	$\Delta_f^*:=f^*-\frac{1}{m}\summ f_i^*$,
	$\tau:=\frac{(1+\omega)D_D}{m}$,
	$D_A:=\frac{1}{m} \summ  D_{1,i}A_{2,i}$,
	$D_B:=\frac{1}{m} \summ  D_{1,i} B_{2,i}$,
	$D_D:=\frac{1}{m} \summ  D_{1,i} D_{2,i}$,
	and 
	$D_C:=\frac{1}{m} \summ  D_{1,i}C_{2,i}$.
\end{theorem}
\addtocounter{theorem}{-1}
\endgroup

\begin{proofof}{Theorem \ref{thm:dc}}
	First, we show the that gradient estimator $\gk$ (see Line \ref{line:gk-dc} of Algorithm \ref{alg:dc}) is unbiased:
	\begin{align}
	\Ek[\gk] 
	&=\Ek\left[\frac{1}{m}\summ \Ci(\gi) \right]  
	=\frac{1}{m}\summ \nabla f_i(\xk)=\nabla f(\xk). 
	\end{align}
	Then, we prove the upper bound for the second moment of gradient estimator $\gk$:
	\begin{align}
	&\Ek[\ns{\gk}] \notag\\
	&=\Ek\left[\nsB{\frac{1}{m}\summ\Ci(\gi) -\frac{1}{m}\summ \gi + \frac{1}{m}\summ \gi}\right] \notag\\
	&\overset{\eqref{eq:compress}}{=} \Ek\left[\nsB{\frac{1}{m}\summ \left(\Ci(\gi) - \gi\right)}\right]  
	+\Ek\left[\nsB{\frac{1}{m}\summ \gi} \right] \notag
	\end{align}
	\begin{align}
	&\overset{\eqref{eq:compress}}{\leq} \frac{\omega}{m^2}  \EkB{\summ \ns{\gi}}
	+\Ek\left[\nsB{\frac{1}{m}\summ (\gi - \nabla f_i(\xk)) + \frac{1}{m}\summ \nabla f_i(\xk)} \right] \notag\\
	&=\frac{\omega}{m^2} \EkB{\summ \ns{\gi}}
	+\frac{1}{m^2} \EkB{\summ\ns{\gi - \nabla f_i(\xk)}}
	+ \ns{\nabla f(\xk)} \notag\\
	&=\frac{\omega}{m^2} \EkB{\summ \ns{\gi}}
	+\frac{1}{m^2} \summ \left(\Ek[\ns{\gi}] - \ns{\nabla f_i(\xk)}\right)
	+ \ns{\nabla f(\xk)} \notag\\
	&=
	\frac{1}{m^2} \summ \left((1+\omega)\Ek[\ns{\gi}] - \ns{\nabla f_i(\xk)}\right)
	+ \ns{\nabla f(\xk)} \notag\\
	&\overset{\eqref{eq:gi1-dc}}{\leq} \frac{1}{m^2} \summ \left((1+\omega)(2A_{1,i}(f_i(\xk)-f_i^*)+B_{1,i}\ns{\nabla f_i(\xk)} + D_{1}' \tsk +C_{1,i}) - \ns{\nabla f_i(\xk)}\right) \notag\\
	&\qquad \qquad
	+ \ns{\nabla f(\xk)} \notag\\
	&\leq \frac{1+\omega}{m^2} \summ \left(2\left((A_{1,i}+B_{1,i}L_i)-L_i/(1+\omega)\right)(f_i(\xk)-f_i^*) + D_{1}' \tsk +C_{1,i} \right) \notag\\
	&\qquad \qquad
	+ \ns{\nabla f(\xk)} \notag\\
	&\leq \frac{2(1+\omega)A}{m^2}\summ (f_i(\xk)-f_i^*)  + \frac{1+\omega}{m} D_{1}' \tsk + \frac{1+\omega}{m^2}\summ C_{1,i}
	+ \ns{\nabla f(\xk)} \label{eq:define-a-dc}\\
	&= \frac{2(1+\omega)A}{m}(f(\xk)-f^*)  
	+ \frac{1+\omega}{m}D_{1}'\tsk 
	+ \frac{(1+\omega)C}{m}
	+ \ns{\nabla f(\xk)}, \label{eq:define-fstar-dc}
	\end{align}
	where \eqref{eq:define-a-dc} holds by defining $A:=\max_i  (A_{1,i}+B_{1,i}L_i-L_i/(1+\omega))$, and \eqref{eq:define-fstar-dc} holds by defining 	$C := \frac{1}{m} \summ C_{1,i} + 2A\Delta_f^*$ and $\Delta_f^*:=f^*-\frac{1}{m}\summ f_i^*$.
	
	Thus, we have proved the first part, i.e., \eqref{eq:gk1-dc} holds with
	\begin{align}
	&A_1 =\frac{(1+\omega)A}{m}, \qquad
	B_1 =1, \qquad 
	C_1  =  \frac{(1+\omega)C}{m}, \\ 
	&D_1 =\frac{1+\omega}{m}, \qquad
	\sk =  D_{1}'\tsk. \label{eq:sk-dc}
	\end{align}
	Now we prove the second part $\Ek[\skn]$ (i.e., \eqref{eq:gk2-dc}). According to \eqref{eq:gi2-dc}, we have 
	\begin{align}
	&\Ek[\skn] \notag\\
	&\overset{\eqref{eq:sk-dc}}{:=} \Ek[D_{1}'\tskn] \notag\\
	&\overset{\eqref{eq:gi2-dc}}{\leq} (1-\rho')D_{1}'\tsk + 2D_{1}'A_{2}'(f(\xk)-f^*)+D_{1}'B_{2}'\ns{\nabla f(\xk)} + D_{1}'D_2'\Ek[\ns{\gk}] +D_{1}'C_2'\notag\\
	&\overset{\eqref{eq:define-fstar-dc}}{\leq} (1-\rho')\sk + 2D_{1}'A_{2}'(f(\xk)-f^*)+D_{1}'B_{2}'\ns{\nabla f(\xk)} +D_{1}'C_2' \notag\\
	&\qquad 
	+D_{1}'D_2'\left(\frac{2(1+\omega)A}{m}(f(\xk)-f^*)  
	+ \frac{1+\omega}{m}D_{1}'\tsk 
	+ \frac{(1+\omega)C}{m}
	+ \ns{\nabla f(\xk)}\right) \notag\\
	&= (1-\rho'+\tau) \sk + 2(D_{1}'A_{2}'+\tau A) (f(\xk)-f^*) \notag\\
	&\qquad  + (D_{1}'B_{2}'+D_{1}'D_2')\ns{\nabla f(\xk)} 
	+ D_{1}'C_2' + \tau C,
	\label{eq:term2-dc}
	\end{align}	
	where \eqref{eq:term2-dc} holds by defining $\tau:= \frac{(1+\omega)D_{1}'D_2'}{m}$.
	
	Now, we have proved the second part, i.e., \eqref{eq:gk2-dc} holds with
	\begin{align*}
	&\rho =\rho'-\tau, \qquad
	A_2 =D_{1}'A_{2}'+ \tau A, \qquad
	B_2 =D_{1}'B_{2}'+D_{1}'D_2', \qquad 
	C_2  = D_{1}'C_2' +  \tau C.
	\end{align*}
\end{proofof}

\begin{proofof}{Theorem \ref{thm:dc-diff} }
	Similar to the proof of Theorem \ref{thm:dc}, 
	we know that gradient estimator $\gk$ (see Line \ref{line:gk-dc} of Algorithm \ref{alg:dc}) is unbiased, i.e., 
	\begin{align}
	\Ek[\gk] 
	&=\Ek\left[\frac{1}{m}\summ \Ci(\gi) \right]  
	=\frac{1}{m}\summ \nabla f_i(\xk)=\nabla f(\xk). 
	\end{align}
	Then, we prove the upper bound for the second moment of gradient estimator $\gk$:
	\begin{align}
	&\Ek[\ns{\gk}] \notag\\
	&=\Ek\left[\nsB{\frac{1}{m}\summ\Ci(\gi) -\frac{1}{m}\summ \gi + \frac{1}{m}\summ \gi}\right] \notag\\
	&\overset{\eqref{eq:compress}}{=} \Ek\left[\nsB{\frac{1}{m}\summ \left(\Ci(\gi) - \gi\right)}\right]  
	+\Ek\left[\nsB{\frac{1}{m}\summ \gi} \right] \notag\\
	&\overset{\eqref{eq:compress}}{\leq} \frac{\omega}{m^2}  \EkB{\summ \ns{\gi}}
	+\Ek\left[\nsB{\frac{1}{m}\summ (\gi - \nabla f_i(\xk)) + \frac{1}{m}\summ \nabla f_i(\xk)} \right] \notag\\
	&=\frac{\omega}{m^2} \EkB{\summ \ns{\gi}}
	+\frac{1}{m^2} \EkB{\summ\ns{\gi - \nabla f_i(\xk)}}
	+ \ns{\nabla f(\xk)} \notag\\
	&=\frac{\omega}{m^2} \EkB{\summ \ns{\gi}}
	+\frac{1}{m^2} \summ \left(\Ek[\ns{\gi}] - \ns{\nabla f_i(\xk)}\right)
	+ \ns{\nabla f(\xk)} \notag\\
	&=
	\frac{1}{m^2} \summ \left((1+\omega)\Ek[\ns{\gi}] - \ns{\nabla f_i(\xk)}\right)
	+ \ns{\nabla f(\xk)} \notag\\
	&\overset{\eqref{eq:gi1-dc-diff}}{\leq} \frac{1}{m^2} \summ \left((1+\omega)(2A_{1,i}(f_i(\xk)-f_i^*)+B_{1,i}\ns{\nabla f_i(\xk)} + D_{1,i} \ski +C_{1,i}) - \ns{\nabla f_i(\xk)}\right) \notag\\
	&\qquad \qquad
	+ \ns{\nabla f(\xk)} \notag\\
	&\leq \frac{1+\omega}{m^2} \summ \left(2\left((A_{1,i}+B_{1,i}L_i)-L_i/(1+\omega)\right)(f_i(\xk)-f_i^*) + D_{1,i} \ski +C_{1,i} \right) \notag\\
	&\qquad \qquad
	+ \ns{\nabla f(\xk)} \notag\\
	&\leq \frac{2(1+\omega)A}{m^2}\summ (f_i(\xk)-f_i^*)  + \frac{1+\omega}{m^2}\summ  D_{1,i} \ski + \frac{1+\omega}{m^2}\summ C_{1,i}
	+ \ns{\nabla f(\xk)} \label{eq:define-a-dc-diff}\\
	&= \frac{2(1+\omega)A}{m}(f(\xk)-f^*)  
	+ \frac{1+\omega}{m^2}\summ  D_{1,i} \ski 
	+ \frac{(1+\omega)C}{m}
	+ \ns{\nabla f(\xk)}, \label{eq:define-fstar-dc-diff}
	\end{align}
	where \eqref{eq:define-a-dc-diff} holds by defining $A:=\max_i  (A_{1,i}+B_{1,i}L_i-L_i/(1+\omega))$, and \eqref{eq:define-fstar-dc-diff} holds by defining 	$C := \frac{1}{m} \summ C_{1,i} + 2A\Delta_f^*$ and $\Delta_f^*:=f^*-\frac{1}{m}\summ f_i^*$.
	
	Thus, we have proved the first part, i.e., \eqref{eq:gk1-dc-diff} holds with
	\begin{align}
	&A_1 =\frac{(1+\omega)A}{m}, \qquad
	B_1 =1, \qquad 
	C_1  =  \frac{(1+\omega)C}{m}, \\ 
	&D_1 =\frac{1+\omega}{m}, \qquad
	\sk =  \frac{1}{m}\summ  D_{1,i} \ski . \label{eq:sk-dc-diff}
	\end{align}
	Now we prove the second part $\Ek[\skn]$ (i.e., \eqref{eq:gk2-dc-diff}). According to \eqref{eq:gi2-dc-diff}, we have 
	\begin{align}
	&\Ek[\skn] \notag\\
	&\overset{\eqref{eq:sk-dc-diff}}{:=} \EkB{\frac{1}{m}\summ  D_{1,i} \skin} \notag\\
	&\overset{\eqref{eq:gi2-dc-diff}}{\leq} \frac{1}{m} \summ  D_{1,i} \left( (1-\rho_i)\sigma_{k,i}^2 + 2A_{2,i}(f(\xk)-f^*)+B_{2,i}\ns{\nabla f(\xk)} + D_{2,i}\Ek[\ns{\gk}] +C_{2,i}\right) \notag\\
	&~=\frac{1}{m} \summ (1-\rho_i) D_{1,i} \ski 
	+ 2D_A(f(\xk)-f^*)  
	+  D_B\ns{\nabla f(\xk)}
	+ D_D \Ek[\ns{\gk}]
	+D_C \label{eq:dd-dc-diff}\\
	&\overset{\eqref{eq:define-fstar-dc-diff}}{\leq}\frac{1}{m} \summ (1-\rho_i) D_{1,i} \ski 
	+ 2D_A(f(\xk)-f^*)  
	+  D_B\ns{\nabla f(\xk)}+D_C \notag\\
	&\qquad 
	+D_D\left(\frac{2(1+\omega)A}{m}(f(\xk)-f^*)  
	+ \frac{1+\omega}{m^2}\summ  D_{1,i} \ski 
	+ \frac{(1+\omega)C}{m}
	+ \ns{\nabla f(\xk)}\right) \notag\\
	&~=\frac{1}{m} \summ (1-\rho_i+\tau) D_{1,i} \ski 
	+ 2(D_A+\tau A)(f(\xk)-f^*)  \notag\\
	&\qquad  \qquad
	+  (D_B+D_D)\ns{\nabla f(\xk)}+D_C +\tau C \label{eq:tau-dc-diff}\\
	&~~\leq (1-\rho_i+\tau) \sk + 2(D_A+\tau A) (f(\xk)-f^*) + (D_B+D_D)\ns{\nabla f(\xk)} 
	+ D_C + \tau C,
	\label{eq:term2-dc-diff}
	\end{align}	
	where  \eqref{eq:dd-dc-diff} holds by defining 
	$D_A:=\frac{1}{m} \summ  D_{1,i}A_{2,i}$,
	$D_B:=\frac{1}{m} \summ  D_{1,i} B_{2,i}$,
	$D_D:=\frac{1}{m} \summ  D_{1,i} D_{2,i}$,
	and $D_C:=\frac{1}{m} \summ  D_{1,i}C_{2,i}$,
	\eqref{eq:tau-dc-diff} holds by defining $\tau:= \frac{(1+\omega)D_D}{m}$, 
	and the last inequality holds by defining $\rho :=\min_i\rho_i-\tau$.
	
	Now, we have proved the second part, i.e., \eqref{eq:gk2-dc-diff} holds with
	\begin{align*}
	&\rho =\min_i\rho_i-\tau, \qquad
	A_2 =D_A+ \tau A, \qquad
	B_2 =D_B+D_D, \qquad 
	C_2  = D_C+  \tau C.
	\end{align*}
\end{proofof}

In the following sections, we prove that if the parallel workers use some specific methods, i.e., GD, SGD, L-SVRG and SAGA, for computing their local stochastic gradient $\widetilde{g}_i^k$ (see Line \ref{line:localgrad-dc} of Algorithm \ref{alg:dc}), then $g^k$ (see Line \ref{line:gk-dc} of Algorithm \ref{alg:dc}) satisfies the unified Assumption \ref{asp:boge}.
Then, we plug their corresponding parameters (i.e., specific values for $A_1, A_2, B_1, B_2, C_1,C_2,D_1,\rho$) into our unified Theorem \ref{thm:main} to obtain the detailed convergence rates for these methods. 

\subsubsection{DC-GD method}

In this section, we show that if the parallel workers use GD for computing their local  gradient $\widetilde{g}_i^k$, then $g^k$ (see Line \ref{line:gk-dc-gd} of Algorithm \ref{alg:dc-gd}) satisfies the unified Assumption \ref{asp:boge}.

\begin{algorithm}[htb]
	\caption{DC-GD}
	\label{alg:dc-gd}
	\begin{algorithmic}[1]
		\REQUIRE ~
		initial point $x^0$,  stepsize $\eta_k$
		\FOR {$k=0,1,2,\ldots$}
		\STATE {\bf{for all machines $i= 1,2,\ldots,m$ do in parallel}}
		\STATE \quad Compute local gradient $\widetilde{g}_i^k = \nabla f_i(x^k)$ \label{line:localgrad-dc-gd}
		\STATE \quad Compress local gradient $\cC_i^k(\widetilde{g}_i^k)$ and send it to the server
		\STATE {\bf{end for}}
		\STATE Aggregate received compressed gradient information
		$g^k = \frac{1}{m}\sum \limits_{i=1}^m \cC_i^k(\widetilde{g}_i^k)$ \label{line:gk-dc-gd}
		\STATE $x^{k+1} = x^k - \eta_k g^k$  
		\ENDFOR
	\end{algorithmic}
\end{algorithm}

\begin{lemma}[DC-GD]\label{lem:dc-gd}
	Let the local gradient estimator $\widetilde{g}_i^k=\nabla f_i(x^k)$ (see Line \ref{line:localgrad-dc-gd} of Algorithm \ref{alg:dc-gd}), then we know that $\widetilde{g}_i^k$ satisfies \eqref{eq:gi1-dc} and \eqref{eq:gi2-dc} with $B_{1,i}=1, A_{1,i}=C_{1,i}=D_{1}'=0, \tsk \equiv 0, \rho'=1, A_2'=B_2'=C_2'=D_2'=0$. 
	Thus, according to Theorem \ref{thm:dc}, $g^k$ (see Line \ref{line:gk-dc-gd} of Algorithm \ref{alg:dc-gd})  satisfies the unified Assumption \ref{asp:boge} with
	\begin{align*}
	&A_1 =\frac{(1+\omega)A}{m}, \qquad
	B_1 =1, \qquad 
	C_1  =  \frac{(1+\omega)C}{m}, \\ 
	&D_1 =\frac{1+\omega}{m}, \qquad
	\sk \equiv 0, \qquad
	\rho =1, \\
	&A_2 =0, \qquad
	B_2 =0, \qquad 
	C_2  = 0,
	\end{align*}
	where 
	$A :=\max_i (L_i-L_i/(1+\omega))$,
	$C:= 2A\Delta_f^*$,
	and
	$\Delta_f^*:=f^*-\frac{1}{m}\summ f_i^*$.
\end{lemma}

\begin{proofof}{Lemma \ref{lem:dc-gd}}
	If the local gradient estimator $\gi=\nabla f_i(\xk)$ (see Line \ref{line:localgrad-dc-gd} of Algorithm \ref{alg:dc-gd}), it is easy to see that 
	$$\Ek[\gi] = \nabla f_i(\xk)$$ 
	and 
	$$\Ek[\ns{\gi}] \leq \ns{ \nabla f_i(\xk)}$$
	since there is no randomness. Thus the local gradient estimator $\gi$ satisfies \eqref{eq:gi1-dc} and \eqref{eq:gi2-dc} with 
	$$A_{1,i}=C_{1,i}=D_{1}'=0, B_{1,i}=1,\tsk \equiv 0, \rho'=1, A_2'=B_2'=C_2'=D_2'=0.$$ 
	Then, according to Theorem \ref{thm:dc}, $g^k$ (see Line \ref{line:gk-dc-gd} of Algorithm \ref{alg:dc-gd}) satisfies the unified Assumption \ref{asp:boge} with
	\begin{align*}
	&A_1 =\frac{(1+\omega)A}{m}, \qquad
	B_1 =1, \qquad 
	C_1  =  \frac{(1+\omega)C}{m}, \\ 
	&D_1 =\frac{1+\omega}{m}, \qquad
	\sk \equiv 0, \qquad
	\rho =1, \\
	&A_2 =0, \qquad
	B_2 =0, \qquad 
	C_2  = 0,
	\end{align*}
	where 
	$A :=\max_i (L_i-L_i/(1+\omega))$,
	$C:= 2A\Delta_f^*$,
	and
	$\Delta_f^*:=f^*-\frac{1}{m}\summ f_i^*$.
\end{proofof}

\begin{corollary}[DC-GD]\label{cor:dc-gd}
	Suppose that Assumption \ref{asp:lsmooth-diana} holds. 
	Let stepsize 
	$$ \eta \leq \min\left\{ \frac{1}{L},~ \sqrt{\frac{m\ln 2}{(1+\omega)LAK}},~ \frac{m\epsilon^2}{2(1+\omega)LC} \right\},$$ 
	then the number of iterations performed by DC-GD (Algorithm \ref{alg:dc-gd}) to find an $\epsilon$-solution of nonconvex federated problem \eqref{eq:prob-fed}, i.e. a point $\hx$ such that $\E[\n{\nabla f(\hx)}] \leq \epsilon$, can be bounded by
	\begin{align}
	K = \frac{8\fgap L}{\epsilon^2} \max \left\{1,~ \frac{12(1+\omega)\fgap A }{\epsilon^2m},~ \frac{2(1+\omega)C}{\epsilon^2m}\right\},
	\end{align}
	where 
	$A:=\max_i (L_i-L_i/(1+\omega))$,
	$C:= 2A\Delta_f^*$,
	and
	$\Delta_f^*:=f^*-\frac{1}{m}\summ f_i^*$.
\end{corollary}

\begin{proofof}{Corollary \ref{cor:dc-gd}}
	According to our unified Theorem \ref{thm:main}, if the stepsize is chosen as 
	\begin{align}
	\eta &\leq \min\left\{ \frac{1}{LB_1+LD_1B_2\rho^{-1}}, \sqrt{\frac{\ln 2}{(LA_1 + LD_1A_2\rho^{-1})K}}, \frac{\epsilon^2}{2L(C_1+D_1C_2\rho^{-1})} \right\}  \notag\\
	&= \min\left\{ \frac{1}{L},~ \sqrt{\frac{m\ln 2}{(1+\omega)LAK}},~ \frac{m\epsilon^2}{2(1+\omega)LC} \right\}
	\end{align}
	since $B_2=A_2=C_2=0, B_1=1, A_1 =\frac{(1+\omega)A}{m}$ and $C_1  =  \frac{(1+\omega)C}{m}$
	according to Lemma \ref{lem:dc-gd},
	then the number of iterations performed by DC-GD (Algorithm \ref{alg:dc-gd}) to find an $\epsilon$-solution of problem \eqref{eq:prob-fed} can be bounded by
	\begin{align}
	K &= \frac{8\fgapp L}{\epsilon^2} \max \left\{ B_1+D_1B_2\rho^{-1}, \frac{12\fgapp (A_1+D_1 A_2\rho^{-1})}{\epsilon^2}, \frac{2(C_1+D_1C_2\rho^{-1})}{\epsilon^2}\right\}  \notag\\
	&= \frac{8\fgap L}{\epsilon^2} \max \left\{1,~ \frac{12(1+\omega)\fgap A }{\epsilon^2m},~ \frac{2(1+\omega)C}{\epsilon^2m}\right\}
	\end{align}
	since $B_2=A_2=C_2=0, B_1=1, A_1 =\frac{(1+\omega)A}{m}, C_1  =  \frac{(1+\omega)C}{m}, \sigma_0^2 = 0$, and 
	$$\fgapp:= f(x^0) - f^* + 2^{-1}L\eta^2D_1\rho^{-1} \sigma_0^2 = f(x^0) - f^* =\fgap.$$
\end{proofof}

\subsubsection{DC-SGD method}

In this section, we show that if the parallel workers use SGD for computing their local  stochastic gradient $\widetilde{g}_i^k$, then $g^k$ (see Line \ref{line:gk-dc-sgd} of Algorithm \ref{alg:dc-sgd}) satisfies the unified Assumption \ref{asp:boge}.

\begin{algorithm}[h]
	\caption{DC-SGD}
	\label{alg:dc-sgd}
	\begin{algorithmic}[1]
		\REQUIRE ~
		initial point $x^0$,  stepsize $\eta_k$
		\FOR {$k=0,1,2,\ldots$}
		\STATE {\bf{for all machines $i= 1,2,\ldots,m$ do in parallel}}
		\STATE \quad Compute local stochastic gradient $\widetilde{g}_i^k$ with Algorithm \ref{alg:sgd} by changing $f(x)$ to the local $f_i(x)$ \label{line:localgrad-dc-sgd}
		\STATE \quad Compress local gradient $\cC_i^k(\widetilde{g}_i^k)$ and send it to the server
		\STATE {\bf{end for}}
		\STATE Aggregate received compressed gradient information
		$g^k = \frac{1}{m}\sum \limits_{i=1}^m \cC_i^k(\widetilde{g}_i^k)$ \label{line:gk-dc-sgd}
		\STATE $x^{k+1} = x^k - \eta_k g^k$  
		\ENDFOR
	\end{algorithmic}
\end{algorithm}

\begin{lemma}[DC-SGD]\label{lem:dc-sgd}
	Let the local stochastic gradient estimator $\widetilde{g}_i^k$ (see Line \ref{line:localgrad-dc-sgd} of Algorithm \ref{alg:dc-sgd}) satisfy Assumption \ref{asp:es}, i.e.,
	$$ \E_k[\ns{\widetilde{g}_i^k}] \leq 2A_{1,i}(f_i(x^k)-f_i^*)+B_{1,i}\ns{\nabla f_i(x^k)} + C_{1,i},$$
	then we know that $\widetilde{g}_i^k$ satisfies \eqref{eq:gi1-dc} and \eqref{eq:gi2-dc} with $D_{1}'=0, \tsk \equiv 0, \rho'=1, A_2'=B_2'=C_2'=D_2'=0$. 
	Thus, according to Theorem \ref{thm:dc}, $g^k$ (see Line \ref{line:gk-dc-sgd} of Algorithm \ref{alg:dc-sgd})  satisfies the unified Assumption \ref{asp:boge} with
	\begin{align*}
	&A_1 =\frac{(1+\omega)A}{m}, \qquad
	B_1 =1, \qquad 
	C_1  =  \frac{(1+\omega)C}{m}, \\ 
	&D_1 =\frac{1+\omega}{m}, \qquad
	\sk \equiv 0, \qquad
	\rho =1, \\
	&A_2 =0, \qquad
	B_2 =0, \qquad 
	C_2  = 0,
	\end{align*}
	where 
	$A:=\max_i  (A_{1,i}+B_{1,i}L_i-L_i/(1+\omega))$,
	$C := \frac{1}{m} \summ C_{1,i} + 2A\Delta_f^*$
	and
	$\Delta_f^*:=f^*-\frac{1}{m}\summ f_i^*$.
\end{lemma}

\begin{proofof}{Lemma \ref{lem:dc-sgd}}
	Suppose that the local stochastic gradient estimator $\widetilde{g}_i^k$ (see Line \ref{line:localgrad-dc-sgd} of Algorithm \ref{alg:dc-sgd}) satisfies Assumption \ref{asp:es}, i.e., 
	$$\Ek[\gi] = \nabla f_i(\xk)$$
	and 
	$$ \Ek[\ns{\gi}] \leq 2A_{1,i}(f_i(x^k)-f_i^*)+B_{1,i}\ns{\nabla f_i(x^k)} + C_{1,i},$$
	Thus the local stochastic gradient estimator $\gi$ satisfies \eqref{eq:gi1-dc} and \eqref{eq:gi2-dc} with 
	$$A_{1,i}=A_{1,i}, B_{1,i}=B_{1,i},C_{1,i}=C_{1,i}, D_{1}'=0,\tsk \equiv 0, \rho'=1, A_2'=B_2'=C_2'=D_2'=0.$$ 
	Then, according to Theorem \ref{thm:dc}, $g^k$ (see Line \ref{line:gk-dc-sgd} of Algorithm \ref{alg:dc-sgd}) satisfies the unified Assumption \ref{asp:boge} with
	\begin{align*}
	&A_1 =\frac{(1+\omega)A}{m}, \qquad
	B_1 =1, \qquad 
	C_1  =  \frac{(1+\omega)C}{m}, \\ 
	&D_1 =\frac{1+\omega}{m}, \qquad
	\sk \equiv 0, \qquad
	\rho =1, \\
	&A_2 =0, \qquad
	B_2 =0, \qquad 
	C_2  = 0,
	\end{align*}
	where 
	$A:=\max_i  (A_{1,i}+B_{1,i}L_i-L_i/(1+\omega))$,
	$C := \frac{1}{m} \summ C_{1,i} + 2A\Delta_f^*$
	and
	$\Delta_f^*:=f^*-\frac{1}{m}\summ f_i^*$.
\end{proofof}

\begin{corollary}[DC-SGD]\label{cor:dc-sgd}
	Suppose that Assumption \ref{asp:lsmooth-diana} holds. 
	Let stepsize 
	$$ \eta \leq \min\left\{ \frac{1}{L},~ \sqrt{\frac{m\ln 2}{(1+\omega)LAK}},~ \frac{m\epsilon^2}{2(1+\omega)LC} \right\},$$ 
	then the number of iterations performed by DC-SGD (Algorithm \ref{alg:dc-sgd}) to find an $\epsilon$-solution of nonconvex federated problem \eqref{eq:prob-fed} with \eqref{prob-fed:exp} or \eqref{prob-fed:finite}, i.e. a point $\hx$ such that $\E[\n{\nabla f(\hx)}] \leq \epsilon$, can be bounded by
	\begin{align}
	K = \frac{8\fgap L}{\epsilon^2} \max \left\{1,~ \frac{12(1+\omega)\fgap A }{\epsilon^2m},~ \frac{2(1+\omega)C}{\epsilon^2m}\right\},
	\end{align}
	where 
	where 
	$A:=\max_i  (A_{1,i}+B_{1,i}L_i-L_i/(1+\omega))$,
	$C := \frac{1}{m} \summ C_{1,i} + 2A\Delta_f^*$
	and
	$\Delta_f^*:=f^*-\frac{1}{m}\summ f_i^*$.
\end{corollary}

\begin{proofof}{Corollary \ref{cor:dc-sgd}}
	According to our unified Theorem \ref{thm:main}, if the stepsize is chosen as 
	\begin{align}
	\eta &\leq \min\left\{ \frac{1}{LB_1+LD_1B_2\rho^{-1}}, \sqrt{\frac{\ln 2}{(LA_1 + LD_1A_2\rho^{-1})K}}, \frac{\epsilon^2}{2L(C_1+D_1C_2\rho^{-1})} \right\}  \notag\\
	&= \min\left\{ \frac{1}{L},~ \sqrt{\frac{m\ln 2}{(1+\omega)LAK}},~ \frac{m\epsilon^2}{2(1+\omega)LC} \right\}
	\end{align}
	since $B_2=A_2=C_2=0, B_1=1, A_1 =\frac{(1+\omega)A}{m}$ and $C_1  =  \frac{(1+\omega)C}{m}$
	according to Lemma \ref{lem:dc-sgd},
	then the number of iterations performed by DC-SGD (Algorithm \ref{alg:dc-sgd}) to find an $\epsilon$-solution of problem \eqref{eq:prob-fed} with \eqref{prob-fed:exp} or \eqref{prob-fed:finite} can be bounded by
	\begin{align}
	K &= \frac{8\fgapp L}{\epsilon^2} \max \left\{ B_1+D_1B_2\rho^{-1}, \frac{12\fgapp (A_1+D_1 A_2\rho^{-1})}{\epsilon^2}, \frac{2(C_1+D_1C_2\rho^{-1})}{\epsilon^2}\right\}  \notag\\
	&= \frac{8\fgap L}{\epsilon^2} \max \left\{1,~ \frac{12(1+\omega)\fgap A }{\epsilon^2m},~ \frac{2(1+\omega)C}{\epsilon^2m}\right\}
	\end{align}
	since $B_2=A_2=C_2=0, B_1=1, A_1 =\frac{(1+\omega)A}{m}, C_1  =  \frac{(1+\omega)C}{m}, \sigma_0^2 = 0$, and 
	$$\fgapp:= f(x^0) - f^* + 2^{-1}L\eta^2D_1\rho^{-1} \sigma_0^2 = f(x^0) - f^* =\fgap.$$
\end{proofof}

\newpage
\subsubsection{DC-LSVRG method}

In this section, we show that if the parallel workers use L-SVRG for computing their local  gradient $\widetilde{g}_i^k$, then $g^k$ (see Line \ref{line:gk-dc-lsvrg} of Algorithm \ref{alg:dc-lsvrg}) satisfies the unified Assumption \ref{asp:boge}.

\begin{algorithm}[htb]
	\caption{DC-LSVRG}
	\label{alg:dc-lsvrg}
	\begin{algorithmic}[1]
		\REQUIRE ~
		initial point $x^0=w^0$, stepsize $\eta_k$, minibatch size $b$, probability $p\in (0,1]$
		\FOR {$k=0,1,2,\ldots$}
		\STATE {\bf{for all machines $i= 1,2,\ldots,m$ do in parallel}}
		\STATE \quad Compute local stochastic gradient $\widetilde{g}_i^k=\frac{1}{b} \sum_{j\in I_b} (\nabla f_{i,j}(x^k)- \nabla f_{i,j}(w^k)) +\nabla f_i(w^k)$  \label{line:localgrad-dc-lsvrg}
		\STATE \quad Compress local gradient $\cC_i^k(\widetilde{g}_i^k)$ and send it to the server
		\STATE \quad $w^{k+1} = \begin{cases}
		x^k &\text {with probability } p\\
		w^k &\text {with probability } 1-p
		\end{cases}$ \label{line:w_prob-dc-lsvrg}
		\STATE {\bf{end for}}
		\STATE Aggregate received compressed gradient information
		$g^k = \frac{1}{m}\sum \limits_{i=1}^m \cC_i^k(\widetilde{g}_i^k)$ \label{line:gk-dc-lsvrg}
		\STATE $x^{k+1} = x^k - \eta_k g^k$  \label{line:update-dc-lsvrg}
		\ENDFOR
	\end{algorithmic}
\end{algorithm}

\begin{lemma}[DC-LSVRG]\label{lem:dc-lsvrg}
	Let the local stochastic gradient estimator $\widetilde{g}_i^k=\frac{1}{b} \sum_{j\in I_b} (\nabla f_{i,j}(x^k)- \nabla f_{i,j}(w^k)) +\nabla f_i(w^k)$ (see Line \ref{line:localgrad-dc-lsvrg} of Algorithm \ref{alg:dc-lsvrg}), then we know that $\widetilde{g}_i^k$ satisfies \eqref{eq:gi1-dc} and \eqref{eq:gi2-dc} with $A_{1,i}=0, B_{1,i}=1, C_{1,i} =0, D_{1}'=\frac{\bar{L}^2}{b}, \tsk = \ns{\xk-\wk}$,
	$\rho'=p+p\gamma-\gamma, A_2'=0, B_2'=(1-p)\eta^2\gamma^{-1},C_2'=0, D_2'=\eta^2$, and $\forall\gamma>0$.
	Thus, according to Theorem \ref{thm:dc}, $g^k$ (see Line \ref{line:gk-dc-lsvrg} of Algorithm \ref{alg:dc-lsvrg})  satisfies the unified Assumption \ref{asp:boge} with
	\begin{align*}
	&A_1 =\frac{(1+\omega)A}{m}, \qquad
	B_1 =1, \qquad 
	C_1  =  \frac{(1+\omega)C}{m}, \\ 
	&D_1 =\frac{1+\omega}{m}, \qquad
	\sk =\frac{\bar{L}^2}{b}\ns{\xk-\wk}, \qquad
	\rho =p+p\gamma-\gamma-\tau, \\
	&A_2 =\tau A, \qquad
	B_2 = \frac{\bar{L}^2 ((1-p)\eta^2\gamma^{-1} + \eta^2)}{b}, \qquad 
	C_2  = \tau C,
	\end{align*}
	where 
	$A:=\max_i (L_i-L_i/(1+\omega))$,
	$C:= 2A\Delta_f^*$,
	$\Delta_f^*:=f^*-\frac{1}{m}\summ f_i^*$,
	$\tau:= \frac{(1+\omega)\bar{L}^2 \eta^2}{mb}$
	and $\forall\gamma>0$.
\end{lemma}

\begin{proofof}{Lemma \ref{lem:dc-lsvrg}}
	If the local stochastic gradient estimator $\gi=\frac{1}{b} \sum_{j\in I_b} (\nabla f_{i,j}(x^k)- \nabla f_{i,j}(w^k)) +\nabla f_i(w^k)$ (see Line \ref{line:localgrad-dc-lsvrg} of Algorithm \ref{alg:dc-lsvrg}), we show the following equations:
	\begin{align}
	\Ek[\gi] &= \Ek\left[\frac{1}{b} \sum_{j\in I_b} (\nabla f_{i,j}(x^k)- \nabla f_{i,j}(w^k)) +\nabla f_i(w^k)\right] \notag\\
	&= \nabla f_i(\xk) - \nabla f_i(\wk) + \nabla f_i(\wk) = \nabla f_i(\xk) \label{eq:dc-svrg0}
	\end{align}
	and 
	\begin{align}
	\Ek[\ns{\gi}] & = \Ek\left[\ns{\gi-\nabla f_i(\xk)}\right] + \ns{\nabla f_i(\xk) } \notag\\
	&= \Ek\left[\nsB{\frac{1}{b} \sum_{j\in I_b} (\nabla f_{i,j}(x^k)- \nabla f_{i,j}(w^k)) +\nabla f_i(w^k) - \nabla f_i(\xk)} \right]  + \ns{\nabla f_i(\xk) }  \notag
	\end{align}
	\begin{align}
	&= \frac{1}{b^2} \Ek\left[\nsB{\sum_{j\in I_b} \left((\nabla f_{i,j}(\xk)- \nabla f_{i,j}(\wk)) -(\nabla f_i(\xk) - \nabla f_i(\wk)) \right)} \right]  + \ns{\nabla f_i(\xk) }  \notag\\
	&= \frac{1}{b} \Ek\left[\nsB{(\nabla f_{i,j}(\xk)- \nabla f_{i,j}(\wk)) -(\nabla f_i(\xk) - \nabla f_i(\wk))} \right]  + \ns{\nabla f_i(\xk) }  \notag\\
	&\leq \frac{1}{b} \Ek\left[\nsB{\nabla f_{i,j}(\xk)- \nabla f_{i,j}(\wk)} \right]  + \ns{\nabla f_i(\xk) }  \label{eq:dc-var}\\
	& \leq  \frac{\bL^2}{b} \ns{\xk-\wk} + \ns{\nabla f_i(\xk)}, \label{eq:dc-svrg1}
	\end{align}
	where \eqref{eq:dc-var} uses the fact $\E[\ns{x-\E[x]}]\leq \E[\ns{x}]$, and the last inequality uses Assumption \ref{asp:avgsmooth-fed} (i.e., \eqref{eq:avgsmooth-fed}).
	Now, we define $\tsk:=\ns{\xk-\wk}$ and obtain
	\begin{align}
	&\Ek[\tskn] \notag\\
	& := \Ek[\ns{\xkn-\wkn}]  \notag\\
	&= p\Ek[\ns{\xkn-\xk}] + (1-p)\Ek[\ns{\xkn-\wk}]  \label{eq:dc-useprob}\\
	&=p\eta^2\Ek[\ns{\gk}] + (1-p)\Ek[\ns{\xk-\eta \gk-\wk}]   \label{eq:dc-useupdate} \\
	&=p\eta^2\Ek[\ns{\gk}] + (1-p)\Ek[\ns{\xk-\wk} + \ns{\eta \gk} -2\inner{\xk-\wk}{\eta\gk}]  \notag\\
	&= p\eta^2\Ek[\ns{\gk}] + (1-p)\ns{\xk-\wk} + (1-p)\eta^2\Ek[\ns{\gk}] -2(1-p)\inner{\xk-\wk}{\eta \nabla f(\xk)} \notag \\
	&\leq \eta^2\Ek[\ns{\gk}] + (1-p)\ns{\xk-\wk} +(1-p)\gamma\ns{\xk-\wk} + \frac{1-p}{\gamma}\ns{\eta \nabla f(\xk)} \label{eq:dc-useyoung} \\
	&= \eta^2\Ek[\ns{\gk}] + (1-p)(1+\gamma)\ns{\xk-\wk}  
	+ \frac{(1-p)\eta^2}{\gamma}\ns{\nabla f(\xk)} \notag\\
	&= (1-(p+p\gamma-\gamma)) \tsk + \frac{(1-p)\eta^2}{\gamma}\ns{\nabla f(\xk)}
	+ \eta^2\Ek[\ns{\gk}], 
	\label{eq:dc-svrg2}
	\end{align}
	where \eqref{eq:dc-useprob} uses Line \ref{line:w_prob-dc-lsvrg} of Algorithm \ref{alg:dc-lsvrg}, \eqref{eq:dc-useupdate} uses Line \ref{line:update-dc-lsvrg} of Algorithm \ref{alg:dc-lsvrg}, \eqref{eq:dc-useyoung} uses Young's inequality for $\forall \gamma>0$, and \eqref{eq:dc-svrg2} follows from the definition $\tsk:=\ns{\xk-\wk}$.
	
	Now, according to \eqref{eq:dc-svrg1} and \eqref{eq:dc-svrg2}, we know the local stochastic gradient estimator $\gi$ (see Line \ref{line:localgrad-dc-lsvrg} of Algorithm \ref{alg:dc-lsvrg}) satisfies \eqref{eq:gi1-dc} and \eqref{eq:gi2-dc} with 
	\begin{align*}
	&A_{1,i}=0,\quad  B_{1,i}=1,\quad 
	D_{1}'=\frac{\bar{L}^2}{b},\quad \tsk = \ns{\xk-\wk},\quad
	C_{1,i} =0, \\
	&\rho'=p+p\gamma-\gamma,\quad
	A_2'=0,\quad
	B_2'=(1-p)\eta^2\gamma^{-1},\quad
	D_2'=\eta^2,  \quad
	C_2'=0,\quad \forall\gamma>0.
	\end{align*}
	Thus, according to Theorem \ref{thm:dc}, $g^k$ (see Line \ref{line:gk-dc-lsvrg} of Algorithm \ref{alg:dc-lsvrg})  satisfies the unified Assumption \ref{asp:boge} with
	\begin{align*}
	&A_1 =\frac{(1+\omega)A}{m}, \qquad
	B_1 =1, \qquad 
	C_1  =  \frac{(1+\omega)C}{m}, \\ 
	&D_1 =\frac{1+\omega}{m}, \qquad
	\sk =\frac{\bar{L}^2}{b}\ns{\xk-\wk}, \qquad
	\rho =p+p\gamma-\gamma-\tau, \\
	&A_2 =\tau A, \qquad
	B_2 = \frac{\bar{L}^2 ((1-p)\eta^2\gamma^{-1} + \eta^2)}{b}, \qquad 
	C_2  = \tau C,
	\end{align*}
	where 
	$A:=\max_i (L_i-L_i/(1+\omega))$,
	$C:= 2A\Delta_f^*$,
	$\Delta_f^*:=f^*-\frac{1}{m}\summ f_i^*$,
	$\tau:= \frac{(1+\omega)\bar{L}^2 \eta^2}{mb}$
	and $\forall\gamma>0$.
\end{proofof}

\begin{corollary}[DC-LSVRG]\label{cor:dc-lsvrg}
	Suppose that Assumption \ref{asp:lsmooth-diana} and \ref{asp:avgsmooth-fed} hold. 
	Let stepsize 
	$$\eta \leq \min\left\{ \frac{1}{L+L(1+\omega)Bm^{-1}b^{-1}\rho^{-1}},~ \sqrt{\frac{m\ln 2}{(1+\omega)( 1+\tau\rho^{-1})LAK}},~ \frac{m\epsilon^2}{2(1+\omega)( 1+\tau\rho^{-1})LC} \right\},$$ 
	then the number of iterations performed by DC-LSVRG (Algorithm \ref{alg:dc-lsvrg}) to find an $\epsilon$-solution of nonconvex federated problem \eqref{eq:prob-fed} with \eqref{prob-fed:finite}, i.e. a point $\hx$ such that $\E[\n{\nabla f(\hx)}] \leq \epsilon$, can be bounded by	
	\begin{align}
	K = \frac{8\fgap L}{\epsilon^2} \max \left\{ 1+\frac{(1+\omega)B}{mb\rho},~ \frac{12(1+\omega)( 1+\tau\rho^{-1})\fgap A}{m\epsilon^2},~ \frac{2(1+\omega)( 1+\tau\rho^{-1})LC}{m\epsilon^2}\right\},
	\end{align}
	where 
	$A:=\max_i (L_i-L_i/(1+\omega))$,
	$B: = \bar{L}^2 ((1-p)\eta^2\gamma^{-1} + \eta^2)$,
	$C:= 2A\Delta_f^*$,
	$\Delta_f^*:=f^*-\frac{1}{m}\summ f_i^*$,
	$\rho :=p+p\gamma-\gamma-\tau$,
	$\tau:= \frac{(1+\omega)\bar{L}^2 \eta^2}{mb}$
	and $\forall\gamma>0$.
\end{corollary}

\begin{proofof}{Corollary \ref{cor:dc-lsvrg}}
	According to our unified Theorem \ref{thm:main}, if the stepsize is chosen as 
	\begin{align*}
	\eta &\leq \min\left\{ \frac{1}{LB_1+LD_1B_2\rho^{-1}}, \sqrt{\frac{\ln 2}{(LA_1 + LD_1A_2\rho^{-1})K}}, \frac{\epsilon^2}{2L(C_1+D_1C_2\rho^{-1})} \right\}  \notag\\
	&= \min\left\{\frac{1}{L+L(1+\omega)Bm^{-1}b^{-1}\rho^{-1}},~ \sqrt{\frac{m\ln 2}{(1+\omega)( 1+\tau\rho^{-1})LAK}},~ \frac{m\epsilon^2}{2(1+\omega)( 1+\tau\rho^{-1})LC} \right\}
	\end{align*}
	since $B_1 =1, D_1 =\frac{1+\omega}{m}, B_2=\frac{B}{b}, B=\bar{L}^2 ((1-p)\eta^2\gamma^{-1} + \eta^2), A_1 =\frac{(1+\omega)A}{m},
	A_2 =\tau A, C_1  =  \frac{(1+\omega)C}{m}$ and $C_2  = \tau C$ 
	according to Lemma \ref{lem:dc-lsvrg},
	then the number of iterations performed by DC-LSVRG (Algorithm \ref{alg:dc-lsvrg}) to find an $\epsilon$-solution of problem \eqref{eq:prob-fed} with \eqref{prob-fed:finite} can be bounded by
	\begin{align}
	K &= \frac{8\fgapp L}{\epsilon^2} \max \left\{ B_1+D_1B_2\rho^{-1}, \frac{12\fgapp (A_1+D_1 A_2\rho^{-1})}{\epsilon^2}, \frac{2(C_1+D_1C_2\rho^{-1})}{\epsilon^2}\right\}  \notag\\
	&= \frac{8\fgap L}{\epsilon^2} \max \left\{1+\frac{(1+\omega)B}{mb\rho},~ \frac{12(1+\omega)( 1+\tau\rho^{-1})\fgap A}{m\epsilon^2},~ \frac{2(1+\omega)( 1+\tau\rho^{-1})LC}{m\epsilon^2}\right\}
	\end{align}
	since $B_1 =1, D_1 =\frac{1+\omega}{m}, B_2=\frac{B}{b}, B=\bar{L}^2 ((1-p)\eta^2\gamma^{-1} + \eta^2), A_1 =\frac{(1+\omega)A}{m},
	A_2 =\tau A, C_1  =  \frac{(1+\omega)C}{m}, C_2  = \tau C$, and 
	$\fgapp:= f(x^0) - f^* + 2^{-1}L\eta^2D_1\rho^{-1} \sigma_0^2 = f(x^0) - f^* =\fgap$ due to $\sigma_0^2 = \frac{\bar{L}^2}{b}\ns{x^0-w^0}=0$.
\end{proofof}

\subsubsection{DC-SAGA method}
In this section, we show that if the parallel workers use SAGA for computing their local  gradient $\widetilde{g}_i^k$, then $g^k$ (see Line \ref{line:gk-dc-saga} of Algorithm \ref{alg:dc-saga}) satisfies the unified Assumption \ref{asp:boge}.

\begin{algorithm}[htb]
	\caption{DC-SAGA}
	\label{alg:dc-saga}
	\begin{algorithmic}[1]
		\REQUIRE ~
		initial point $x^0, \{w_i^0\}_{i=1}^n$, stepsize $\eta_k$, minibatch size $b$
		\FOR {$k=0,1,2,\ldots$}
		\STATE {\bf{for all machines $i= 1,2,\ldots,m$ do in parallel}}
		\STATE \quad Compute local stochastic gradient $\widetilde{g}_i^k=\frac{1}{b} \sum_{j\in I_b} (\nabla f_{i,j}(x^k)- \nabla f_{i,j}(w_{i,j}^k)) + \frac{1}{n}\sum_{j=1}^{n}\nabla f_{i,j}(w_{i,j}^k)$  \label{line:localgrad-dc-saga}
		\STATE \quad Compress local gradient $\cC_i^k(\widetilde{g}_i^k)$ and send it to the server
		\STATE \quad $w_{i,j}^{k+1} = \begin{cases}
		x^k & \text{for~~}  j\in I_b\\
		w_{i,j}^k &\text{for~~} j\notin I_b
		\end{cases}$ \label{line:w_prob-dc-saga}
		\STATE {\bf{end for}}
		\STATE Aggregate received compressed gradient information
		$g^k = \frac{1}{m}\sum \limits_{i=1}^m \cC_i^k(\widetilde{g}_i^k)$ \label{line:gk-dc-saga}
		\STATE $x^{k+1} = x^k - \eta_k g^k$  \label{line:update-dc-saga}
		\ENDFOR
	\end{algorithmic}
\end{algorithm}

\begin{lemma}[DC-SAGA]\label{lem:dc-saga}
	Let the local stochastic gradient estimator $\widetilde{g}_i^k=\frac{1}{b} \sum_{j\in I_b} (\nabla f_{i,j}(x^k)- \nabla f_{i,j}(w_{i,j}^k)) + \frac{1}{n}\sum_{j=1}^{n}\nabla f_{i,j}(w_{i,j}^k)$ (see Line \ref{line:localgrad-dc-saga} of Algorithm \ref{alg:dc-saga}), then we know that $\widetilde{g}_i^k$ satisfies \eqref{eq:gi1-dc-diff} and \eqref{eq:gi2-dc-diff} with $A_{1,i}=0, B_{1,i}=1, C_{1,i} =0, D_{1,i}=\frac{\bar{L}^2}{b}, \ski = \frac{1}{n}\sum_{j=1}^{n}\ns{\xk-w_{i,j}^k}$,
	$\rho_i=\frac{b}{n}+\frac{b}{n}\gamma-\gamma, A_{2,i}=0, B_{2,i}=(1-\frac{b}{n})\eta^2\gamma^{-1},C_{2,i}=0, D_{2,i}=\eta^2$, and $\forall\gamma>0$.
	Thus, according to Theorem \ref{thm:dc-diff}, $g^k$ (see Line \ref{line:gk-dc-saga} of Algorithm \ref{alg:dc-saga})  satisfies the unified Assumption \ref{asp:boge} with
	\begin{align*}
	&A_1 =\frac{(1+\omega)A}{m}, \qquad
	B_1 =1, \qquad 
	C_1  =  \frac{(1+\omega)C}{m}, \\ 
	&D_1 =\frac{1+\omega}{m}, \qquad
	\sk =\frac{1}{mn}\summ \sum_{j=1}^{n}\frac{\bar{L}^2}{b}\ns{\xk-w_{i,j}^k}, \qquad
	\rho =\frac{b}{n}+\frac{b}{n}\gamma-\gamma-\tau, \\
	&A_2 =\tau A, \qquad
	B_2 = \frac{\bar{L}^2 ((1-\frac{b}{n})\eta^2\gamma^{-1} + \eta^2)}{b}, \qquad 
	C_2  = \tau C,
	\end{align*}
	where 
	$A:=\max_i (L_i-L_i/(1+\omega))$,
	$C:= 2A\Delta_f^*$,
	$\Delta_f^*:=f^*-\frac{1}{m}\summ f_i^*$,
	$\tau:= \frac{(1+\omega)\bar{L}^2 \eta^2}{mb}$
	and $\forall\gamma>0$.
\end{lemma}

\begin{proofof}{Lemma \ref{lem:dc-saga}}
	If the local stochastic gradient estimator $\gi=\frac{1}{b} \sum_{j\in I_b} (\nabla f_{i,j}(x^k)- \nabla f_{i,j}(w_{i,j}^k)) + \frac{1}{n}\sum_{j=1}^{n}\nabla f_{i,j}(w_{i,j}^k)$(see Line \ref{line:localgrad-dc-saga} of Algorithm \ref{alg:dc-saga}), we show the following equations:
	\begin{align}
	\Ek[\gi] &= \Ek\left[\frac{1}{b} \sum_{j\in I_b} (\nabla f_{i,j}(x^k)- \nabla f_{i,j}(w_{i,j}^k)) + \frac{1}{n}\sum_{j=1}^{n}\nabla f_{i,j}(w_{i,j}^k)\right] \notag\\
	&= \nabla f_i(\xk) - \frac{1}{n}\sum_{j=1}^{n}\nabla f_{i,j}(w_{i,j}^k) + \frac{1}{n}\sum_{j=1}^{n}\nabla f_{i,j}(w_{i,j}^k) \notag\\
	&= \nabla f_i(\xk) \label{eq:dc-saga0}
	\end{align}
	and 
	\begin{align}
	&\Ek[\ns{\gi}] \notag\\
	& = \Ek\left[\ns{\gi-\nabla f_i(\xk)}\right] + \ns{\nabla f_i(\xk) } \notag\\
	&= \Ek\left[\nsB{\frac{1}{b} \sum_{j\in I_b} (\nabla f_{i,j}(x^k)- \nabla f_{i,j}(w_{i,j}^k)) + \frac{1}{n}\sum_{j=1}^{n}\nabla f_{i,j}(w_{i,j}^k) - \nabla f_i(\xk)} \right]  + \ns{\nabla f_i(\xk) }  \notag\\
	&= \frac{1}{b^2} \Ek\left[\nsB{\sum_{j\in I_b} \left((\nabla f_{i,j}(\xk)- \nabla f_{i,j}(w_{i,j}^k)) -\left(\frac{1}{n}\sum_{j=1}^{n}\nabla f_{i,j}(\xk)- \frac{1}{n}\sum_{j=1}^{n}\nabla f_{i,j}(w_{i,j}^k)\right) \right)} \right]  \notag\\
	&\qquad \quad + \ns{\nabla f_i(\xk) }  \notag\\
	&= \frac{1}{b} \Ek\left[\nsB{(\nabla f_{i,j}(\xk)- \nabla f_{i,j}(w_{i,j}^k)) -\left(\frac{1}{n}\sum_{j=1}^{n}\nabla f_{i,j}(\xk)- \frac{1}{n}\sum_{j=1}^{n}\nabla f_{i,j}(w_{i,j}^k)\right)} \right]  \notag\\
	&\qquad \quad + \ns{\nabla f_i(\xk) }  \notag\\
	&\leq \frac{1}{b} \Ek\left[\nsB{\nabla f_{i,j}(\xk)- \nabla f_{i,j}(w_{i,j}^k)} \right]  + \ns{\nabla f_i(\xk) }  \label{eq:dc-var1}\\
	& \leq  \frac{\bL^2}{b}\frac{1}{n}\sum_{j=1}^{n}\ns{\xk-w_{i,j}^k} + \ns{\nabla f_i(\xk)}, \label{eq:dc-saga1}
	\end{align}
	where \eqref{eq:dc-var1} uses the fact $\E[\ns{x-\E[x]}]\leq \E[\ns{x}]$, and the last inequality uses Assumption \ref{asp:avgsmooth-saga-fed} (i.e., \eqref{eq:avgsmooth-saga-fed}).
	Now, we define $\ski:=\frac{1}{n}\sum_{j=1}^{n}\ns{\xk-w_{i,j}^k}$ and obtain
	\begin{align}
	&\Ek[\skin] \notag\\
	& := \EkB{\frac{1}{n}\sum_{j=1}^{n}\ns{\xkn-\wijkn}}  \notag\\
	&= \EkB{\frac{1}{n}\sum_{j=1}^{n}\frac{b}{n}\ns{\xkn-\xk}+ \frac{1}{n}\sum_{j=1}^{n}\left(1-\frac{b}{n}\right)\ns{\xkn-\wijk}} \label{eq:dc-useprob1}\\
	&= \frac{b}{n}\Ek\eta^2\ns{\gk} 
	+ \left(1-\frac{b}{n}\right)\EkB{\frac{1}{n}\sum_{j=1}^{n} \ns{\xk-\eta\gk-\wijk} }  \label{eq:dc-useupdate1} \\
	&= \frac{b\eta^2}{n}\Ek\ns{\gk} 
	+ \left(1-\frac{b}{n}\right)\EkB{\frac{1}{n}\sum_{j=1}^{n} \left(\ns{\xk-\wijk} + \ns{\eta \gk} -2\inner{\xk-\wijk}{\eta\gk}\right)}  \notag\\
	&= \eta^2\Ek\ns{\gk} 
	+ \left(1-\frac{b}{n}\right)\frac{1}{n}\sum_{j=1}^{n} \ns{\xk-\wijk} 
	+ 2\left(1-\frac{b}{n}\right)\frac{1}{n}\sum_{j=1}^{n}\inner{\xk-\wijk}{\eta \nabla f(\xk)}  \notag\\
	&\leq \eta^2\Ek\ns{\gk} 
	+ \left(1-\frac{b}{n}\right)\frac{1}{n}\sum_{j=1}^{n} \ns{\xk-\wijk} \notag\\
	&\qquad \qquad 
	+ \left(1-\frac{b}{n}\right)\frac{1}{n}\sum_{j=1}^{n}\left(\gamma\ns{\xk-\wijk} +\frac{\eta^2}{\gamma}\ns{\nabla f(\xk)}\right) \label{eq:dc-useyoung1} \\
	&= \left(1-(\frac{b}{n}+\frac{b}{n}\gamma-\gamma)\right) \ski + \frac{(1-\frac{b}{n})\eta^2}{\gamma}\ns{\nabla f(\xk)}
	+ \eta^2\Ek[\ns{\gk}], 
	\label{eq:dc-saga2}
	\end{align}
	where \eqref{eq:dc-useprob1} uses Line \ref{line:w_prob-dc-saga} of Algorithm \ref{alg:dc-saga}, \eqref{eq:dc-useupdate1} uses Line \ref{line:update-dc-saga} of Algorithm \ref{alg:dc-saga}, \eqref{eq:dc-useyoung1} uses Young's inequality for $\forall \gamma>0$, and \eqref{eq:dc-saga2} follows from the definition $\ski:=\frac{1}{n}\sum_{j=1}^{n}\ns{\xk-w_{i,j}^k}$.
	
	Now, according to \eqref{eq:dc-saga1} and \eqref{eq:dc-saga2}, we know the local stochastic gradient estimator $\gi$ (see Line \ref{line:localgrad-dc-saga} of Algorithm \ref{alg:dc-saga}) satisfies \eqref{eq:gi1-dc-diff} and \eqref{eq:gi2-dc-diff} with 
	\begin{align*}
	&A_{1,i}=0,\quad  B_{1,i}=1,\quad 
	D_{1,i}=\frac{\bar{L}^2}{b},\quad 
	\ski = \frac{1}{n}\sum_{j=1}^{n}\ns{\xk-w_{i,j}^k},\quad
	C_{1,i} =0, \\
	&\rho_i=\frac{b}{n}+\frac{b}{n}\gamma-\gamma,\quad
	A_{2,i}=0,\quad
	B_{2,i}=(1-\frac{b}{n})\eta^2\gamma^{-1},\quad
	D_{2,i}=\eta^2,  \quad
	C_{2,i}=0,\quad \forall\gamma>0.
	\end{align*}
	Thus, according to Theorem \ref{thm:dc-diff}, $g^k$ (see Line \ref{line:gk-dc-saga} of Algorithm \ref{alg:dc-saga}) satisfies the unified Assumption \ref{asp:boge} with
	\begin{align*}
	&A_1 =\frac{(1+\omega)A}{m}, \qquad
	B_1 =1, \qquad 
	C_1  =  \frac{(1+\omega)C}{m}, \\ 
	&D_1 =\frac{1+\omega}{m}, \qquad
	\sk =\frac{1}{mn}\summ \sum_{j=1}^{n}\frac{\bar{L}^2}{b}\ns{\xk-w_{i,j}^k}, \qquad
	\rho =\frac{b}{n}+\frac{b}{n}\gamma-\gamma - \tau, \\
	&A_2 =\tau A, \qquad
	B_2 = \frac{\bar{L}^2 ((1-\frac{b}{n})\eta^2\gamma^{-1} + \eta^2)}{b}, \qquad 
	C_2  = \tau C,
	\end{align*}
	where 
	$A:=\max_i (L_i-L_i/(1+\omega))$,
	$C:= 2A\Delta_f^*$,
	$\Delta_f^*:=f^*-\frac{1}{m}\summ f_i^*$,
	$\tau:= \frac{(1+\omega)\bar{L}^2 \eta^2}{mb}$
	and $\forall\gamma>0$.
\end{proofof}

\begin{corollary}[DC-SAGA]\label{cor:dc-saga}
	Suppose that Assumption \ref{asp:lsmooth-diana} and \ref{asp:avgsmooth-saga-fed} hold. 
	Let stepsize 
	$$\eta \leq \min\left\{ \frac{1}{L+L(1+\omega)Bm^{-1}b^{-1}\rho^{-1}},~ \sqrt{\frac{m\ln 2}{(1+\omega)( 1+\tau\rho^{-1})LAK}},~ \frac{m\epsilon^2}{2(1+\omega)( 1+\tau\rho^{-1})LC} \right\},$$ 
	then the number of iterations performed by DC-SAGA (Algorithm \ref{alg:dc-saga}) to find an $\epsilon$-solution of nonconvex federated problem \eqref{eq:prob-fed} with \eqref{prob-fed:finite}, i.e. a point $\hx$ such that $\E[\n{\nabla f(\hx)}] \leq \epsilon$, can be bounded by	
	\begin{align}
	K = \frac{8\fgap L}{\epsilon^2} \max \left\{ 1+\frac{(1+\omega)B}{mb\rho},~ \frac{12(1+\omega)( 1+\tau\rho^{-1})\fgap A}{m\epsilon^2},~ \frac{2(1+\omega)( 1+\tau\rho^{-1})LC}{m\epsilon^2}\right\},
	\end{align}
	where 
	$A:=\max_i (L_i-L_i/(1+\omega))$,
	$B: = \bar{L}^2 ((1-\frac{b}{n})\eta^2\gamma^{-1} + \eta^2)$,
	$C:= 2A\Delta_f^*$,
	$\Delta_f^*:=f^*-\frac{1}{m}\summ f_i^*$,
	$\rho :=\frac{b}{n}+\frac{b}{n}\gamma-\gamma$,
	$\tau:= \frac{(1+\omega)\bar{L}^2 \eta^2}{mb}$
	and $\forall\gamma>0$.
\end{corollary}

\begin{proofof}{Corollary \ref{cor:dc-saga}}
	According to our unified Theorem \ref{thm:main}, if the stepsize is chosen as 
	\begin{align*}
	\eta &\leq \min\left\{ \frac{1}{LB_1+LD_1B_2\rho^{-1}}, \sqrt{\frac{\ln 2}{(LA_1 + LD_1A_2\rho^{-1})K}}, \frac{\epsilon^2}{2L(C_1+D_1C_2\rho^{-1})} \right\}  \notag\\
	&= \min\left\{\frac{1}{L+L(1+\omega)Bm^{-1}b^{-1}\rho^{-1}},~ \sqrt{\frac{m\ln 2}{(1+\omega)( 1+\tau\rho^{-1})LAK}},~ \frac{m\epsilon^2}{2(1+\omega)( 1+\tau\rho^{-1})LC} \right\}
	\end{align*}
	since $B_1 =1, D_1 =\frac{1+\omega}{m}, B_2=\frac{B}{b}, B=\bar{L}^2 ((1-\frac{b}{n})\eta^2\gamma^{-1} + \eta^2), A_1 =\frac{(1+\omega)A}{m},
	A_2 =\tau A, C_1  =  \frac{(1+\omega)C}{m}$ and $C_2  = \tau C$ 
	according to Lemma \ref{lem:dc-saga},
	then the number of iterations performed by DC-LSVRG (Algorithm \ref{alg:dc-saga}) to find an $\epsilon$-solution of problem \eqref{eq:prob-fed} with \eqref{prob-fed:finite} can be bounded by
	\begin{align}
	K &= \frac{8\fgapp L}{\epsilon^2} \max \left\{ B_1+D_1B_2\rho^{-1}, \frac{12\fgapp (A_1+D_1 A_2\rho^{-1})}{\epsilon^2}, \frac{2(C_1+D_1C_2\rho^{-1})}{\epsilon^2}\right\}  \notag\\
	&= \frac{8\fgap L}{\epsilon^2} \max \left\{1+\frac{(1+\omega)B}{mb\rho},~ \frac{12(1+\omega)( 1+\tau\rho^{-1})\fgap A}{m\epsilon^2},~ \frac{2(1+\omega)( 1+\tau\rho^{-1})LC}{m\epsilon^2}\right\}
	\end{align}
	since $B_1 =1, D_1 =\frac{1+\omega}{m}, B_2=\frac{B}{b}, B=\bar{L}^2 ((1-\frac{b}{n})\eta^2\gamma^{-1} + \eta^2), A_1 =\frac{(1+\omega)A}{m},
	A_2 =\tau A, C_1  =  \frac{(1+\omega)C}{m}, C_2  = \tau C$, and 
	$\fgapp:= f(x^0) - f^* + 2^{-1}L\eta^2D_1\rho^{-1} \sigma_0^2 = f(x^0) - f^* =\fgap$ by letting $\sigma_0^2 = 0$.
\end{proofof}

\subsection{DIANA framework for nonconvex federated optimization}
\label{sec:diana-app}

Similar to Section \ref{sec:dc-app},
we first prove a general Theorem for DIANA framework (Algorithm \ref{alg:diana})  which shows that several (new) methods belonging to the general DIANA framework satisfy Assumption \ref{asp:boge} and thus can be captured by our unified analysis.
Then, we plug their corresponding parameters into our unified Theorem \ref{thm:main} to obtain the detailed convergence rates for these methods. 

Before proving the Theorem \ref{thm:diana-type-diff}, we first provide a simple version as in Theorem \ref{thm:diana-type} where all workers share the same variance term $\tsk$ (see \eqref{eq:gi1}).
If the parallel workers use GD, SGD or L-SVRG for computing their local stochastic gradient $\widetilde{g}_i^k$ (see Line \ref{line:localgrad} of Algorithm \ref{alg:diana}), then they indeed share the same variance term $\tsk$, i.e., Theorem \ref{thm:diana-type} includes these settings.
However, if the parallel workers use SAGA-type methods for computing their local stochastic gradient $\widetilde{g}_i^k$, then the variance term $\ski$ (see \eqref{eq:gi1-diff}) is different for different worker $i$, i.e., the more general Theorem \ref{thm:diana-type-diff} includes this SAGA setting while Theorem \ref{thm:diana-type} does not.

\begin{theorem}[DIANA framework with same variance for all workers]\label{thm:diana-type} 
	Suppose that the local stochastic gradient $\widetilde{g}_i^k$ (see Line \ref{line:localgrad} of Algorithm \ref{alg:diana}) satisfies 
	\begin{align}
	\E_k[\ns{\widetilde{g}_i^k}] & \leq 2A_{1,i}(f_i(x^k)-f_i^*)+B_{1,i}\ns{\nabla f_i(x^k)} + \blue{D_{1}' \widetilde{\sigma}_{k}^2} +C_{1,i},  \label{eq:gi1}\\
	\E_k[\widetilde{\sigma}_{k+1}^2] & \leq \blue{(1-\rho')\widetilde{\sigma}_{k}^2 + 2A_{2}'(f(x^k)-f^*)+B_{2}'\ns{\nabla f(x^k)} + D_2'\E_k[\ns{g^k}] +C_2'}, \label{eq:gi2}
	\end{align}
	then $g^k$ (see Line \ref{line:gk} of Algorithm \ref{alg:diana}) satisfies the unified Assumption \ref{asp:boge}, i.e., 
	\begin{align}
	\E_k[\ns{g^k}] & \leq 2A_1(f(x^k)-f^*)+B_1\ns{\nabla f(x^k)} + D_1 \sigma_k^2 +C_1, \label{eq:gk1}\\
	\E_k[\sigma_{k+1}^2] & \leq (1-\rho)\sigma_k^2 + 2A_2(f(x^k)-f^*)+B_2\ns{\nabla f(x^k)} +C_2, \label{eq:gk2}
	\end{align}
	with parameters
	\begin{align*}
	&A_1 =\frac{(1+\omega)A}{m}, \qquad
	B_1 =1, \qquad 
	C_1  =  \frac{(1+\omega)C}{m}, \\ 
	&D_1 =\frac{1+\omega}{m}, \qquad
	\sk =  D_{1}'\tsk + \frac{\omega}{(1+\omega)m}
	\summ \ns{\nabla f_i(\xk) -\hi}, \\
	&\rho =\min\{\rho'-\tau,~ 2\alpha -(1-\alpha)\beta^{-1} -\alpha^2-\tau \},\\
	&A_2 =D_{1}'A_{2}'+ \tau A, \qquad
	B_2 =D_{1}'B_{2}'+B, \qquad 
	C_2  = D_{1}'C_2' +  \tau C,
	\end{align*}
	where 
	$A:=\max_i  (A_{1,i}+(B_{1,i}-1)L_i)$, 
	$B:=\frac{\omega(1+\beta)L^2\eta^2}{1+\omega}+ D_{1}'D_2'$, 
	$C := \frac{1}{m} \summ C_{1,i} + 2A\Delta_f^*$,
	$\Delta_f^*:=f^*-\frac{1}{m}\summ f_i^*$,
	$\tau:=\alpha^2 \omega + \frac{(1+\omega)B}{m}$,
	and $\forall \beta>0$.
\end{theorem}

Before providing the proof for Theorem \ref{thm:diana-type}, we recall the more general Theorem \ref{thm:diana-type-diff} here for better comparison. Then we provide the detailed proofs for Theorems \ref{thm:diana-type} and \ref{thm:diana-type-diff}.

\begingroup
\def\thetheorem{\ref{thm:diana-type-diff}}
\begin{theorem}[DIANA framework with different variance for different worker]
	Suppose that the local stochastic gradient $\widetilde{g}_i^k$ (see Line \ref{line:localgrad-dc} of Algorithm \ref{alg:dc}) satisfies 
	\begin{align}
	\E_k[\ns{\widetilde{g}_i^k}] & \leq 2A_{1,i}(f_i(x^k)-f_i^*)+B_{1,i}\ns{\nabla f_i(x^k)} + \blue{D_{1,i} \sigma_{k,i}^2} +C_{1,i}, \label{eq:gi1-diff}\\
	\E_k[\sigma_{k+1,i}^2] & \leq \blue{(1-\rho_i)\sigma_{k,i}^2 + 2A_{2,i}(f(x^k)-f^*)+B_{2,i}\ns{\nabla f(x^k)} + D_{2,i}\E_k[\ns{g^k}] +C_{2,i}}, \label{eq:gi2-diff}
	\end{align}
	then $g^k$ (see Line \ref{line:gk} of Algorithm \ref{alg:diana})  also satisfies the unified Assumption \ref{asp:boge}, i.e., 
	\begin{align}
	\E_k[\ns{g^k}] & \leq 2A_1(f(x^k)-f^*)+B_1\ns{\nabla f(x^k)} + D_1 \sigma_k^2 +C_1, \label{eq:gk1-diff}\\
	\E_k[\sigma_{k+1}^2] & \leq (1-\rho)\sigma_k^2 + 2A_2(f(x^k)-f^*)+B_2\ns{\nabla f(x^k)} +C_2, \label{eq:gk2-diff}
	\end{align}
	with parameters
	\begin{align*}
	&A_1 =\frac{(1+\omega)A}{m}, \qquad
	B_1 =1, \qquad 
	C_1  =  \frac{(1+\omega)C}{m}, \\ 
	&D_1 =\frac{1+\omega}{m}, \qquad
	\sk =  \frac{1}{m} \summ  D_{1,i} \ski  + \frac{\omega}{(1+\omega)m}
	\summ \ns{\nabla f_i(\xk) -\hi}, \\
	&\rho =\min\{\min_i\rho_i-\tau,~ 2\alpha -(1-\alpha)\beta^{-1} -\alpha^2-\tau \}, \\
	&A_2 =D_A+ \tau A, \qquad
	B_2 =D_B+B, \qquad 
	C_2  = D_C +  \tau C,
	\end{align*}
	where 
	$A:=\max_i  (A_{1,i}+(B_{1,i}-1)L_i)$, 
	$B:=\frac{\omega(1+\beta)L^2\eta^2}{1+\omega} + D_D$, 
	$C := \frac{1}{m} \summ C_{1,i} + 2A\Delta_f^*$,
	$\Delta_f^*:=f^*-\frac{1}{m}\summ f_i^*$,
	$\tau:=\alpha^2 \omega + \frac{(1+\omega)B}{m}$,
	$D_A:=\frac{1}{m} \summ  D_{1,i}A_{2,i}$,
	$D_B:=\frac{1}{m} \summ  D_{1,i} B_{2,i}$,
	$D_D:=\frac{1}{m} \summ  D_{1,i} D_{2,i}$,
	$D_C:=\frac{1}{m} \summ  D_{1,i}C_{2,i}$,
	and $\forall \beta>0$.
\end{theorem}
\addtocounter{theorem}{-1}
\endgroup

\begin{proofof}{Theorem \ref{thm:diana-type}}
	First, we show the that gradient estimator $\gk$ (see Line \ref{line:gk} of Algorithm \ref{alg:diana}) is unbiased:
	\begin{align}
	\Ek[\gk] & = \Ek\left[h^k + \frac{1}{m}\summ \Di \right] \notag\\
	&=\Ek\left[\frac{1}{m}\summ \hi + \frac{1}{m}\summ \Ci(\gi-\hi) \right] \notag\\
	&= \Ek\left[\frac{1}{m}\summ \hi + \frac{1}{m}\summ (\gi-\hi) \right]  \notag\\
	&= \Ek\left[\frac{1}{m}\summ \gi \right]  =\nabla f(\xk) 
	\end{align}
	Then, we prove the upper bound for the second moment of gradient estimator $\gk$:
	\begin{align}
	&\Ek[\ns{\gk}] \notag\\
	&=\Ek\left[\nsB{\frac{1}{m}\summ \hi + \frac{1}{m}\summ \Ci(\gi-\hi) } \right] \notag\\
	&=\Ek\left[\nsB{\frac{1}{m}\summ \hi + \frac{1}{m}\summ \Ci(\gi-\hi) -\frac{1}{m}\summ \gi + \frac{1}{m}\summ \gi} \right] \notag\\
	&\overset{\eqref{eq:compress}}{=} \Ek\left[\nsB{\frac{1}{m}\summ \left(\Ci(\gi-\hi) - \gi +\hi \right)}\right]  
	+\Ek\left[\nsB{\frac{1}{m}\summ \gi} \right] \notag\\
	&\overset{\eqref{eq:compress}}{\leq} \frac{\omega}{m^2}  \EkB{\summ \ns{\gi-\hi}}
	+\Ek\left[\nsB{\frac{1}{m}\summ (\gi - \nabla f_i(\xk)) + \frac{1}{m}\summ \nabla f_i(\xk)} \right] \notag\\
	&=\frac{\omega}{m^2}  \EkB{\summ \ns{\gi-\hi}}
	+\frac{1}{m^2} \EkB{\summ\ns{\gi - \nabla f_i(\xk)}}
	+ \ns{\nabla f(\xk)} \notag\\
	&=\frac{\omega}{m^2}  \EkB{\summ \ns{\gi-\nabla f_i(\xk) + \nabla f_i(\xk) -\hi}}  +\frac{1}{m^2} \EkB{\summ\ns{\gi - \nabla f_i(\xk)}}
	+ \ns{\nabla f(\xk)} \notag\\
	&= \frac{1+\omega}{m^2} \EkB{\summ\ns{\gi - \nabla f_i(\xk)}}
	+ \frac{\omega}{m^2} \summ \ns{\nabla f_i(\xk) -\hi} 
	+ \ns{\nabla f(\xk)} \notag\\
	&=\frac{1+\omega}{m^2} \summ \left(\Ek[\ns{\gi}] - \ns{\nabla f_i(\xk)}\right)
	+ \frac{\omega}{m^2} \summ \ns{\nabla f_i(\xk) -\hi} 
	+ \ns{\nabla f(\xk)} \notag\\
	&\overset{\eqref{eq:gi1}}{\leq} \frac{1+\omega}{m^2} \summ \left(2A_{1,i}(f_i(\xk)-f_i^*)+B_{1,i}\ns{\nabla f_i(\xk)} + D_{1}' \tsk +C_{1,i} - \ns{\nabla f_i(\xk)}\right) \notag\\
	&\qquad \qquad
	+ \frac{\omega}{m^2} \summ \ns{\nabla f_i(\xk) -\hi} 
	+ \ns{\nabla f(\xk)} \notag
	\end{align}
	\begin{align}
	&\leq \frac{1+\omega}{m^2} \summ \left(2(A_{1,i}+(B_{1,i}-1)L_i)(f_i(\xk)-f_i^*) + D_{1}' \tsk +C_{1,i} \right) \notag\\
	&\qquad \qquad
	+ \frac{\omega}{m^2} \summ \ns{\nabla f_i(\xk) -\hi} 
	+ \ns{\nabla f(\xk)} \notag\\
	&\leq \frac{2(1+\omega)A}{m^2}\summ (f_i(\xk)-f_i^*)  + \frac{1+\omega}{m} D_{1}' \tsk + \frac{1+\omega}{m^2}\summ C_{1,i}\notag\\
	&\qquad \qquad
	+ \frac{\omega}{m^2} \summ \ns{\nabla f_i(\xk) -\hi} 
	+ \ns{\nabla f(\xk)} \label{eq:define-a}\\
	&= \frac{2(1+\omega)A}{m}(f(\xk)-f^*)  
	+ \frac{1+\omega}{m}  \left(D_{1}'\tsk + \frac{\omega}{(1+\omega)m}
	\summ \ns{\nabla f_i(\xk) -\hi}\right) 
	+ \ns{\nabla f(\xk)} \notag\\
	&\qquad \qquad	+ \frac{(1+\omega)C}{m}, \label{eq:define-fstar}
	\end{align}
	where \eqref{eq:define-a} holds by defining $A:=\max_i  (A_{1,i}+(B_{1,i}-1)L_i)$, and \eqref{eq:define-fstar} holds by defining 	$C := \frac{1}{m} \summ C_{1,i} + 2A\Delta_f^*$ and $\Delta_f^*:=f^*-\frac{1}{m}\summ f_i^*$.
	
	Thus, we have proved the first part, i.e., \eqref{eq:gk1} holds with
	\begin{align}
	&A_1 =\frac{(1+\omega)A}{m}, \qquad
	B_1 =1, \qquad 
	C_1  =  \frac{(1+\omega)C}{m}, \\ 
	&D_1 =\frac{1+\omega}{m}, \qquad
	\sk =  D_{1}'\tsk + \frac{\omega}{(1+\omega)m}
	\summ \ns{\nabla f_i(\xk) -\hi}.
	\end{align}
	Now we prove the second part (i.e., \eqref{eq:gk2}). According to \eqref{eq:gi2}, we have 
	\begin{align}
	\Ek[D_{1}'\tskn] & \leq (1-\rho')D_{1}'\tsk + 2D_{1}'A_{2}'(f(\xk)-f^*)+D_{1}'B_{2}'\ns{\nabla f(\xk)} + D_{1}'D_2'\Ek[\ns{\gk}] +D_{1}'C_2'. \label{eq:term1}
	\end{align}	
	
	For the second term of $\tskn$, we bound it as follows:
	\begin{align}
	&\EkB{\frac{\omega}{(1+\omega)m}
		\summ \ns{\nabla f_i(\xkn) -\hin}}  \notag\\
	& = \EkB{\frac{\omega}{(1+\omega)m}
		\summ \ns{\nabla f_i(\xkn) -\nabla f_i(\xk) + \nabla f_i(\xk) -\hin}} \notag\\
	& = \EkB{\frac{\omega}{(1+\omega)m}
		\summ \ns{\nabla f_i(\xkn) -\nabla f_i(\xk) + \nabla f_i(\xk) -\hi - \alpha \Ci(\gi-\hi)}} \notag\\
	& = \frac{\omega}{(1+\omega)m} \summ \Ek\Big[
	\ns{\nabla f_i(\xkn) -\nabla f_i(\xk)} 
	+\ns{\nabla f_i(\xk) -\hi - \alpha \Ci(\gi-\hi)} \notag\\
	& \qquad \qquad 
	+2\inner{\nabla f_i(\xkn) -\nabla f_i(\xk)}{\nabla f_i(\xk) -\hi - \alpha \Ci(\gi-\hi)}\Big] \notag\\
	& = \frac{\omega}{(1+\omega)m} \summ \Ek\Big[
	\ns{\nabla f_i(\xkn) -\nabla f_i(\xk)} 
	+\ns{\nabla f_i(\xk) -\hi - \alpha \Ci(\gi-\hi)} \notag\\
	& \qquad \qquad 
	+2\inner{\nabla f_i(\xkn) -\nabla f_i(\xk)}{(1-\alpha)(\nabla f_i(\xk) -\hi)}\Big] \notag
	\end{align}
	\begin{align}
	& = \frac{\omega}{(1+\omega)m} \summ \Ek\Big[
	\ns{\nabla f_i(\xkn) -\nabla f_i(\xk)} 
	+(1-2\alpha)\ns{\nabla f_i(\xk) -\hi}
	+ \alpha^2 \ns{\Ci(\gi-\hi)} \notag\\
	& \qquad \qquad 
	+2\inner{\nabla f_i(\xkn) -\nabla f_i(\xk)}{(1-\alpha)(\nabla f_i(\xk) -\hi)}\Big] \notag\\
	& \leq \frac{\omega}{(1+\omega)m} \summ \Ek\Big[
	\ns{\nabla f_i(\xkn) -\nabla f_i(\xk)} 
	+(1-2\alpha)\ns{\nabla f_i(\xk) -\hi}
	+ \alpha^2 \ns{\Ci(\gi-\hi)} \notag\\
	& \qquad \qquad 
	+\beta\ns{\nabla f_i(\xkn) -\nabla f_i(\xk)} 
	+\frac{(1-\alpha)^2}{\beta}\ns{\nabla f_i(\xk) -\hi}\Big] \qquad \forall \beta>0\notag\\
	& = \frac{\omega}{(1+\omega)m} \summ \Ek\Big[
	(1+\beta)\ns{\nabla f_i(\xkn) -\nabla f_i(\xk)} 
	+(1-2\alpha +\frac{(1-\alpha)^2}{\beta})\ns{\nabla f_i(\xk) -\hi}
	\notag\\
	& \qquad \qquad 
	+ \alpha^2 \ns{\Ci(\gi-\hi)}\Big]  \notag\\
	& \overset{\eqref{eq:compress}}{\leq}  \frac{\omega}{(1+\omega)m} \summ \Ek\Big[
	(1+\beta)\ns{\nabla f_i(\xkn) -\nabla f_i(\xk)} 
	+(1-2\alpha +\frac{(1-\alpha)^2}{\beta})\ns{\nabla f_i(\xk) -\hi}
	\notag\\
	& \qquad \qquad 
	+ \alpha^2 (1+\omega)\ns{\gi-\hi}\Big]  \notag\\
	& = \frac{\omega}{(1+\omega)m} \summ \Ek\Big[
	(1+\beta)\ns{\nabla f_i(\xkn) -\nabla f_i(\xk)} 
	+(1-2\alpha +\frac{(1-\alpha)^2}{\beta})\ns{\nabla f_i(\xk) -\hi}
	\notag\\
	& \qquad \qquad 
	+ \alpha^2 (1+\omega)\ns{\gi-\nabla f_i(\xk) + \nabla f_i(\xk) -\hi}\Big]  \notag\\
	& = \left(1-2\alpha +\frac{(1-\alpha)^2}{\beta}+\alpha^2 (1+\omega)\right) \frac{\omega}{(1+\omega)m} \summ \ns{\nabla f_i(\xk) -\hi}   \notag\\
	&\qquad \qquad 
	+ \frac{\omega(1+\beta)}{(1+\omega)m} \summ 
	\Ek[\ns{\nabla f_i(\xkn) 	-\nabla f_i(\xk)} ]
	+ \frac{\alpha^2 \omega}{m} \summ 
	\Ek[\ns{\gi-\nabla f_i(\xk)}]  \notag\\
	& \leq \left(1-2\alpha +\frac{(1-\alpha)^2}{\beta}+\alpha^2 (1+\omega)\right) \frac{\omega}{(1+\omega)m} \summ \ns{\nabla f_i(\xk) -\hi}   \notag\\
	&\qquad \qquad 
	+ \frac{\omega(1+\beta)}{(1+\omega)m} \summ 
	\Ek[\ns{\nabla f_i(\xkn) 	-\nabla f_i(\xk)} ]
	+ \alpha^2 \omega \left(2A (f(\xk)-f^*) +D_{1}' \tsk +C\right)\notag\\
	& \leq \left(1-2\alpha +\frac{(1-\alpha)^2}{\beta}+\alpha^2 (1+\omega)\right) \frac{\omega}{(1+\omega)m} \summ \ns{\nabla f_i(\xk) -\hi}   \notag\\
	&\qquad \qquad 
	+ \frac{\omega(1+\beta)}{(1+\omega)m} \summ 
	L_i^2\Ek[\ns{\xkn -\xk} ]
	+ \alpha^2 \omega \left(2A (f(\xk)-f^*) +D_{1}' \tsk +C\right)  \notag\\
	& = \left(1-2\alpha +\frac{(1-\alpha)^2}{\beta}+\alpha^2 (1+\omega)\right) \frac{\omega}{(1+\omega)m} \summ \ns{\nabla f_i(\xk) -\hi}   \notag\\
	&\qquad \qquad 
	+ \frac{\omega(1+\beta)}{(1+\omega)m} \summ 
	L_i^2\eta^2\Ek[\ns{\gk} ]
	+ \alpha^2 \omega \left(2A (f(\xk)-f^*) +D_{1}' \tsk +C\right)  \label{eq:term2}
	\end{align}	
	By combining \eqref{eq:term1} and \eqref{eq:term2}, we have 
	\begin{align}
	&\Ek[\skn]  \notag\\
	&=\EkB{D_{1}'\tskn + \frac{\omega}{(1+\omega)m}
		\summ \ns{\nabla f_i(\xkn) -\hin}}  \notag
	\end{align}
	\begin{align}
	&\leq (1-\rho')D_{1}'\tsk + 2D_{1}'A_{2}'(f(\xk)-f^*)+D_{1}'B_{2}'\ns{\nabla f(\xk)} + D_{1}'D_2'\Ek[\ns{\gk}] +D_{1}'C_2'  \notag\\
	&\quad +\left(1-2\alpha +\frac{(1-\alpha)^2}{\beta}+\alpha^2 (1+\omega)\right) \frac{\omega}{(1+\omega)m} \summ \ns{\nabla f_i(\xk) -\hi}   \notag\\
	&\qquad \qquad 
	+ \frac{\omega(1+\beta)}{(1+\omega)m} \summ 
	L_i^2\eta^2\Ek[\ns{\gk} ]
	+ \alpha^2 \omega \left(2A (f(\xk)-f^*) +D_{1}' \tsk +C\right)  \notag\\
	&= (1-\rho' +\alpha^2 \omega)D_{1}'\tsk + 2(D_{1}'A_{2}'+\alpha^2 \omega A)(f(\xk)-f^*)+D_{1}'B_{2}'\ns{\nabla f(\xk)} +D_{1}'C_2'  \notag\\
	&\quad + \alpha^2 \omega C +\left(1-2\alpha +\frac{(1-\alpha)^2}{\beta}+\alpha^2 (1+\omega)\right) \frac{\omega}{(1+\omega)m} \summ \ns{\nabla f_i(\xk) -\hi}   \notag\\
	&\qquad \qquad 
	+ B\Ek[\ns{\gk}] \label{eq:define-B}\\
	&\leq (1-\rho' +\alpha^2 \omega)D_{1}'\tsk + 2(D_{1}'A_{2}'+\alpha^2 \omega A)(f(\xk)-f^*)+D_{1}'B_{2}'\ns{\nabla f(\xk)} +D_{1}'C_2'   \notag\\
	&\quad + \alpha^2 \omega C +\left(1-2\alpha +\frac{(1-\alpha)^2}{\beta}+\alpha^2 (1+\omega)\right) \frac{\omega}{(1+\omega)m} \summ \ns{\nabla f_i(\xk) -\hi}   \notag\\
	&\qquad  
	+ \frac{2(1+\omega)BA}{m}(f(\xk)-f^*)  
	+ \frac{(1+\omega)B}{m}  \left(D_{1}'\tsk + \frac{\omega}{(1+\omega)m}
	\summ \ns{\nabla f_i(\xk) -\hi}\right)  \notag\\
	&\qquad 		+ B\ns{\nabla f(\xk)} + \frac{(1+\omega)B}{m}C \label{eq:pluggk}\\
	&= (1-\rho' +\tau)D_{1}'\tsk + 2(D_{1}'A_{2}'+ \tau A)(f(\xk)-f^*)+(D_{1}'B_{2}'+B)\ns{\nabla f(\xk)} +D_{1}'C_2'  \notag\\
	&\quad + \tau C +\left(1-2\alpha +\frac{(1-\alpha)^2}{\beta}+\alpha^2 + \tau \right) \frac{\omega}{(1+\omega)m} \summ \ns{\nabla f_i(\xk) -\hi}   \label{eq:define-tau}\\
	&\leq (1-\rho) \tsk + 2(D_{1}'A_{2}'+ \tau A)(f(\xk)-f^*)+(D_{1}'B_{2}'+B)\ns{\nabla f(\xk)} +D_{1}'C_2' +  \tau C,  \label{eq:define-rho}
	\end{align}
	where \eqref{eq:define-B} holds by defining $B:=\frac{\omega(1+\beta)L^2\eta^2}{1+\omega}+ D_{1}'D_2'$, 
	\eqref{eq:pluggk} follows from \eqref{eq:define-fstar},
	\eqref{eq:define-tau} holds by defining  $\tau:=\alpha^2 \omega + \frac{(1+\omega)B}{m}$,
	and the last inequality holds by defining $\rho :=\min\{\rho'-\tau, 2\alpha -(1-\alpha)\beta^{-1} -\alpha^2-\tau \}$.
	
	Now, we have proved the second part, i.e., \eqref{eq:gk2} holds with
	\begin{align*}
	&\rho =\min\{\rho'-\tau, 2\alpha -(1-\alpha)\beta^{-1} -\alpha^2-\tau \}, \\
	&A_2 =D_{1}'A_{2}'+ \tau A, \qquad
	B_2 =D_{1}'B_{2}'+B, \qquad 
	C_2  = D_{1}'C_2' +  \tau C.
	\end{align*}
\end{proofof}

\begin{proofof}{Theorem \ref{thm:diana-type-diff}}
	Similar to the proof of Theorem \ref{thm:diana-type}, 
	we know that gradient estimator $\gk$ (see Line \ref{line:gk} of Algorithm \ref{alg:diana}) is unbiased, i.e., 
	\begin{align}
	\Ek[\gk] & = \Ek\left[h^k + \frac{1}{m}\summ \Di \right] \notag\\
	&=\Ek\left[\frac{1}{m}\summ \hi + \frac{1}{m}\summ \Ci(\gi-\hi) \right] \notag\\
	&= \Ek\left[\frac{1}{m}\summ \hi + \frac{1}{m}\summ (\gi-\hi) \right]  \notag\\
	&= \Ek\left[\frac{1}{m}\summ \gi \right]  =\nabla f(\xk) 
	\end{align}
	Then, we prove the upper bound for the second moment of gradient estimator $\gk$:
	\begin{align}
	&\Ek[\ns{\gk}] \notag\\
	&=\Ek\left[\nsB{\frac{1}{m}\summ \hi + \frac{1}{m}\summ \Ci(\gi-\hi) } \right] \notag\\
	&=\Ek\left[\nsB{\frac{1}{m}\summ \hi + \frac{1}{m}\summ \Ci(\gi-\hi) -\frac{1}{m}\summ \gi + \frac{1}{m}\summ \gi} \right] \notag\\
	&\overset{\eqref{eq:compress}}{=} \Ek\left[\nsB{\frac{1}{m}\summ \left(\Ci(\gi-\hi) - \gi +\hi \right)}\right]  
	+\Ek\left[\nsB{\frac{1}{m}\summ \gi} \right] \notag\\
	&\overset{\eqref{eq:compress}}{\leq} \frac{\omega}{m^2}  \EkB{\summ \ns{\gi-\hi}}
	+\Ek\left[\nsB{\frac{1}{m}\summ (\gi - \nabla f_i(\xk)) + \frac{1}{m}\summ \nabla f_i(\xk)} \right] \notag\\
	&=\frac{\omega}{m^2}  \EkB{\summ \ns{\gi-\hi}}
	+\frac{1}{m^2} \EkB{\summ\ns{\gi - \nabla f_i(\xk)}}
	+ \ns{\nabla f(\xk)} \notag\\
	&=\frac{\omega}{m^2}  \EkB{\summ \ns{\gi-\nabla f_i(\xk) + \nabla f_i(\xk) -\hi}}  +\frac{1}{m^2} \EkB{\summ\ns{\gi - \nabla f_i(\xk)}}  \notag\\
	&\qquad \qquad
	+ \ns{\nabla f(\xk)} \notag\\
	&= \frac{1+\omega}{m^2} \EkB{\summ\ns{\gi - \nabla f_i(\xk)}}
	+ \frac{\omega}{m^2} \summ \ns{\nabla f_i(\xk) -\hi} 
	+ \ns{\nabla f(\xk)} \notag\\
	&=\frac{1+\omega}{m^2} \summ \left(\Ek[\ns{\gi}] - \ns{\nabla f_i(\xk)}\right)
	+ \frac{\omega}{m^2} \summ \ns{\nabla f_i(\xk) -\hi} 
	+ \ns{\nabla f(\xk)} \notag\\
	&\overset{\eqref{eq:gi1-diff}}{\leq} \frac{1+\omega}{m^2} \summ \left(2A_{1,i}(f_i(\xk)-f_i^*)+B_{1,i}\ns{\nabla f_i(\xk)} + D_{1,i} \ski +C_{1,i} - \ns{\nabla f_i(\xk)}\right) \notag\\
	&\qquad \qquad
	+ \frac{\omega}{m^2} \summ \ns{\nabla f_i(\xk) -\hi} 
	+ \ns{\nabla f(\xk)} \notag\\
	&\leq \frac{1+\omega}{m^2} \summ \left(2(A_{1,i}+(B_{1,i}-1)L_i)(f_i(\xk)-f_i^*) + D_{1,i} \ski   +C_{1,i} \right) \notag\\
	&\qquad \qquad
	+ \frac{\omega}{m^2} \summ \ns{\nabla f_i(\xk) -\hi} 
	+ \ns{\nabla f(\xk)} \notag\\
	&\leq \frac{2(1+\omega)A}{m^2}\summ (f_i(\xk)-f_i^*)  + \frac{1+\omega}{m^2} \summ  D_{1,i} \ski  + \frac{1+\omega}{m^2}\summ C_{1,i} \notag\\
	&\qquad \qquad
	+ \frac{\omega}{m^2} \summ \ns{\nabla f_i(\xk) -\hi} 
	+ \ns{\nabla f(\xk)} \label{eq:define-a-diff}\\
	&= \frac{2(1+\omega)A}{m}(f(\xk)-f^*)  
	+ \frac{1+\omega}{m}  \left(\frac{1}{m} \summ  D_{1,i} \ski  +\frac{\omega}{(1+\omega)m} \summ \ns{\nabla f_i(\xk) -\hi}\right) \notag\\
	&\qquad \qquad		+ \ns{\nabla f(\xk)}  + \frac{(1+\omega)C}{m}, \label{eq:define-fstar-diff}
	\end{align}
	where \eqref{eq:define-a-diff} holds by defining $A:=\max_i  (A_{1,i}+(B_{1,i}-1)L_i)$, and \eqref{eq:define-fstar-diff} holds by defining $C := \frac{1}{m} \summ C_{1,i} + 2A\Delta_f^*$ and $\Delta_f^*:=f^*-\frac{1}{m}\summ f_i^*$.	
	
	Thus, we have proved the first part, i.e., \eqref{eq:gk1-diff} holds with
	\begin{align}
	&A_1 =\frac{(1+\omega)A}{m}, \qquad
	B_1 =1, \qquad 
	C_1  =  \frac{(1+\omega)C}{m}, \\ 
	&D_1 =\frac{1+\omega}{m}, \qquad
	\sk =  \frac{1}{m} \summ  D_{1,i} \ski  + \frac{\omega}{(1+\omega)m}
	\summ \ns{\nabla f_i(\xk) -\hi}.
	\end{align}
	Now we prove the second part (i.e., \eqref{eq:gk2-diff}). According to \eqref{eq:gi2-diff}, we have 
	\begin{align}
	&\EkB{\frac{1}{m} \summ  D_{1,i} \skin}  \notag\\ 
	&\overset{\eqref{eq:gi2-diff}}{\leq}
	\frac{1}{m} \summ  D_{1,i} \left( (1-\rho_i)\sigma_{k,i}^2 + 2A_{2,i}(f(\xk)-f^*)+B_{2,i}\ns{\nabla f(\xk)} + D_{2,i}\Ek[\ns{\gk}] +C_{2,i}\right) \\
	&= \frac{1}{m} \summ (1-\rho_i) D_{1,i} \ski 
	+ 2D_A(f(\xk)-f^*)  
	+  D_B\ns{\nabla f(\xk)}
	+ D_D \Ek[\ns{\gk}]
	+D_C,   \label{eq:term1-diff}
	\end{align}	
	where the last equality holds by defining 
	$D_A:=\frac{1}{m} \summ  D_{1,i}A_{2,i}$,
	$D_B:=\frac{1}{m} \summ  D_{1,i} B_{2,i}$,
	$D_D:=\frac{1}{m} \summ  D_{1,i} D_{2,i}$,
	and $D_C:=\frac{1}{m} \summ  D_{1,i}C_{2,i}$.
	
	For the second term of $\skin$, we bound it as follows:
	\begin{align}
	&\EkB{\frac{\omega}{(1+\omega)m}
		\summ \ns{\nabla f_i(\xkn) -\hin}}  \notag\\
	& = \EkB{\frac{\omega}{(1+\omega)m}
		\summ \ns{\nabla f_i(\xkn) -\nabla f_i(\xk) + \nabla f_i(\xk) -\hin}} \notag\\
	& = \EkB{\frac{\omega}{(1+\omega)m}
		\summ \ns{\nabla f_i(\xkn) -\nabla f_i(\xk) + \nabla f_i(\xk) -\hi - \alpha \Ci(\gi-\hi)}} \notag\\
	& = \frac{\omega}{(1+\omega)m} \summ \Ek\Big[
	\ns{\nabla f_i(\xkn) -\nabla f_i(\xk)} 
	+\ns{\nabla f_i(\xk) -\hi - \alpha \Ci(\gi-\hi)} \notag\\
	& \qquad \qquad 
	+2\inner{\nabla f_i(\xkn) -\nabla f_i(\xk)}{\nabla f_i(\xk) -\hi - \alpha \Ci(\gi-\hi)}\Big] \notag\\
	& = \frac{\omega}{(1+\omega)m} \summ \Ek\Big[
	\ns{\nabla f_i(\xkn) -\nabla f_i(\xk)} 
	+\ns{\nabla f_i(\xk) -\hi - \alpha \Ci(\gi-\hi)} \notag\\
	& \qquad \qquad 
	+2\inner{\nabla f_i(\xkn) -\nabla f_i(\xk)}{(1-\alpha)(\nabla f_i(\xk) -\hi)}\Big] \notag\\
	& = \frac{\omega}{(1+\omega)m} \summ \Ek\Big[
	\ns{\nabla f_i(\xkn) -\nabla f_i(\xk)} 
	+(1-2\alpha)\ns{\nabla f_i(\xk) -\hi}
	+ \alpha^2 \ns{\Ci(\gi-\hi)} \notag\\
	& \qquad \qquad 
	+2\inner{\nabla f_i(\xkn) -\nabla f_i(\xk)}{(1-\alpha)(\nabla f_i(\xk) -\hi)}\Big] \notag\\
	& \leq \frac{\omega}{(1+\omega)m} \summ \Ek\Big[
	\ns{\nabla f_i(\xkn) -\nabla f_i(\xk)} 
	+(1-2\alpha)\ns{\nabla f_i(\xk) -\hi}
	+ \alpha^2 \ns{\Ci(\gi-\hi)} \notag\\
	& \qquad \qquad 
	+\beta\ns{\nabla f_i(\xkn) -\nabla f_i(\xk)} 
	+\frac{(1-\alpha)^2}{\beta}\ns{\nabla f_i(\xk) -\hi}\Big] \qquad \forall \beta>0\notag
	\end{align}
	\begin{align}
	& = \frac{\omega}{(1+\omega)m} \summ \Ek\Big[
	(1+\beta)\ns{\nabla f_i(\xkn) -\nabla f_i(\xk)} 
	+(1-2\alpha +\frac{(1-\alpha)^2}{\beta})\ns{\nabla f_i(\xk) -\hi}
	\notag\\
	& \qquad \qquad 
	+ \alpha^2 \ns{\Ci(\gi-\hi)}\Big]  \notag\\
	& \overset{\eqref{eq:compress}}{\leq}  \frac{\omega}{(1+\omega)m} \summ \Ek\Big[
	(1+\beta)\ns{\nabla f_i(\xkn) -\nabla f_i(\xk)} 
	+(1-2\alpha +\frac{(1-\alpha)^2}{\beta})\ns{\nabla f_i(\xk) -\hi}
	\notag\\
	& \qquad \qquad 
	+ \alpha^2 (1+\omega)\ns{\gi-\hi}\Big]  \notag\\
	& = \frac{\omega}{(1+\omega)m} \summ \Ek\Big[
	(1+\beta)\ns{\nabla f_i(\xkn) -\nabla f_i(\xk)} 
	+(1-2\alpha +\frac{(1-\alpha)^2}{\beta})\ns{\nabla f_i(\xk) -\hi}
	\notag\\
	& \qquad \qquad 
	+ \alpha^2 (1+\omega)\ns{\gi-\nabla f_i(\xk) + \nabla f_i(\xk) -\hi}\Big]  \notag\\
	& = \left(1-2\alpha +\frac{(1-\alpha)^2}{\beta}+\alpha^2 (1+\omega)\right) \frac{\omega}{(1+\omega)m} \summ \ns{\nabla f_i(\xk) -\hi}   \notag\\
	&\qquad \qquad 
	+ \frac{\omega(1+\beta)}{(1+\omega)m} \summ 
	\Ek[\ns{\nabla f_i(\xkn) 	-\nabla f_i(\xk)} ]
	+ \frac{\alpha^2 \omega}{m} \summ 
	\Ek[\ns{\gi-\nabla f_i(\xk)}]  \notag\\
	& \leq \left(1-2\alpha +\frac{(1-\alpha)^2}{\beta}+\alpha^2 (1+\omega)\right) \frac{\omega}{(1+\omega)m} \summ \ns{\nabla f_i(\xk) -\hi}   \notag\\
	&\qquad \qquad 
	+ \frac{\omega(1+\beta)}{(1+\omega)m} \summ 
	\Ek[\ns{\nabla f_i(\xkn) 	-\nabla f_i(\xk)} ]
	\notag\\
	&\qquad \qquad 
	+ \alpha^2 \omega \left(2A (f(\xk)-f^*) +\frac{1}{m} \summ  D_{1,i} \ski  +C\right)\notag\\
	& \leq \left(1-2\alpha +\frac{(1-\alpha)^2}{\beta}+\alpha^2 (1+\omega)\right) \frac{\omega}{(1+\omega)m} \summ \ns{\nabla f_i(\xk) -\hi}   \notag\\
	&\qquad \qquad 
	+ \frac{\omega(1+\beta)}{(1+\omega)m} \summ 
	L_i^2\Ek[\ns{\xkn -\xk} ]\notag\\
	&\qquad \qquad 
	+ \alpha^2 \omega \left(2A (f(\xk)-f^*) +\frac{1}{m} \summ  D_{1,i} \ski +C\right)  \notag\\
	& = \left(1-2\alpha +\frac{(1-\alpha)^2}{\beta}+\alpha^2 (1+\omega)\right) \frac{\omega}{(1+\omega)m} \summ \ns{\nabla f_i(\xk) -\hi}   \notag\\
	&\qquad \qquad 
	+ \frac{\omega(1+\beta)}{(1+\omega)m} \summ 
	L_i^2\eta^2\Ek[\ns{\gk} ]
	+ \alpha^2 \omega \left(2A (f(\xk)-f^*) +\frac{1}{m} \summ  D_{1,i} \ski  +C\right).  \label{eq:term2-diff}
	\end{align}	
	By combining \eqref{eq:term1-diff} and \eqref{eq:term2-diff}, we have 
	\begin{align}
	&\Ek[\skn]  \notag\\
	&=\EkB{\frac{1}{m} \summ  D_{1,i} \ski  + \frac{\omega}{(1+\omega)m}
		\summ \ns{\nabla f_i(\xkn) -\hin}} \notag
	\end{align}
	\begin{align}
	&\leq \frac{1}{m} \summ (1-\rho_i) D_{1,i} \ski + 2D_A(f(\xk)-f^*)  
	+ D_B\ns{\nabla f(\xk)} +D_D \Ek[\ns{\gk}] +D_C \notag\\
	&\qquad 
	+ \left(1-2\alpha +\frac{(1-\alpha)^2}{\beta}+\alpha^2 (1+\omega)\right) \frac{\omega}{(1+\omega)m} \summ \ns{\nabla f_i(\xk) -\hi}   \notag\\
	&\qquad 
	+ \frac{\omega(1+\beta)}{(1+\omega)m} \summ 
	L_i^2\eta^2\Ek[\ns{\gk} ]
	+ \alpha^2 \omega \left(2A (f(\xk)-f^*) +\frac{1}{m} \summ  D_{1,i} \ski  +C\right) \notag\\
	&= \frac{1}{m} \summ (1-\rho_i +\alpha^2 \omega)D_{1,i}\ski + 2(D_A+\alpha^2 \omega A)(f(\xk)-f^*)+D_B\ns{\nabla f(\xk)} +D_C  \notag\\
	&\qquad + \alpha^2 \omega C  +\left(1-2\alpha +\frac{(1-\alpha)^2}{\beta}+\alpha^2 (1+\omega)\right) \frac{\omega}{(1+\omega)m} \summ \ns{\nabla f_i(\xk) -\hi}   \notag\\
	&\qquad 
	+ B\Ek[\ns{\gk}] \label{eq:define-B-diff} \\
	&= \frac{1}{m} \summ (1-\rho_i +\alpha^2 \omega)D_{1,i}\ski + 2(D_A+\alpha^2 \omega A)(f(\xk)-f^*)+D_B\ns{\nabla f(\xk)} +D_C   \notag\\
	&\qquad + \alpha^2 \omega C +\left(1-2\alpha +\frac{(1-\alpha)^2}{\beta}+\alpha^2 (1+\omega)\right) \frac{\omega}{(1+\omega)m} \summ \ns{\nabla f_i(\xk) -\hi}   \notag\\
	&\qquad  
	+ \frac{2(1+\omega)BA}{m}(f(\xk)-f^*)  \notag\\
	&\qquad 
	+ \frac{(1+\omega)B}{m}  \left(\frac{1}{m} \summ  D_{1,i} \ski + \frac{\omega}{(1+\omega)m}
	\summ \ns{\nabla f_i(\xk) -\hi}\right)  \notag\\
	&\qquad 		+ B\ns{\nabla f(\xk)} + \frac{(1+\omega)B}{m}C \label{eq:pluggk-diff}\\
	&= \frac{1}{m} \summ (1-\rho_i +\tau)D_{1,i}\ski  + 2(D_A+ \tau A)(f(\xk)-f^*)+(D_B+B)\ns{\nabla f(\xk)} +D_C  \notag\\
	&\qquad + \tau C +\left(1-2\alpha +\frac{(1-\alpha)^2}{\beta}+\alpha^2 + \tau \right) \frac{\omega}{(1+\omega)m} \summ \ns{\nabla f_i(\xk) -\hi}   \label{eq:define-tau-diff}\\
	&\leq (1-\rho) \tsk + 2(D_A+ \tau A)(f(\xk)-f^*)+(D_B+B)\ns{\nabla f(\xk)} +D_C +  \tau C,  \label{eq:define-rho-diff}
	\end{align}
	where \eqref{eq:define-B-diff} holds by defining $B:=\frac{\omega(1+\beta)L^2\eta^2}{1+\omega} + D_D$, 
	\eqref{eq:pluggk-diff} follows from \eqref{eq:define-fstar-diff},
	\eqref{eq:define-tau-diff} holds by defining  $\tau:=\alpha^2 \omega + \frac{(1+\omega)B}{m}$,
	and the last inequality holds by defining $\rho :=\min\{\min_i\rho_i-\tau, 2\alpha -(1-\alpha)\beta^{-1} -\alpha^2-\tau \}$.
	
	Now, we have proved the second part, i.e., \eqref{eq:gk2-diff} holds with
	\begin{align*}
	&\rho =\min\{\min_i\rho_i-\tau, 2\alpha -(1-\alpha)\beta^{-1} -\alpha^2-\tau \}, \\
	&A_2 =D_A+ \tau A, \qquad
	B_2 =D_B+B, \qquad 
	C_2  = D_C +  \tau C.
	\end{align*}
\end{proofof}

In the following sections, we prove that if the parallel workers use some specific methods, i.e., GD, SGD, L-SVRG and SAGA, for computing their local stochastic gradient $\widetilde{g}_i^k$ (see Line \ref{line:localgrad} of Algorithm \ref{alg:diana}), then $g^k$ (see Line \ref{line:gk} of Algorithm \ref{alg:diana}) satisfies the unified Assumption \ref{asp:boge}.
Then, we plug their corresponding parameters (i.e., specific values for $A_1, A_2, B_1, B_2, C_1,C_2,D_1,\rho$) into our unified Theorem \ref{thm:main} to obtain the detailed convergence rates for these methods. 

\subsubsection{DIANA-GD method}

In this section, we show that if the parallel workers use GD for computing their local  gradient $\widetilde{g}_i^k$, then $g^k$ (see Line \ref{line:gk-diana-gd} of Algorithm \ref{alg:diana-gd}) satisfies the unified Assumption \ref{asp:boge}.

\begin{algorithm}[htb]
	\caption{DIANA-GD}
	\label{alg:diana-gd}
	\begin{algorithmic}[1]
		\REQUIRE ~
		initial point $x^0$, $\{h_i^0\}_{i=1}^m$, $h^0=\frac{1}{m}\sum_{i=1}^{m}h_i^0$, parameters $\eta_k, \alpha_k$
		\FOR {$k=0,1,2,\ldots$}
		\STATE {\bf{for all machines $i= 1,2,\ldots,m$ do in parallel}}
		\STATE \quad Compute local  gradient $\widetilde{g}_i^k=\nabla f_i(x^k)$ \label{line:localgrad-diana-gd}
		\STATE \quad Compress shifted local gradient $\widehat{\Delta}_i^k= \cC_i^k(\widetilde{g}_i^k- h_i^k)$ and send $\widehat{\Delta}_i^k$ to the server
		\STATE \quad Update local shift $h_i^{k+1}=h_i^k+\alpha_k \cC_i^k(\widetilde{g}_i^k - h_i^k)$
		\STATE {\bf{end for}}
		\STATE Aggregate received compressed gradient information
		$g^k = h^k + \frac{1}{m}\sum_{i=1}^m \widehat{\Delta}_i^k$ \label{line:gk-diana-gd}
		\STATE $x^{k+1} = x^k - \eta_k g^k$  
		\STATE $h^{k+1} = h^k + \alpha_k  \frac{1}{m}\sum_{i=1}^m  \widehat{\Delta}_i^k$
		\ENDFOR
	\end{algorithmic}
\end{algorithm}

\begin{lemma}[DIANA-GD]\label{lem:diana-gd}
	Let the local gradient estimator $\widetilde{g}_i^k=\nabla f_i(x^k)$ (see Line \ref{line:localgrad-diana-gd} of Algorithm \ref{alg:diana-gd}), then we know that $\widetilde{g}_i^k$ satisfies \eqref{eq:gi1} and \eqref{eq:gi2} with $B_{1,i}=1, A_{1,i}=C_{1,i}=D_{1}'=0, \tsk \equiv 0, \rho'=1, A_2'=B_2'=C_2'=D_2'=0$. 
	Thus, according to Theorem \ref{thm:diana-type}, $g^k$ (see Line \ref{line:gk-diana-gd} of Algorithm \ref{alg:diana-gd})  satisfies the unified Assumption \ref{asp:boge} with
	\begin{align*}
	&A_1 =0, \qquad
	B_1 =1, \qquad 
	C_1  =  0, \\ 
	&D_1 =\frac{1+\omega}{m}, \qquad
	\sk =  \frac{\omega}{(1+\omega)m}
	\summ \ns{\nabla f_i(\xk) -\hi}, \\
	&\rho =\min\{1-\tau,~ 2\alpha -(1-\alpha)\beta^{-1} -\alpha^2-\tau \},\\
	&A_2 =0, \qquad
	B_2 =B, \qquad 
	C_2  = 0,
	\end{align*}
	where 
	$B:=\frac{\omega(1+\beta)L^2\eta^2}{1+\omega}$,
	$\tau:=\alpha^2 \omega + \frac{(1+\omega)B}{m}$,
	and $\forall \beta>0$.
\end{lemma}

\begin{proofof}{Lemma \ref{lem:diana-gd}}
	If the local gradient estimator $\gi=\nabla f_i(\xk)$ (see Line \ref{line:localgrad-diana-gd} of Algorithm \ref{alg:diana-gd}), it is easy to see that 
	$$\Ek[\gi] = \nabla f_i(\xk)$$ 
	and 
	$$\Ek[\ns{\gi}] \leq \ns{ \nabla f_i(\xk)}$$
	since there is no randomness. Thus the local gradient estimator $\gi$ satisfies \eqref{eq:gi1} and \eqref{eq:gi2} with 
	$$A_{1,i}=C_{1,i}=D_{1}'=0, B_{1,i}=1,\tsk \equiv 0, \rho'=1, A_2'=B_2'=C_2'=D_2'=0.$$ 
	Then, according to Theorem \ref{thm:diana-type}, $g^k$ (see Line \ref{line:gk-diana-gd} of Algorithm \ref{alg:diana-gd}) satisfies the unified Assumption \ref{asp:boge} with
	\begin{align*}
	&A_1 =0, \qquad
	B_1 =1, \qquad 
	C_1  =  0, \\ 
	&D_1 =\frac{1+\omega}{m}, \qquad
	\sk =  \frac{\omega}{(1+\omega)m}
	\summ \ns{\nabla f_i(\xk) -\hi}, \\
	&\rho =\min\{1-\tau,~ 2\alpha -(1-\alpha)\beta^{-1} -\alpha^2-\tau \},\\
	&A_2 =0, \qquad
	B_2 =B, \qquad 
	C_2  = 0,
	\end{align*}
	where 
	$B:=\frac{\omega(1+\beta)L^2\eta^2}{1+\omega}$,
	$\tau:=\alpha^2 \omega + \frac{(1+\omega)B}{m}$,
	and $\forall \beta>0$.
\end{proofof}

\begin{corollary}[DIANA-GD]\label{cor:diana-gd}
	Suppose that Assumption \ref{asp:lsmooth-diana} holds. 
	Let stepsize 
	$\eta \leq \frac{1}{L+L(1+\omega)Bm^{-1}\rho^{-1}}$, then the number of iterations performed by DIANA-GD (Algorithm \ref{alg:diana-gd}) to find an $\epsilon$-solution of nonconvex federated problem \eqref{eq:prob-fed}, i.e. a point $\hx$ such that $\E[\n{\nabla f(\hx)}] \leq \epsilon$, can be bounded by
	\begin{align}
	K = \frac{8\fgap L}{\epsilon^2} \left(1+\frac{(1+\omega)B}{m\rho}\right),
	\end{align}
	where $B:=\frac{\omega(1+\beta)L^2\eta^2}{1+\omega}$, 
	$\rho:=\min\{1-\tau,~ 2\alpha -(1-\alpha)\beta^{-1} -\alpha^2-\tau \}$,
	$\tau:=\alpha^2 \omega + \frac{(1+\omega)B}{m}$,
	and $\forall \beta>0$.
\end{corollary}

\begin{proofof}{Corollary \ref{cor:diana-gd}}
	According to our unified Theorem \ref{thm:main}, if the stepsize is chosen as 
	\begin{align}
	\eta &\leq \min\left\{ \frac{1}{LB_1+LD_1B_2\rho^{-1}}, \sqrt{\frac{\ln 2}{(LA_1 + LD_1A_2\rho^{-1})K}}, \frac{\epsilon^2}{2L(C_1+D_1C_2\rho^{-1})} \right\}  \notag\\
	&= \frac{1}{L+L(1+\omega)Bm^{-1}\rho^{-1}}
	\end{align}
	since $B_1=1, D_1 =\frac{1+\omega}{m}, B_2=B, B=\frac{\omega(1+\beta)L^2\eta^2}{1+\omega}$ and $A_1=A_2=C_1=C_2=0$
	according to Lemma \ref{lem:diana-gd},
	then the number of iterations performed by DIANA-GD (Algorithm \ref{alg:diana-gd}) to find an $\epsilon$-solution of problem \eqref{eq:prob-fed} can be bounded by
	\begin{align}
	K &= \frac{8\fgapp L}{\epsilon^2} \max \left\{ B_1+D_1B_2\rho^{-1}, \frac{12\fgapp (A_1+D_1 A_2\rho^{-1})}{\epsilon^2}, \frac{2(C_1+D_1C_2\rho^{-1})}{\epsilon^2}\right\}  \notag\\
	&= \frac{8\fgap L}{\epsilon^2}  \left(1+\frac{(1+\omega)B}{m\rho}\right)
	\end{align}
	since $B_1=1, D_1 =\frac{1+\omega}{m}, B_2=B, B=\frac{\omega(1+\beta)L^2\eta^2}{1+\omega}, A_1=A_2=C_1=C_2=0$ and 
	$\fgapp:= f(x^0) - f^* + 2^{-1}L\eta^2D_1\rho^{-1} \sigma_0^2 = f(x^0) - f^* =\fgap$ by letting $\sigma_0^2 =\frac{\omega}{(1+\omega)m}
	\summ \ns{\nabla f_i(x^0) -h^0}=0$.
\end{proofof}

\subsubsection{DIANA-SGD method}
In this section, we show that if the parallel workers use SGD for computing their local  stochastic gradient $\widetilde{g}_i^k$, then $g^k$ (see Line \ref{line:gk-diana-sgd} of Algorithm \ref{alg:diana-sgd}) satisfies the unified Assumption \ref{asp:boge}.

\begin{algorithm}[htb]
	\caption{DIANA-SGD}
	\label{alg:diana-sgd}
	\begin{algorithmic}[1]
		\REQUIRE ~
		initial point $x^0$, $\{h_i^0\}_{i=1}^m$, $h^0=\frac{1}{m}\sum_{i=1}^{m}h_i^0$, parameters $\eta_k, \alpha_k$
		\FOR {$k=0,1,2,\ldots$}
		\STATE {\bf{for all machines $i= 1,2,\ldots,m$ do in parallel}}
		\STATE \quad Compute local stochastic gradient $\widetilde{g}_i^k$ with Algorithm \ref{alg:sgd} by changing $f(x)$ to the local $f_i(x)$ \label{line:localgrad-diana-sgd}
		\STATE \quad Compress shifted local gradient $\widehat{\Delta}_i^k= \cC_i^k(\widetilde{g}_i^k- h_i^k)$ and send $\widehat{\Delta}_i^k$ to the server
		\STATE \quad Update local shift $h_i^{k+1}=h_i^k+\alpha_k \cC_i^k(\widetilde{g}_i^k - h_i^k)$
		\STATE {\bf{end for}}
		\STATE Aggregate received compressed gradient information
		$g^k = h^k + \frac{1}{m}\sum_{i=1}^m \widehat{\Delta}_i^k$
		\label{line:gk-diana-sgd}
		\STATE $x^{k+1} = x^k - \eta_k g^k$  
		\STATE $h^{k+1} = h^k + \alpha_k  \frac{1}{m}\sum_{i=1}^m  \widehat{\Delta}_i^k$
		\ENDFOR
	\end{algorithmic}
\end{algorithm}

\begin{lemma}[DIANA-SGD]\label{lem:diana-sgd}
	Let the local stochastic gradient estimator $\widetilde{g}_i^k$ (see Line \ref{line:localgrad-diana-sgd} of Algorithm \ref{alg:diana-sgd}) satisfy Assumption \ref{asp:es}, i.e.,
	$$ \E_k[\ns{\widetilde{g}_i^k}] \leq 2A_{1,i}(f_i(x^k)-f_i^*)+B_{1,i}\ns{\nabla f_i(x^k)} + C_{1,i},$$
	then we know that $\widetilde{g}_i^k$ satisfies \eqref{eq:gi1} and \eqref{eq:gi2} with $D_{1}'=0, \tsk \equiv 0, \rho'=1, A_2'=B_2'=C_2'=D_2'=0$. 
	Thus, according to Theorem \ref{thm:diana-type}, $g^k$ (see Line \ref{line:gk-diana-sgd} of Algorithm \ref{alg:diana-sgd})  satisfies the unified Assumption \ref{asp:boge} with
	\begin{align*}
	&A_1 =\frac{(1+\omega)A}{m}, \qquad
	B_1 =1, \qquad 
	C_1  =  \frac{(1+\omega)C}{m}, \\ 
	&D_1 =\frac{1+\omega}{m}, \qquad
	\sk = \frac{\omega}{(1+\omega)m}
	\summ \ns{\nabla f_i(\xk) -\hi}, \\
	&\rho =\min\{1-\tau,~ 2\alpha -(1-\alpha)\beta^{-1} -\alpha^2-\tau \},\\
	&A_2 =\tau A, \qquad
	B_2 =B, \qquad 
	C_2  = \tau C,
	\end{align*}
	where 
	$A:=\max_i  (A_{1,i}+(B_{1,i}-1)L_i)$, 
	$B:=\frac{\omega(1+\beta)L^2\eta^2}{1+\omega}$,
	$C := \frac{1}{m} \summ C_{1,i} + 2A\Delta_f^*$,
	$\Delta_f^*:=f^*-\frac{1}{m}\summ f_i^*$,
	$\tau:=\alpha^2 \omega + \frac{(1+\omega)B}{m}$,
	and $\forall\beta>0$.
\end{lemma}

\begin{proofof}{Lemma \ref{lem:diana-sgd}}
	Suppose that the local stochastic gradient estimator $\widetilde{g}_i^k$ (see Line \ref{line:localgrad-diana-sgd} of Algorithm \ref{alg:diana-sgd}) satisfies Assumption \ref{asp:es}, i.e., 
	$$\Ek[\gi] = \nabla f_i(\xk)$$
	and 
	$$ \Ek[\ns{\gi}] \leq 2A_{1,i}(f_i(x^k)-f_i^*)+B_{1,i}\ns{\nabla f_i(x^k)} + C_{1,i},$$
	Thus the local stochastic gradient estimator $\gi$ satisfies \eqref{eq:gi1} and \eqref{eq:gi2} with 
	$$A_{1,i}=A_{1,i}, B_{1,i}=B_{1,i},C_{1,i}=C_{1,i}, D_{1}'=0,\tsk \equiv 0, \rho'=1, A_2'=B_2'=C_2'=D_2'=0.$$ 
	Then, according to Theorem \ref{thm:diana-type}, $g^k$ (see Line \ref{line:gk-diana-sgd} of Algorithm \ref{alg:diana-sgd}) satisfies the unified Assumption \ref{asp:boge} with
	\begin{align*}
	&A_1 =\frac{(1+\omega)A}{m}, \qquad
	B_1 =1, \qquad 
	C_1  =  \frac{(1+\omega)C}{m}, \\ 
	&D_1 =\frac{1+\omega}{m}, \qquad
	\sk = \frac{\omega}{(1+\omega)m}
	\summ \ns{\nabla f_i(\xk) -\hi}, \\
	&\rho =\min\{1-\tau,~ 2\alpha -(1-\alpha)\beta^{-1} -\alpha^2-\tau \},\\
	&A_2 =\tau A, \qquad
	B_2 =B, \qquad 
	C_2  = \tau C,
	\end{align*}
	where 
	$A:=\max_i  (A_{1,i}+(B_{1,i}-1)L_i)$, 
	$B:=\frac{\omega(1+\beta)L^2\eta^2}{1+\omega}$,
	$C := \frac{1}{m} \summ C_{1,i} + 2A\Delta_f^*$,
	$\Delta_f^*:=f^*-\frac{1}{m}\summ f_i^*$,
	$\tau:=\alpha^2 \omega + \frac{(1+\omega)B}{m}$,
	and $\forall\beta>0$.
\end{proofof}

\begin{corollary}[DIANA-SGD]\label{cor:diana-sgd}
	Suppose that Assumption \ref{asp:lsmooth-diana} holds. 
	Let stepsize 
	$$\eta \leq \min\left\{ \frac{1}{L+L(1+\omega)Bm^{-1}\rho^{-1}},~ \sqrt{\frac{m\ln 2}{(1+\omega)( 1+\tau\rho^{-1})LAK}},~ \frac{m\epsilon^2}{2(1+\omega)( 1+\tau\rho^{-1})LC} \right\},$$ 
	then the number of iterations performed by DIANA-SGD (Algorithm \ref{alg:diana-sgd}) to find an $\epsilon$-solution of nonconvex federated problem \eqref{eq:prob-fed} with \eqref{prob-fed:exp} or \eqref{prob-fed:finite}, i.e. a point $\hx$ such that $\E[\n{\nabla f(\hx)}] \leq \epsilon$, can be bounded by	
	\begin{align}
	K = \frac{8\fgap L}{\epsilon^2} \max \left\{ 1+\frac{(1+\omega)B}{m\rho},~ \frac{12(1+\omega)( 1+\tau\rho^{-1})\fgap A}{m\epsilon^2},~ \frac{2(1+\omega)( 1+\tau\rho^{-1})LC}{m\epsilon^2}\right\},
	\end{align}
	where 
	$A:=\max_i  (A_{1,i}+(B_{1,i}-1)L_i)$, 
	$B:=\frac{\omega(1+\beta)L^2\eta^2}{1+\omega}$, 
	$C := \frac{1}{m} \summ C_{1,i} + 2A\Delta_f^*$,
	$\Delta_f^*:=f^*-\frac{1}{m}\summ f_i^*$,
	$\rho :=\min\{1-\tau, 2\alpha -(1-\alpha)\beta^{-1} -\alpha^2-\tau \}$,
	$\tau:=\alpha^2 \omega + \frac{(1+\omega)B}{m}$,
	and $\forall \beta>0$.
\end{corollary}

\begin{proofof}{Corollary \ref{cor:diana-sgd}}
	According to our unified Theorem \ref{thm:main}, if the stepsize is chosen as 
	\begin{align*}
	\eta &\leq \min\left\{ \frac{1}{LB_1+LD_1B_2\rho^{-1}}, \sqrt{\frac{\ln 2}{(LA_1 + LD_1A_2\rho^{-1})K}}, \frac{\epsilon^2}{2L(C_1+D_1C_2\rho^{-1})} \right\}  \notag\\
	&= \min\left\{\frac{1}{L+L(1+\omega)Bm^{-1}\rho^{-1}},~ \sqrt{\frac{m\ln 2}{(1+\omega)( 1+\tau\rho^{-1})LAK}},~ \frac{m\epsilon^2}{2(1+\omega)( 1+\tau\rho^{-1})LC} \right\}
	\end{align*}
	since $B_1 =1, D_1 =\frac{1+\omega}{m}, B_2=B, B=\frac{\omega(1+\beta)L^2\eta^2}{1+\omega}, A_1 =\frac{(1+\omega)A}{m},
	A_2 =\tau A, C_1  =  \frac{(1+\omega)C}{m}$ and $C_2  = \tau C$ 
	according to Lemma \ref{lem:diana-sgd},
	then the number of iterations performed by DIANA-SGD (Algorithm \ref{alg:diana-sgd}) to find an $\epsilon$-solution of problem \eqref{eq:prob-fed} with \eqref{prob-fed:exp} or \eqref{prob-fed:finite} can be bounded by
	\begin{align}
	K &= \frac{8\fgapp L}{\epsilon^2} \max \left\{ B_1+D_1B_2\rho^{-1}, \frac{12\fgapp (A_1+D_1 A_2\rho^{-1})}{\epsilon^2}, \frac{2(C_1+D_1C_2\rho^{-1})}{\epsilon^2}\right\}  \notag\\
	&= \frac{8\fgap L}{\epsilon^2} \max \left\{1+\frac{(1+\omega)B}{m\rho},~ \frac{12(1+\omega)( 1+\tau\rho^{-1})\fgap A}{m\epsilon^2},~ \frac{2(1+\omega)( 1+\tau\rho^{-1})LC}{m\epsilon^2}\right\}
	\end{align}
	since $B_1 =1, D_1 =\frac{1+\omega}{m}, B_2=B, B=\frac{\omega(1+\beta)L^2\eta^2}{1+\omega}, A_1 =\frac{(1+\omega)A}{m},
	A_2 =\tau A, C_1  =  \frac{(1+\omega)C}{m}, C_2  = \tau C$ and 
	$\fgapp:= f(x^0) - f^* + 2^{-1}L\eta^2D_1\rho^{-1} \sigma_0^2 = f(x^0) - f^* =\fgap$ by letting $\sigma_0^2 =\frac{\omega}{(1+\omega)m}
	\summ \ns{\nabla f_i(x^0) -h^0}=0$.
\end{proofof}

\subsubsection{DIANA-LSVRG method}
In this section, we show that if the parallel workers use L-SVRG for computing their local  gradient $\widetilde{g}_i^k$, then $g^k$ (see Line \ref{line:gk-diana-lsvrg} of Algorithm \ref{alg:diana-lsvrg}) satisfies the unified Assumption \ref{asp:boge}.

\begin{algorithm}[htb]
	\caption{DIANA-LSVRG}
	\label{alg:diana-lsvrg}
	\begin{algorithmic}[1]
		\REQUIRE ~
		initial point $x^0=w^0$, $\{h_i^0\}_{i=1}^m$, $h^0=\frac{1}{m}\sum_{i=1}^{m}h_i^0$, parameters $b, \eta_k, \alpha_k$, probability $p\in (0,1]$
		\FOR {$k=0,1,2,\ldots$}
		\STATE {\bf{for all machines $i= 1,2,\ldots,m$ do in parallel}}
		\STATE \quad Compute local stochastic gradient $\widetilde{g}_i^k=\frac{1}{b} \sum_{j\in I_b} (\nabla f_{i,j}(x^k)- \nabla f_{i,j}(w^k)) +\nabla f_i(w^k)$ \label{line:localgrad-diana-lsvrg}
		\STATE \quad Compress shifted local gradient $\widehat{\Delta}_i^k= \cC_i^k(\widetilde{g}_i^k- h_i^k)$ and send $\widehat{\Delta}_i^k$ to the server
		\STATE \quad Update local shift $h_i^{k+1}=h_i^k+\alpha_k \cC_i^k(\widetilde{g}_i^k - h_i^k)$
		\STATE \quad $w^{k+1} = \begin{cases}
		x^k, &\text {with probability } p\\
		w^k, &\text {with probability } 1-p
		\end{cases}$ \label{line:w_prob-diana-lsvrg}
		\STATE {\bf{end for}}
		\STATE Aggregate received compressed gradient information
		$g^k = h^k + \frac{1}{m}\sum_{i=1}^m \widehat{\Delta}_i^k$
		\label{line:gk-diana-lsvrg}
		\STATE $x^{k+1} = x^k - \eta_k g^k$   \label{line:update-diana-lsvrg}
		\STATE $h^{k+1} = h^k + \alpha_k  \frac{1}{m}\sum_{i=1}^m  \widehat{\Delta}_i^k$
		\ENDFOR
	\end{algorithmic}
\end{algorithm}

\begin{lemma}[DIANA-LSVRG]\label{lem:diana-lsvrg}
	Let the local stochastic gradient estimator $\widetilde{g}_i^k=\frac{1}{b} \sum_{j\in I_b} (\nabla f_{i,j}(x^k)- \nabla f_{i,j}(w^k)) +\nabla f_i(w^k)$ (see Line \ref{line:localgrad-diana-lsvrg} of Algorithm \ref{alg:diana-lsvrg}), then we know that $\widetilde{g}_i^k$ satisfies \eqref{eq:gi1} and \eqref{eq:gi2} with $A_{1,i}=0, B_{1,i}=1, C_{1,i} =0, D_{1}'=\frac{\bar{L}^2}{b}, \tsk = \ns{\xk-\wk}$,
	$\rho'=p+p\gamma-\gamma, A_2'=0, B_2'=(1-p)\eta^2\gamma^{-1},C_2'=0, D_2'=\eta^2$, and $\forall\gamma>0$.
	Thus, according to Theorem \ref{thm:diana-type}, $g^k$ (see Line \ref{line:gk-diana-lsvrg} of Algorithm \ref{alg:diana-lsvrg})  satisfies the unified Assumption \ref{asp:boge} with
	\begin{align*}
	&A_1 =0, \qquad
	B_1 =1, \qquad 
	C_1  =  0, \\ 
	&D_1 =\frac{1+\omega}{m}, \qquad
	\sk = \frac{\bar{L}^2}{b}\ns{\xk-\wk} + \frac{\omega}{(1+\omega)m}
	\summ \ns{\nabla f_i(\xk) -\hi}, \\
	&\rho =\min\{p+p\gamma-\gamma-\tau,~ 2\alpha -(1-\alpha)\beta^{-1} -\alpha^2-\tau \},\\
	&A_2 =0 \qquad
	B_2 =(1-p)\bar{L}^2\eta^2\gamma^{-1} b^{-1} +B, \qquad 
	C_2  = 0,
	\end{align*}
	where 
	$B:=\frac{\omega(1+\beta)L^2\eta^2}{1+\omega} + \bar{L}^2\eta^2 b^{-1}$,
	$\tau:=\alpha^2 \omega + \frac{(1+\omega)B}{m}$,
	and $\forall \gamma, \beta>0$.
\end{lemma}

\begin{proofof}{Lemma \ref{lem:diana-lsvrg}}
	If the local stochastic gradient estimator $\gi=\frac{1}{b} \sum_{j\in I_b} (\nabla f_{i,j}(x^k)- \nabla f_{i,j}(w^k)) +\nabla f_i(w^k)$ (see Line \ref{line:localgrad-diana-lsvrg} of Algorithm \ref{alg:diana-lsvrg}), we show the following equations:
	\begin{align}
	\Ek[\gi] &= \Ek\left[\frac{1}{b} \sum_{j\in I_b} (\nabla f_{i,j}(x^k)- \nabla f_{i,j}(w^k)) +\nabla f_i(w^k)\right] \notag\\
	&= \nabla f_i(\xk) - \nabla f_i(\wk) + \nabla f_i(\wk) = \nabla f_i(\xk) \label{eq:diana-svrg0}
	\end{align}
	and 
	\begin{align}
	\Ek[\ns{\gi}] & = \Ek\left[\ns{\gi-\nabla f_i(\xk)}\right] + \ns{\nabla f_i(\xk) } \notag\\
	&= \Ek\left[\nsB{\frac{1}{b} \sum_{j\in I_b} (\nabla f_{i,j}(x^k)- \nabla f_{i,j}(w^k)) +\nabla f_i(w^k) - \nabla f_i(\xk)} \right]  + \ns{\nabla f_i(\xk) }  \notag\\
	&= \frac{1}{b^2} \Ek\left[\nsB{\sum_{j\in I_b} \left((\nabla f_{i,j}(\xk)- \nabla f_{i,j}(\wk)) -(\nabla f_i(\xk) - \nabla f_i(\wk)) \right)} \right]  + \ns{\nabla f_i(\xk) }  \notag\\
	&= \frac{1}{b} \Ek\left[\nsB{(\nabla f_{i,j}(\xk)- \nabla f_{i,j}(\wk)) -(\nabla f_i(\xk) - \nabla f_i(\wk))} \right]  + \ns{\nabla f_i(\xk) }  \notag\\
	&\leq \frac{1}{b} \Ek\left[\nsB{\nabla f_{i,j}(\xk)- \nabla f_{i,j}(\wk)} \right]  + \ns{\nabla f_i(\xk) }  \label{eq:diana-var}\\
	& \leq  \frac{\bL^2}{b} \ns{\xk-\wk} + \ns{\nabla f_i(\xk)}, \label{eq:diana-svrg1}
	\end{align}
	where \eqref{eq:diana-var} uses the fact $\E[\ns{x-\E[x]}]\leq \E[\ns{x}]$, and the last inequality uses Assumption \ref{asp:avgsmooth-fed} (i.e., \eqref{eq:avgsmooth-fed}).
	Now, we define $\tsk:=\ns{\xk-\wk}$ and obtain
	\begin{align}
	&\Ek[\tskn] \notag\\
	& := \Ek[\ns{\xkn-\wkn}]  \notag\\
	&= p\Ek[\ns{\xkn-\xk}] + (1-p)\Ek[\ns{\xkn-\wk}]  \label{eq:diana-useprob}\\
	&=p\eta^2\Ek[\ns{\gk}] + (1-p)\Ek[\ns{\xk-\eta \gk-\wk}]   \label{eq:diana-useupdate} \\
	&=p\eta^2\Ek[\ns{\gk}] + (1-p)\Ek[\ns{\xk-\wk} + \ns{\eta \gk} -2\inner{\xk-\wk}{\eta\gk}]  \notag\\
	&= p\eta^2\Ek[\ns{\gk}] + (1-p)\ns{\xk-\wk} + (1-p)\eta^2\Ek[\ns{\gk}] -2(1-p)\inner{\xk-\wk}{\eta \nabla f(\xk)} \notag \\
	&\leq \eta^2\Ek[\ns{\gk}] + (1-p)\ns{\xk-\wk} +(1-p)\gamma\ns{\xk-\wk} + \frac{1-p}{\gamma}\ns{\eta \nabla f(\xk)} \label{eq:diana-useyoung} \\
	&= \eta^2\Ek[\ns{\gk}] + (1-p)(1+\gamma)\ns{\xk-\wk}  
	+ \frac{(1-p)\eta^2}{\gamma}\ns{\nabla f(\xk)} \notag\\
	&= (1-(p+p\gamma-\gamma)) \tsk + \frac{(1-p)\eta^2}{\gamma}\ns{\nabla f(\xk)}
	+ \eta^2\Ek[\ns{\gk}], 
	\label{eq:diana-svrg2}
	\end{align}
	where \eqref{eq:diana-useprob} uses Line \ref{line:w_prob-diana-lsvrg} of Algorithm \ref{alg:diana-lsvrg}, \eqref{eq:diana-useupdate} uses Line \ref{line:update-diana-lsvrg} of Algorithm \ref{alg:diana-lsvrg}, \eqref{eq:diana-useyoung} uses Young's inequality for $\forall \gamma>0$, and \eqref{eq:diana-svrg2} follows from the definition $\tsk:=\ns{\xk-\wk}$.
	
	Now, according to \eqref{eq:diana-svrg1} and \eqref{eq:diana-svrg2}, we know the local stochastic gradient estimator $\gi$ (see Line \ref{line:localgrad-diana-lsvrg} of Algorithm \ref{alg:diana-lsvrg}) satisfies \eqref{eq:gi1} and \eqref{eq:gi2} with 
	\begin{align*}
	&A_{1,i}=0,\quad  B_{1,i}=1,\quad 
	D_{1}'=\frac{\bar{L}^2}{b},\quad \tsk = \ns{\xk-\wk},\quad
	C_{1,i} =0, \\
	&\rho'=p+p\gamma-\gamma,\quad
	A_2'=0,\quad
	B_2'=(1-p)\eta^2\gamma^{-1},\quad
	D_2'=\eta^2,  \quad
	C_2'=0,\quad \forall\gamma>0.
	\end{align*}
	Thus, according to Theorem \ref{thm:diana-type}, $g^k$ (see Line \ref{line:gk-diana-lsvrg} of Algorithm \ref{alg:diana-lsvrg})  satisfies the unified Assumption \ref{asp:boge} with
	\begin{align*}
	&A_1 =0, \qquad
	B_1 =1, \qquad 
	C_1  =  0, \\ 
	&D_1 =\frac{1+\omega}{m}, \qquad
	\sk = \frac{\bar{L}^2}{b}\ns{\xk-\wk} + \frac{\omega}{(1+\omega)m}
	\summ \ns{\nabla f_i(\xk) -\hi}, \\
	&\rho =\min\{p+p\gamma-\gamma-\tau,~ 2\alpha -(1-\alpha)\beta^{-1} -\alpha^2-\tau \},\\
	&A_2 =0 \qquad
	B_2 =(1-p)\bar{L}^2\eta^2\gamma^{-1} b^{-1} +B, \qquad 
	C_2  = 0,
	\end{align*}
	where 
	$B:=\frac{\omega(1+\beta)L^2\eta^2}{1+\omega} + \bar{L}^2\eta^2 b^{-1}$,
	$\tau:=\alpha^2 \omega + \frac{(1+\omega)B}{m}$,
	and $\forall \gamma, \beta>0$.
\end{proofof}

\begin{corollary}[DIANA-LSVRG]\label{cor:diana-lsvrg}
	Suppose that Assumption \ref{asp:lsmooth-diana} and \ref{asp:avgsmooth-fed} hold. 
	Let stepsize 
	$$\eta \leq  \frac{1}{L+L(1+\omega)B'm^{-1}b^{-1}\rho^{-1}},$$ 
	then the number of iterations performed by DIANA-LSVRG (Algorithm \ref{alg:diana-lsvrg}) to find an $\epsilon$-solution of nonconvex federated problem \eqref{eq:prob-fed} with \eqref{prob-fed:finite}, i.e. a point $\hx$ such that $\E[\n{\nabla f(\hx)}] \leq \epsilon$, can be bounded by	
	\begin{align}
	K = \frac{8\fgap L}{\epsilon^2}  \left(1+\frac{(1+\omega)B'}{mb\rho}\right),
	\end{align}
	where 
	$B' :=  (1-p)\bar{L}^2\eta^2\gamma^{-1} + Bb$,
	$\rho :=\min\{p+p\gamma-\gamma-\tau,~ 2\alpha -(1-\alpha)\beta^{-1} -\alpha^2-\tau \}$,
	$\tau:=\alpha^2 \omega + \frac{(1+\omega)B}{m}$,
	$B:=\frac{\omega(1+\beta)L^2\eta^2}{1+\omega} + \bar{L}^2\eta^2 b^{-1}$,
	and $\forall \gamma, \beta>0$.
\end{corollary}

\begin{proofof}{Corollary \ref{cor:diana-lsvrg}}
	According to our unified Theorem \ref{thm:main}, if the stepsize is chosen as 
	\begin{align}
	\eta &\leq \min\left\{ \frac{1}{LB_1+LD_1B_2\rho^{-1}}, \sqrt{\frac{\ln 2}{(LA_1 + LD_1A_2\rho^{-1})K}}, \frac{\epsilon^2}{2L(C_1+D_1C_2\rho^{-1})} \right\}  \notag\\
	&= \frac{1}{L+L(1+\omega)B'm^{-1}b^{-1}\rho^{-1}}
	\end{align}
	since $B_1=1, D_1 =\frac{1+\omega}{m}, B_2=\frac{B'}{b}, B'=(1-p)\bar{L}^2\eta^2\gamma^{-1} + Bb$ and $A_1=A_2=C_1=C_2=0$
	according to Lemma \ref{lem:diana-lsvrg},
	then the number of iterations performed by DIANA-LSVRG (Algorithm \ref{alg:diana-lsvrg}) to find an $\epsilon$-solution of problem \eqref{eq:prob-fed} with \eqref{prob-fed:finite} can be bounded by
	\begin{align}
	K &= \frac{8\fgapp L}{\epsilon^2} \max \left\{ B_1+D_1B_2\rho^{-1}, \frac{12\fgapp (A_1+D_1 A_2\rho^{-1})}{\epsilon^2}, \frac{2(C_1+D_1C_2\rho^{-1})}{\epsilon^2}\right\}  \notag\\
	&= \frac{8\fgap L}{\epsilon^2}  \left(1+\frac{(1+\omega)B'}{mb\rho}\right)
	\end{align}
	since $B_1=1, D_1 =\frac{1+\omega}{m}, B_2=\frac{B'}{b}, B'=(1-p)\bar{L}^2\eta^2\gamma^{-1} + Bb, A_1=A_2=C_1=C_2=0$ and 
	$\fgapp:= f(x^0) - f^* + 2^{-1}L\eta^2D_1\rho^{-1} \sigma_0^2 = f(x^0) - f^* =\fgap$ by letting $\sigma_0^2 =\frac{\omega}{(1+\omega)m}
	\summ \ns{\nabla f_i(x^0) -h^0}=0$.
\end{proofof}

\subsubsection{DIANA-SAGA method}
In this section, we show that if the parallel workers use SAGA for computing their local  gradient $\widetilde{g}_i^k$, then $g^k$ (see Line \ref{line:gk-diana-saga} of Algorithm \ref{alg:diana-saga}) satisfies the unified Assumption \ref{asp:boge}.

\begin{algorithm}[htb]
	\caption{DIANA-SAGA}
	\label{alg:diana-saga}
	\begin{algorithmic}[1]
		\REQUIRE ~
		initial point  $x^0$, $\{w_{i,j}^0\}_{i,j=1}^{m,n}$, $\{h_i^0\}_{i=1}^m$, $h^0=\frac{1}{m}\sum_{i=1}^{m}h_i^0$, parameters $b, \eta_k, \alpha_k$
		\FOR {$k=0,1,2,\ldots$}
		\STATE {\bf{for all machines $i= 1,2,\ldots,m$ do in parallel}}
		\STATE \quad Compute local stochastic gradient $\widetilde{g}_i^k=\frac{1}{b} \sum_{j\in I_b} (\nabla f_{i,j}(x^k)-  \nabla f_{i,j}(w_{i,j}^k)) +\frac{1}{n}\sum_{j=1}^{n}\nabla f_{i,j}(w_{i,j}^k)$  \label{line:localgrad-diana-saga}
		\STATE \quad Compress shifted local gradient $\widehat{\Delta}_i^k= \cC_i^k(\widetilde{g}_i^k- h_i^k)$ and send $\widehat{\Delta}_i^k$ to the server
		\STATE \quad Update local shift $h_i^{k+1}=h_i^k+\alpha_k \cC_i^k(\widetilde{g}_i^k - h_i^k)$
		\STATE \quad $w_{i,j}^{k+1} = \begin{cases}
		x^k & \text{for~~}  j\in I_b\\
		w_{i,j}^k &\text{for~~} j\notin I_b
		\end{cases}$ \label{line:w_prob-diana-saga}
		\STATE {\bf{end for}}
		\STATE Aggregate received compressed gradient information
		$g^k = h^k + \frac{1}{m}\sum_{i=1}^m \widehat{\Delta}_i^k$
		\label{line:gk-diana-saga}
		\STATE $x^{k+1} = x^k - \eta_k g^k$ \label{line:update-diana-saga}
		\STATE $h^{k+1} = h^k + \alpha_k  \frac{1}{m}\sum_{i=1}^m  \widehat{\Delta}_i^k$
		\ENDFOR
	\end{algorithmic}
\end{algorithm}

\begin{lemma}[DIANA-SAGA]\label{lem:diana-saga}
	Let the local stochastic gradient estimator $\widetilde{g}_i^k=\frac{1}{b} \sum_{j\in I_b} (\nabla f_{i,j}(x^k)-  \nabla f_{i,j}(w_{i,j}^k)) + \frac{1}{n}\sum_{j=1}^{n}\nabla f_{i,j}(w_{i,j}^k)$  (see Line \ref{line:localgrad-diana-saga} of Algorithm \ref{alg:diana-saga}), then we know that $\widetilde{g}_i^k$ satisfies \eqref{eq:gi1-diff} and \eqref{eq:gi2-diff} with $A_{1,i}=0, B_{1,i}=1, C_{1,i} =0, D_{1,i}=\frac{\bar{L}^2}{b}, \ski = \frac{1}{n}\sum_{j=1}^{n}\ns{\xk-w_{i,j}^k}$,
	$\rho_i=\frac{b}{n}+\frac{b}{n}\gamma-\gamma, A_{2,i}=0, B_{2,i}=(1-\frac{b}{n})\eta^2\gamma^{-1},C_{2,i}=0, D_{2,i}=\eta^2$, and $\forall\gamma>0$.
	Thus, according to Theorem \ref{thm:diana-type-diff}, $g^k$ (see Line \ref{line:gk-diana-saga} of Algorithm \ref{alg:diana-saga})  satisfies the unified Assumption \ref{asp:boge} with
	\begin{align*}
	&A_1 =0, \qquad
	B_1 =1, \qquad 
	C_1  =  0, \\ 
	&D_1 =\frac{1+\omega}{m}, \qquad
	\sk = \frac{1}{mn}\summ \sum_{j=1}^{m}\frac{\bar{L}^2}{b}\ns{\xk-w_{i,j}^k} + \frac{\omega}{(1+\omega)m}
	\summ \ns{\nabla f_i(\xk) -\hi}, \\
	&\rho =\min\left\{\frac{b}{n}+\frac{b}{n}\gamma-\gamma-\tau,~ 2\alpha -(1-\alpha)\beta^{-1} -\alpha^2-\tau \right\},\\
	&A_2 =0 \qquad
	B_2 =(1-\frac{b}{n})\bar{L}^2\eta^2\gamma^{-1} b^{-1} +B, \qquad 
	C_2  = 0,
	\end{align*}
	where 
	$B:=\frac{\omega(1+\beta)L^2\eta^2}{1+\omega} + \bar{L}^2\eta^2 b^{-1}$,
	$\tau:=\alpha^2 \omega + \frac{(1+\omega)B}{m}$,
	and $\forall \gamma, \beta>0$.
\end{lemma}

\begin{proofof}{Lemma \ref{lem:diana-saga}}
	If the local stochastic gradient estimator $\gi=\frac{1}{b} \sum_{j\in I_b} (\nabla f_{i,j}(x^k)- \nabla f_{i,j}(w_{i,j}^k)) + \frac{1}{n}\sum_{j=1}^{n}\nabla f_{i,j}(w_{i,j}^k)$(see Line \ref{line:localgrad-diana-saga} of Algorithm \ref{alg:diana-saga}), we show the following equations:
	\begin{align}
	\Ek[\gi] &= \Ek\left[\frac{1}{b} \sum_{j\in I_b} (\nabla f_{i,j}(x^k)- \nabla f_{i,j}(w_{i,j}^k)) + \frac{1}{n}\sum_{j=1}^{n}\nabla f_{i,j}(w_{i,j}^k)\right] \notag\\
	&= \nabla f_i(\xk) - \frac{1}{n}\sum_{j=1}^{n}\nabla f_{i,j}(w_{i,j}^k) + \frac{1}{n}\sum_{j=1}^{n}\nabla f_{i,j}(w_{i,j}^k) \notag\\
	&= \nabla f_i(\xk) \label{eq:diana-saga0}
	\end{align}
	and 
	\begin{align}
	&\Ek[\ns{\gi}] \notag\\
	& = \Ek\left[\ns{\gi-\nabla f_i(\xk)}\right] + \ns{\nabla f_i(\xk) } \notag\\
	&= \Ek\left[\nsB{\frac{1}{b} \sum_{j\in I_b} (\nabla f_{i,j}(x^k)- \nabla f_{i,j}(w_{i,j}^k)) + \frac{1}{n}\sum_{j=1}^{n}\nabla f_{i,j}(w_{i,j}^k) - \nabla f_i(\xk)} \right]  + \ns{\nabla f_i(\xk) }  \notag\\
	&= \frac{1}{b^2} \Ek\left[\nsB{\sum_{j\in I_b} \left((\nabla f_{i,j}(\xk)- \nabla f_{i,j}(w_{i,j}^k)) -\left(\frac{1}{n}\sum_{j=1}^{n}\nabla f_{i,j}(\xk)- \frac{1}{n}\sum_{j=1}^{n}\nabla f_{i,j}(w_{i,j}^k)\right) \right)} \right]  \notag\\
	&\qquad \quad + \ns{\nabla f_i(\xk) }  \notag\\
	&= \frac{1}{b} \Ek\left[\nsB{(\nabla f_{i,j}(\xk)- \nabla f_{i,j}(w_{i,j}^k)) -\left(\frac{1}{n}\sum_{j=1}^{n}\nabla f_{i,j}(\xk)- \frac{1}{n}\sum_{j=1}^{n}\nabla f_{i,j}(w_{i,j}^k)\right)} \right]  \notag\\
	&\qquad \quad + \ns{\nabla f_i(\xk) }  \notag\\
	&\leq \frac{1}{b} \Ek\left[\nsB{\nabla f_{i,j}(\xk)- \nabla f_{i,j}(w_{i,j}^k)} \right]  + \ns{\nabla f_i(\xk) }  \label{eq:diana-var1}\\
	& \leq  \frac{\bL^2}{b}\frac{1}{n}\sum_{j=1}^{n}\ns{\xk-w_{i,j}^k} + \ns{\nabla f_i(\xk)}, \label{eq:diana-saga1}
	\end{align}
	where \eqref{eq:diana-var1} uses the fact $\E[\ns{x-\E[x]}]\leq \E[\ns{x}]$, and the last inequality uses Assumption \ref{asp:avgsmooth-saga-fed} (i.e., \eqref{eq:avgsmooth-saga-fed}).
	Now, we define $\ski:=\frac{1}{n}\sum_{j=1}^{n}\ns{\xk-w_{i,j}^k}$ and obtain
	
	\begin{align}
	&\Ek[\skin] \notag\\
	& := \EkB{\frac{1}{n}\sum_{j=1}^{n}\ns{\xkn-\wijkn}}  \notag\\
	&= \EkB{\frac{1}{n}\sum_{j=1}^{n}\frac{b}{n}\ns{\xkn-\xk}+ \frac{1}{n}\sum_{j=1}^{n}\left(1-\frac{b}{n}\right)\ns{\xkn-\wijk}} \label{eq:diana-useprob1}\\
	&= \frac{b}{n}\Ek\eta^2\ns{\gk} 
	+ \left(1-\frac{b}{n}\right)\EkB{\frac{1}{n}\sum_{j=1}^{n} \ns{\xk-\eta\gk-\wijk} }  \label{eq:diana-useupdate1} \\
	&= \frac{b\eta^2}{n}\Ek\ns{\gk} 
	+ \left(1-\frac{b}{n}\right)\EkB{\frac{1}{n}\sum_{j=1}^{n} \left(\ns{\xk-\wijk} + \ns{\eta \gk} -2\inner{\xk-\wijk}{\eta\gk}\right)}  \notag\\
	&= \eta^2\Ek\ns{\gk} 
	+ \left(1-\frac{b}{n}\right)\frac{1}{n}\sum_{j=1}^{n} \ns{\xk-\wijk} 
	+ 2\left(1-\frac{b}{n}\right)\frac{1}{n}\sum_{j=1}^{n}\inner{\xk-\wijk}{\eta \nabla f(\xk)}  \notag\\
	&\leq \eta^2\Ek\ns{\gk} 
	+ \left(1-\frac{b}{n}\right)\frac{1}{n}\sum_{j=1}^{n} \ns{\xk-\wijk} \notag\\
	&\qquad \qquad 
	+ \left(1-\frac{b}{n}\right)\frac{1}{n}\sum_{j=1}^{n}\left(\gamma\ns{\xk-\wijk} +\frac{\eta^2}{\gamma}\ns{\nabla f(\xk)}\right) \label{eq:diana-useyoung1} 
	\end{align}
	\begin{align}
	&= \left(1-(\frac{b}{n}+\frac{b}{n}\gamma-\gamma)\right) \ski + \frac{(1-\frac{b}{n})\eta^2}{\gamma}\ns{\nabla f(\xk)}
	+ \eta^2\Ek[\ns{\gk}], 
	\label{eq:diana-saga2}
	\end{align}
	where \eqref{eq:diana-useprob1} uses Line \ref{line:w_prob-diana-saga} of Algorithm \ref{alg:diana-saga}, \eqref{eq:diana-useupdate1} uses Line \ref{line:update-diana-saga} of Algorithm \ref{alg:diana-saga}, \eqref{eq:diana-useyoung1} uses Young's inequality for $\forall \gamma>0$, and \eqref{eq:diana-saga2} follows from the definition $\ski:=\frac{1}{n}\sum_{j=1}^{n}\ns{\xk-w_{i,j}^k}$.
	
	Now, according to \eqref{eq:diana-saga1} and \eqref{eq:diana-saga2}, we know the local stochastic gradient estimator $\gi$ (see Line \ref{line:localgrad-diana-saga} of Algorithm \ref{alg:diana-saga}) satisfies \eqref{eq:gi1-diff} and \eqref{eq:gi2-diff} with 
	\begin{align*}
	&A_{1,i}=0,\quad  B_{1,i}=1,\quad 
	D_{1,i}=\frac{\bar{L}^2}{b},\quad 
	\ski = \frac{1}{n}\sum_{j=1}^{n}\ns{\xk-w_{i,j}^k},\quad
	C_{1,i} =0, \\
	&\rho_i=\frac{b}{n}+\frac{b}{n}\gamma-\gamma,\quad
	A_{2,i}=0,\quad
	B_{2,i}=(1-\frac{b}{n})\eta^2\gamma^{-1},\quad
	D_{2,i}=\eta^2,  \quad
	C_{2,i}=0,\quad \forall\gamma>0.
	\end{align*}
	Thus, according to Theorem \ref{thm:diana-type-diff}, $g^k$ (see Line \ref{line:gk-diana-saga} of Algorithm \ref{alg:diana-saga}) satisfies the unified Assumption \ref{asp:boge} with
	\begin{align*}
	&A_1 =0, \qquad
	B_1 =1, \qquad 
	C_1  =  0, \\ 
	&D_1 =\frac{1+\omega}{m}, \qquad
	\sk = \frac{1}{mn}\summ \sum_{j=1}^{m}\frac{\bar{L}^2}{b}\ns{\xk-w_{i,j}^k} + \frac{\omega}{(1+\omega)m}
	\summ \ns{\nabla f_i(\xk) -\hi}, \\
	&\rho =\min\left\{\frac{b}{n}+\frac{b}{n}\gamma-\gamma-\tau,~ 2\alpha -(1-\alpha)\beta^{-1} -\alpha^2-\tau \right\},\\
	&A_2 =0 \qquad
	B_2 =(1-\frac{b}{n})\bar{L}^2\eta^2\gamma^{-1} b^{-1} +B, \qquad 
	C_2  = 0,
	\end{align*}
	where 
	$B:=\frac{\omega(1+\beta)L^2\eta^2}{1+\omega} + \bar{L}^2\eta^2 b^{-1}$,
	$\tau:=\alpha^2 \omega + \frac{(1+\omega)B}{m}$,
	and $\forall \gamma, \beta>0$.
\end{proofof}

\begin{corollary}[DIANA-SAGA]\label{cor:diana-saga}
	Suppose that Assumption \ref{asp:lsmooth-diana} and \ref{asp:avgsmooth-saga-fed} hold. 
	Let stepsize 
	$$\eta \leq  \frac{1}{L+L(1+\omega)B'm^{-1}b^{-1}\rho^{-1}},$$ 
	then the number of iterations performed by DIANA-SAGA (Algorithm \ref{alg:diana-saga}) to find an $\epsilon$-solution of nonconvex federated problem \eqref{eq:prob-fed} with \eqref{prob-fed:finite}, i.e. a point $\hx$ such that $\E[\n{\nabla f(\hx)}] \leq \epsilon$, can be bounded by	
	\begin{align}
	K = \frac{8\fgap L}{\epsilon^2}  \left(1+\frac{(1+\omega)B'}{mb\rho}\right),
	\end{align}
	where 
	$B' :=  (1-\frac{b}{n})\bar{L}^2\eta^2\gamma^{-1} + Bb$,
	$\rho :=\min\{\frac{b}{n}+\frac{b}{n}\gamma-\gamma-\tau,~ 2\alpha -(1-\alpha)\beta^{-1} -\alpha^2-\tau \}$,
	$\tau:=\alpha^2 \omega + \frac{(1+\omega)B}{m}$,
	$B:=\frac{\omega(1+\beta)L^2\eta^2}{1+\omega} + \bar{L}^2\eta^2 b^{-1}$,
	and $\forall \gamma, \beta>0$.
\end{corollary}

\begin{proofof}{Corollary \ref{cor:diana-saga}}
	According to our unified Theorem \ref{thm:main}, if the stepsize is chosen as 
	\begin{align}
	\eta &\leq \min\left\{ \frac{1}{LB_1+LD_1B_2\rho^{-1}}, \sqrt{\frac{\ln 2}{(LA_1 + LD_1A_2\rho^{-1})K}}, \frac{\epsilon^2}{2L(C_1+D_1C_2\rho^{-1})} \right\}  \notag\\
	&= \frac{1}{L+L(1+\omega)B'm^{-1}b^{-1}\rho^{-1}}
	\end{align}
	since $B_1=1, D_1 =\frac{1+\omega}{m}, B_2=\frac{B'}{b}, B'=(1-\frac{b}{n})\bar{L}^2\eta^2\gamma^{-1} + Bb$ and $A_1=A_2=C_1=C_2=0$
	according to Lemma \ref{lem:diana-saga},
	then the number of iterations performed by DIANA-SAGA (Algorithm \ref{alg:diana-saga}) to find an $\epsilon$-solution of problem \eqref{eq:prob-fed} with \eqref{prob-fed:finite} can be bounded by
	\begin{align}
	K &= \frac{8\fgapp L}{\epsilon^2} \max \left\{ B_1+D_1B_2\rho^{-1}, \frac{12\fgapp (A_1+D_1 A_2\rho^{-1})}{\epsilon^2}, \frac{2(C_1+D_1C_2\rho^{-1})}{\epsilon^2}\right\}  \notag\\
	&= \frac{8\fgap L}{\epsilon^2}  \left(1+\frac{(1+\omega)B'}{mb\rho}\right)
	\end{align}
	since $B_1=1, D_1 =\frac{1+\omega}{m}, B_2=\frac{B'}{b}, B'=(1-\frac{b}{n})\bar{L}^2\eta^2\gamma^{-1} + Bb, A_1=A_2=C_1=C_2=0$ and 
	$\fgapp:= f(x^0) - f^* + 2^{-1}L\eta^2D_1\rho^{-1} \sigma_0^2 = f(x^0) - f^* =\fgap$ by letting $\sigma_0^2 =\frac{\omega}{(1+\omega)m}
	\summ \ns{\nabla f_i(x^0) -h^0}=0$.
\end{proofof}

\newpage
\section{Better Convergence for Nonconvex Optimization under PL Condition}
\label{sec:pl}

In this section, we provide the detailed better convergence rates and proofs under PL condition (i.e., Assumption \ref{asp:pl}) for some specific methods in the single machine case (i.e., $m=1$) of nonconvex federated problem \eqref{eq:prob-fed} which reduces to the standard nonconvex problem \eqref{eq:prob} with online form \ref{prob:exp} or finite-sum form \eqref{prob:finite}, i.e., 
\begin{equation*}
\min_{x\in \R^d}   f(x), \text{~~where~}  f(x) := \E_{\zeta\sim \cD}[f(x,\zeta)],  
\text{~~~or~~~}  f(x) := \frac{1}{n}\sum_{i=1}^n{f_i(x)}.
\end{equation*}

In the following, we prove that some specific methods, i.e., GD, SGD, L-SVRG and SAGA satisfy our unified Assumption \ref{asp:boge} and thus can be captured by our unified analysis.
Then, we plug their corresponding parameters (i.e., specific values for $A_1, A_2, B_1, B_2, C_1,C_2,D_1,\rho$) into our unified Theorem \ref{thm:main-pl-dec} or \ref{thm:main-pl} under PL condition to obtain the detailed better convergence rates for these methods. 

\subsection{GD method under PL condition}

\begin{corollary}[GD under PL condition]\label{cor:gd-pl}
	Suppose that Assumption \ref{asp:lsmooth} and \ref{asp:pl} hold. Let stepsize $\eta\leq \frac{1}{L}$, then
	the number of iterations performed by GD (Algorithm \ref{alg:gd}) to find an $\epsilon$-solution of nonconvex problem \eqref{eq:prob}, i.e. a point $x^K$ such that $\E[f(x^K)-f^*] \leq \epsilon$, can be bounded by
	$$
	K = \frac{L}{\mu}\log \frac{2\fgap}{\epsilon}.
	$$
	Note that it recovers the previous result for GD under PL condition given by \citep{polyak1963gradient,karimi2016linear}.
\end{corollary}
\begin{proofof}{Corollary \ref{cor:gd-pl}}
	According to Theorem \ref{thm:main-pl}, if the stepsize is chosen as 
	\begin{align*}
	\eta_k \equiv \eta &\leq \min\left\{ \frac{1}{LB_1+2LD_1B_2\rho^{-1} + (L A_1 + 2LD_1A_2 \rho^{-1})\mu^{-1}},~~ \frac{\mu\epsilon}{LC_1+2LD_1C_2\rho^{-1}}  \right\} \\
	&= \frac{1}{L}
	\end{align*}
	since $B_1=1, A_1=C_1=D_1=0$ according to Lemma \ref{lem:gd},
	then the number of iterations performed by Algorithm \ref{alg:1} to find an $\epsilon$-solution of problem \eqref{eq:prob} can be bounded by
	\begin{align*}
	K &= \max \left\{ B_1+2D_1B_2\rho^{-1} + (L A_1 + 2LD_1A_2 \rho^{-1})\mu^{-1},~~ \frac{C_1+2D_1C_2\rho^{-1}}{\mu \epsilon}\right\} \frac{L}{\mu}\log \frac{2\fgapp}{\epsilon}  \\
	&= \frac{L}{\mu}\log \frac{2\fgap}{\epsilon}
	\end{align*}
	since $B_1=1, A_1=C_1=D_1=0, \sigma_0^2 = 0$, and $\fgapp:= f(x^0) - f^* + L\eta^2D_1\rho^{-1} \sigma_0^2 = f(x^0) - f^* =\fgap$.
\end{proofof}

\subsection{SGD method under PL condition}

\begin{corollary}[SGD under PL condition]\label{cor:sgd-pl}
	Suppose that Assumption \ref{asp:lsmooth} and \ref{asp:pl} hold, and the gradient estimator $g^k$ in Algorithm \ref{alg:sgd} satisfies Assumption \ref{asp:es}. 
	Let stepsize 
	$$\eta_{k} = \begin{cases}
	\eta   & \text{if~~}  k\leq \frac{K}{2}\\
	\frac{2\eta}{2+(k-\frac{K}{2})\mu\eta} &\text{if~~}  k>\frac{K}{2}
	\end{cases}, 	\text{~~~~where~~}
	\eta \leq \frac{1}{LB + L A \mu^{-1}},
	$$
	then the number of iterations performed by SGD (Algorithm \ref{alg:sgd})  for finding an $\epsilon$-solution of nonconvex problem \eqref{eq:prob} with \eqref{prob:exp} or \eqref{prob:finite}, i.e. a point $x^K$ such that $\E[f(x^K)-f^*] \leq \epsilon$,  can be bounded by 
	$$
	K = \max \left\{\frac{2L( B + A \mu^{-1})}{\mu}\log \frac{2\fgap}{\epsilon},~~ \frac{10LC}{\mu^2 \epsilon}\right\}.
	$$
	Note that it recovers the recent result for SGD under PL condition given by \citep{khaled2020better}.
\end{corollary}
\begin{proofof}{Corollary \ref{cor:sgd-pl}}
	According to Theorem \ref{thm:main-pl-dec}, if the stepsize is chosen as 
	\begin{align*}
	\eta_{k} = \begin{cases}
	\eta   & \text {if~~}  k\leq \frac{K}{2}\\
	\frac{2\eta}{2+(k-\frac{K}{2})\mu\eta} &\text {if~~}  k>\frac{K}{2}
	\end{cases},
	\end{align*}
	where 
	$$\eta \leq \frac{1}{LB_1+2LD_1B_2\rho^{-1} + (L A_1 + 2LD_1A_2 \rho^{-1})\mu^{-1}} 
	= \frac{1}{LB+LA\mu^{-1}}$$ 
	since $A_1=A, B_1=B, C_1= C, D_1=0 $ according to Lemma \ref{lem:sgd},
	then the number of iterations performed by Algorithm \ref{alg:1} to find an $\epsilon$-solution of problem \eqref{eq:prob} with \eqref{prob:exp} can be bounded by
	\begin{align*}
	K &= \max \left\{\frac{2L( B_1+2D_1B_2\rho^{-1} +(L A_1 + 2LD_1A_2 \rho^{-1})\mu^{-1})}{\mu}\log \frac{2\fgapp}{\epsilon},~~ \frac{10L(C_1+2D_1C_2\rho^{-1})}{\mu^2 \epsilon}\right\} \\
	&=\max \left\{\frac{2L( B + A \mu^{-1})}{\mu}\log \frac{2\fgap}{\epsilon},~~ \frac{10LC}{\mu^2 \epsilon}\right\}
	\end{align*}
	since $A_1=A, B_1=B, C_1= C, D_1=0, \sigma_0^2 = 0$, and $\fgapp:= f(x^0) - f^* + L\eta^2D_1\rho^{-1} \sigma_0^2 = f(x^0) - f^* =\fgap$.
\end{proofof}

\subsection{L-SVRG method under PL condition}

\begin{corollary}[L-SVRG under PL condition]\label{cor:lsvrg-pl}
	Suppose that Assumption \ref{asp:avgsmooth} and \ref{asp:pl} hold.
	Let stepsize $\eta \leq \frac{1}{L(1+3b^{-1/3}p^{-2/3})}$, 
	then the number of iterations performed by L-SVRG (Algorithm \ref{alg:lsvrg}) for finding an $\epsilon$-solution of nonconvex problem \eqref{eq:prob} with \eqref{prob:finite}, i.e. a point $x^K$ such that $\E[f(x^K)-f^*] \leq \epsilon$, can be bounded by 
	$$
	K = \left(1+ \frac{3}{b^{1/3}p^{2/3}}\right)\frac{L}{\mu}\log\frac{2\fgap}{\epsilon}.
	$$
	In particular, we have
	\begin{enumerate}
		\item let minibatch size $b=1$ and probability $p=\frac{1}{n}$, then the number of iterations $K = \frac{4n^{2/3}L}{\mu}\log\frac{2\fgap}{\epsilon}$.
		\item let minibatch size $b=n^{2/3}$ and probability $p=\frac{1}{n^{1/3}}$, then the number of iterations $K =\frac{4L}{\mu}\log\frac{2\fgap}{\epsilon}$, but each iteration costs $n^{2/3}$ due to minibatch size $b=n^{2/3}$.
	\end{enumerate}
	Note that it recovers the previous result for SVRG under PL condition given by \citep{reddi2016proximal,li2018simple}, but they studied the standard SVRG form \citep{johnson2013accelerating} not this simpler loopless version.
\end{corollary}

\begin{proofof}{Corollary \ref{cor:lsvrg-pl}}
	According to Theorem \ref{thm:main-pl}, the stepsize should be chosen as 
	\begin{align}
	\eta_k \equiv \eta &\leq \min\left\{ \frac{1}{LB_1+2LD_1B_2\rho^{-1} + (L A_1 + 2LD_1A_2 \rho^{-1})\mu^{-1}},~~ \frac{\mu\epsilon}{LC_1+2LD_1C_2\rho^{-1}}  \right\} \notag\\
	&= \frac{1}{LB_1+2LD_1B_2\rho^{-1}}
	\end{align}
	since $A_1=A_2=C_1=C_2=0$ according to Lemma \ref{lem:lsvrg}.
	Then the number of iterations performed by L-SVRG (Algorithm \ref{alg:lsvrg}) to find an $\epsilon$-solution of problem \eqref{eq:prob} with \eqref{prob:finite} can be bounded by
	\begin{align}
	K &= \max \left\{ B_1+2D_1B_2\rho^{-1} + (L A_1 + 2LD_1A_2 \rho^{-1})\mu^{-1},~~ \frac{C_1+2D_1C_2\rho^{-1}}{\mu \epsilon}\right\} \frac{L}{\mu}\log \frac{2\fgapp}{\epsilon}   \notag\\
	&=(B_1+2D_1B_2\rho^{-1} )\frac{L}{\mu}\log \frac{2\fgap}{\epsilon} 
	\end{align}
	since $A_1=A_2=C_1=C_2=0, \sigma_0^2 = \ns{x^0-w^0}=0$, and $\fgapp:= f(x^0) - f^* + 2^{-1}L\eta^2D_1\rho^{-1} \sigma_0^2 = f(x^0) - f^* =\fgap$.
	
	Now, the remaining thing is to upper bound the term $B_1+2D_1B_2\rho^{-1}$,
	\begin{align}
	B_1+2D_1B_2\rho^{-1} 
	&= 1+ \frac{2L^2}{b}\left(\frac{2\eta^2}{p}-\eta^2\right) \left(\frac{p}{2}+\frac{p^2}{2}-\frac{\eta^2L^2}{b}\right)^{-1} \label{eq:svrgpara-pl}\\
	& \leq 1+ \frac{2L^2}{b}\left(\frac{2\eta^2}{p}\right) \left(\frac{p}{4}\right)^{-1}      \label{eq:pb-pl} \\
	& = 1+ \frac{16L^2\eta^2}{bp^2}  \notag\\
	&\leq 1+ \frac{3}{b^{1/3}p^{2/3}}, \label{eq:plugeta-pl}  
	\end{align}
	where \eqref{eq:svrgpara-pl} follows from $B_1=1,  D_1=\frac{L^2}{b}, B_2=\frac{2\eta^2}{p}-\eta^2, \rho=\frac{p}{2}+\frac{p^2}{2}-\frac{\eta^2L^2}{b}$ in Lemma \ref{lem:lsvrg}, 
	the last inequality \eqref{eq:plugeta-pl} holds by setting 
	$\eta \leq \frac{1}{L(1+\frac{3}{b^{1/3}p^{2/3}})} 
	\leq \frac{1}{LB_1+2LD_1B_2\rho^{-1}}$, 
	and \eqref{eq:pb-pl} is due to the fact $\frac{\eta^2L^2}{b}\leq \frac{p}{4}$.
	
	In sum, let stepsize 
	$$	\eta \leq \frac{1}{L(1+\frac{3}{b^{1/3}p^{2/3}})},$$ 
	then the number of iterations performed by L-SVRG (Algorithm \ref{alg:lsvrg}) to find an $\epsilon$-solution of problem \eqref{eq:prob} with \eqref{prob:finite} can be bounded by
	$$
	K = \left(1+ \frac{3}{b^{1/3}p^{2/3}}\right)\frac{L}{\mu}\log\frac{2\fgap}{\epsilon}.
	$$
	In particular, we have
	\begin{enumerate}
		\item let minibatch size $b=1$ and probability $p=\frac{1}{n}$, then the number of iterations $K = \frac{4n^{2/3}L}{\mu}\log\frac{2\fgap}{\epsilon}$.
		\item let minibatch size $b=n^{2/3}$ and probability $p=\frac{1}{n^{1/3}}$, then the number of iterations $K =\frac{4L}{\mu}\log\frac{2\fgap}{\epsilon}$, but each iteration costs $n^{2/3}$ due to minibatch size $b=n^{2/3}$.
	\end{enumerate}
\end{proofof}

\subsection{SAGA method under PL condition}

\begin{corollary}[SAGA under PL condition]\label{cor:saga-pl}
	Suppose that Assumption \ref{asp:avgsmooth-saga} and \ref{asp:pl} hold.
	Let stepsize $\eta \leq \frac{1}{L(1+3n^{2/3}b^{-1})}$, 
	then the number of iterations performed by SAGA (Algorithm \ref{alg:saga}) for finding an $\epsilon$-solution of nonconvex problem \eqref{eq:prob} with \eqref{prob:finite}, i.e. a point $x^K$ such that $\E[f(x^K)-f^*] \leq \epsilon$, can be bounded by 
	$$
	K = \left(1+ \frac{3n^{2/3}}{b}\right)\frac{L}{\mu}\log\frac{2\fgap}{\epsilon}.
	$$
	In particular, we have
	\begin{enumerate}
		\item let minibatch size $b=1$, then the number of iterations $K = \frac{4n^{2/3}L}{\mu}\log\frac{2\fgap}{\epsilon}$.
		\item let minibatch size $b=n^{2/3}$, then iteration  $K =\frac{4L}{\mu}\log\frac{2\fgap}{\epsilon}$, but each iteration costs $n^{2/3}$ due to minibatch size $b=n^{2/3}$.
	\end{enumerate}
	Note that it recovers the previous result for SAGA under PL condition given by \citep{reddi2016proximal}, but \citet{reddi2016proximal} needed to restart SAGA for $O(\log \frac{1}{\epsilon})$ times to achieve this linear convergence rate.
\end{corollary}

\begin{proofof}{Corollary \ref{cor:saga-pl}}
	According to Theorem \ref{thm:main-pl}, the stepsize should be chosen as 
	\begin{align}
	\eta_k \equiv \eta &\leq \min\left\{ \frac{1}{LB_1+2LD_1B_2\rho^{-1} + (L A_1 + 2LD_1A_2 \rho^{-1})\mu^{-1}},~~ \frac{\mu\epsilon}{LC_1+2LD_1C_2\rho^{-1}}  \right\} \notag\\
	&= \frac{1}{LB_1+2LD_1B_2\rho^{-1}}
	\end{align}
	since $A_1=A_2=C_1=C_2=0$ according to Lemma \ref{lem:saga}.
	Then the number of iterations performed by SAGA (Algorithm \ref{alg:saga}) to find an $\epsilon$-solution of problem \eqref{eq:prob} with \eqref{prob:finite} can be bounded by
	\begin{align}
	K &= \max \left\{ B_1+2D_1B_2\rho^{-1} + (L A_1 + 2LD_1A_2 \rho^{-1})\mu^{-1},~~ \frac{C_1+2D_1C_2\rho^{-1}}{\mu \epsilon}\right\} \frac{L}{\mu}\log \frac{2\fgapp}{\epsilon}   \notag\\
	&=(B_1+2D_1B_2\rho^{-1} )\frac{L}{\mu}\log \frac{2\fgap}{\epsilon} 
	\end{align}
	since $A_1=A_2=C_1=C_2=0, \sigma_0^2 = \ns{x^0-w^0}=0$, and $\fgapp:= f(x^0) - f^* + 2^{-1}L\eta^2D_1\rho^{-1} \sigma_0^2 = f(x^0) - f^* =\fgap$.
	
	Now, the remaining thing is to upper bound the term $B_1+2D_1B_2\rho^{-1}$,
	\begin{align}
	B_1+2D_1B_2\rho^{-1} 
	&= 1+ \frac{2L^2}{b}\left(\frac{2\eta^2n}{b}-\eta^2\right) \left(\frac{b}{2n}+\frac{b^2}{2n^2}-\frac{\eta^2L^2}{b}\right)^{-1} \label{eq:sagapara-pl}\\
	& \leq 1+ \frac{2L^2}{b}\left(\frac{2\eta^2n}{b}\right) \left(\frac{b}{4n}\right)^{-1}      \label{eq:pb1-pl} \\
	& = 1+ \frac{16n^2L^2\eta^2}{b^3}  \notag\\
	&\leq 1+ \frac{3n^{2/3}}{b}, \label{eq:plugeta1-pl}  
	\end{align}
	where \eqref{eq:sagapara-pl} follows from $B_1=1,  D_1=\frac{L^2}{b},
	B_2=\frac{2\eta^2n}{b}-\eta^2,
	\rho=\frac{b}{2n}+\frac{b^2}{2n^2}-\frac{\eta^2L^2}{b}$ in Lemma \ref{lem:saga}, 
	the last inequality \eqref{eq:plugeta1-pl} holds by setting 
	$\eta \leq \frac{1}{L(1+\frac{3}{b^{1/3}p^{2/3}})} 
	\leq \frac{1}{LB_1+2LD_1B_2\rho^{-1}}$, 
	and \eqref{eq:pb1-pl} is due to the fact $\frac{\eta^2L^2}{b}\leq \frac{b}{4n}$.
	
	In sum, let stepsize 
	$$	\eta \leq \frac{1}{L(1+\frac{3n^{2/3}}{b})},$$ 
	then the number of iterations performed by SAGA (Algorithm \ref{alg:saga}) to find an $\epsilon$-solution of problem \eqref{eq:prob} with \eqref{prob:finite} can be bounded by
	$$
	K = \left(1+ \frac{3n^{2/3}}{b}\right)\frac{L}{\mu}\log\frac{2\fgap}{\epsilon}.
	$$
	In particular, we have
	\begin{enumerate}
		\item let minibatch size $b=1$, then the number of iterations $K = \frac{4n^{2/3}L}{\mu}\log\frac{2\fgap}{\epsilon}$.
		\item let minibatch size $b=n^{2/3}$, then iteration  $K =\frac{4L}{\mu}\log\frac{2\fgap}{\epsilon}$, but each iteration costs $n^{2/3}$ due to minibatch size $b=n^{2/3}$.
	\end{enumerate}
\end{proofof}

\newpage
\section{Better Convergence for Nonconvex Federated Optimization under PL Condition}

Similar to Section \ref{sec:pl}, we prove better convergence rates for the more general nonconvex distributed/federated optimization problem \eqref{eq:prob-fed} with online form \eqref{prob-fed:exp} or \eqref{prob-fed:finite} under the PL condition (i.e., Assumption \ref{asp:pl}).

\begin{equation*}
\min_{x\in \R^d} \bigg\{  f(x) := \frac{1}{m}\sum_{i=1}^m{f_i(x)}  \bigg\}, 
\text{~~where~}  f_i(x) := \E_{\zeta \sim \cD_i}[f_i(x,\zeta)],  
\text{~~or~~}  f_i(x) := \frac{1}{n}\sum_{j=1}^n{f_{i,j}(x)}.
\end{equation*}
Recall that we allow that different machine/worker $i\in [m]$ can have different data distribution $\cD_i$, i.e., non-IID data (heterogeneous data) setting.

In section \ref{sec:federated-app}, we have already proved that several (new) methods belonging to the proposed general DC framework (Algorithm \ref{alg:dc}) and DIANA framework (Algorithm \ref{alg:diana}) for solving this general distributed/federated problem also satisfy the unified Assumption \ref{asp:boge}
and thus can also be captured by our unified analysis.
In the following, we plug their corresponding parameters (i.e., specific values for $A_1, A_2, B_1, B_2, C_1,C_2,D_1,\rho$) into our unified Theorem \ref{thm:main-pl-dec} or \ref{thm:main-pl} under PL condition to obtain the detailed better convergence rates for these methods. 

Note that the proposed general DC framework (Algorithm \ref{alg:dc}) and DIANA framework (Algorithm \ref{alg:diana}) differing in direct gradient compression or compression of gradient differences.
We would like to point out that the direct gradient compression, i.e., DC framework (Algorithm \ref{alg:dc}) has a main issue.
Consider any stationary point $\hx$ such that $\nabla f(\hx) =\sum_{i=1}^{m} \nabla f_i(\hx) =0$, the issue is that the aggregated compressed gradient (even if the local stochastic gradient uses full gradient, i.e., $\widetilde{g}_i^k=\nabla f_i(x^k)$) does not converge to $0$, i.e., $g(\hx) = \frac{1}{m}\sum_{i=1}^m \cC_i(\nabla f_i(\hx)) \nRightarrow 0$.
However, the proposed DIANA Framework to compress the gradient differences (see Line \ref{line:comgrad-diana} of Algorithm \ref{alg:diana}) indeed can address this issue.
They are also reflected by our theoretical results, i.e., comparing the convergence results in the following Section \ref{sec:result-dc-pl} and Section \ref{sec:result-diana-pl}.

\subsection{Better convergence results of DC-type methods under PL condition}
\label{sec:result-dc-pl}
In this section, we provide the detailed convergence results under PL condition for some specific methods (i.e., DC-GD/SGD/LSVRG/SAGA) belonging to our DC framework (Algorithm \ref{alg:dc}) by using our unified Theorem \ref{thm:main-pl-dec} under PL condition.

\subsubsection{DC-GD method under PL condition}
\begin{corollary}[DC-GD under PL condition]\label{cor:dc-gd-pl}
	Suppose that Assumption \ref{asp:lsmooth-diana} and \ref{asp:pl} hold. 
	Let stepsize 
	$$\eta_{k} = \begin{cases}
	\eta   & \text{if~~}  k\leq \frac{K}{2}\\
	\frac{2\eta}{2+(k-\frac{K}{2})\mu\eta} &\text{if~~}  k>\frac{K}{2}
	\end{cases}, 	\text{~~~~where~~}
	\eta \leq \frac{1}{L+ L(1+\omega)Am^{-1}\mu^{-1}},
	$$
	then the number of iterations performed by DC-GD (Algorithm \ref{alg:dc-gd}) to find an $\epsilon$-solution of nonconvex federated problem \eqref{eq:prob-fed}, i.e. a point $x^K$ such that $\E[f(x^K)-f^*] \leq \epsilon$, can be bounded by
	\begin{align}
	K = \max \left\{\frac{2L( 1+ (1+\omega)Am^{-1}\mu^{-1})}{\mu}\log \frac{2\fgap}{\epsilon},~~ \frac{10(1+\omega)LC}{m\mu^2 \epsilon}\right\},
	\end{align}
	where 
	$A:=\max_i (L_i-L_i/(1+\omega))$,
	$C:= 2A\Delta_f^*$,
	and
	$\Delta_f^*:=f^*-\frac{1}{m}\summ f_i^*$.
\end{corollary}
\begin{proofof}{Corollary \ref{cor:dc-gd-pl}}
	According to Lemma \ref{lem:dc-gd}, we have proved that 
	$g^k$ (see Line \ref{line:gk-dc-gd} of DC-GD Algorithm \ref{alg:dc-gd})  satisfies the unified Assumption \ref{asp:boge} with
	\begin{align*}
	&A_1 =\frac{(1+\omega)A}{m}, \qquad
	B_1 =1, \qquad 
	C_1  =  \frac{(1+\omega)C}{m}, \\ 
	&D_1 =\frac{1+\omega}{m}, \qquad
	\sk \equiv 0, \qquad
	\rho =1, \\
	&A_2 =0, \qquad
	B_2 =0, \qquad 
	C_2  = 0,
	\end{align*}
	where 
	$A :=\max_i (L_i-L_i/(1+\omega))$,
	$C:= 2A\Delta_f^*$,
	and
	$\Delta_f^*:=f^*-\frac{1}{m}\summ f_i^*$.
	Then this corollary is proved by plugging these specific values for $A_1, A_2, B_1, B_2, C_1,C_2,D_1,\rho$ into our unified Theorem \ref{thm:main-pl-dec}. 
\end{proofof}

\subsubsection{DC-SGD method under PL condition}
\begin{corollary}[DC-SGD under PL condition]\label{cor:dc-sgd-pl}
	Suppose that Assumption \ref{asp:lsmooth-diana} and \ref{asp:pl} hold, and the local gradient estimator $\widetilde{g}_i^k$ in Algorithm \ref{alg:dc-sgd} satisfies Assumption \ref{asp:es}. 
	Let stepsize 
	
	$$\eta_{k} = \begin{cases}
	\eta   & \text{if~~}  k\leq \frac{K}{2}\\
	\frac{2\eta}{2+(k-\frac{K}{2})\mu\eta} &\text{if~~}  k>\frac{K}{2}
	\end{cases}, 	\text{~~~~where~~}
	\eta \leq \frac{1}{L+ L(1+\omega)Am^{-1}\mu^{-1}},
	$$
	then the number of iterations performed by DC-SGD (Algorithm \ref{alg:dc-sgd}) to find an $\epsilon$-solution of nonconvex federated problem \eqref{eq:prob-fed} with \eqref{prob-fed:exp} or \eqref{prob-fed:finite}, i.e. a point $x^K$ such that $\E[f(x^K)-f^*] \leq \epsilon$, can be bounded by
	\begin{align}
	K = \max \left\{\left(1+ \frac{(1+\omega)A}{m\mu}\right)\frac{2L}{\mu}\log \frac{2\fgap}{\epsilon},~~ \frac{10(1+\omega)LC}{m\mu^2 \epsilon}\right\},
	\end{align}
	where 
	$A:=\max_i  (A_{1,i}+B_{1,i}L_i-L_i/(1+\omega))$,
	$C := \frac{1}{m} \summ C_{1,i} + 2A\Delta_f^*$
	and
	$\Delta_f^*:=f^*-\frac{1}{m}\summ f_i^*$.
\end{corollary}
\begin{proofof}{Corollary \ref{cor:dc-sgd-pl}}
	According to Lemma \ref{lem:dc-sgd}, we have proved that 
	$g^k$ (see Line \ref{line:gk-dc-sgd} of DC-SGD Algorithm \ref{alg:dc-sgd})  satisfies the unified Assumption \ref{asp:boge} with
	\begin{align*}
	&A_1 =\frac{(1+\omega)A}{m}, \qquad
	B_1 =1, \qquad 
	C_1  =  \frac{(1+\omega)C}{m}, \\ 
	&D_1 =\frac{1+\omega}{m}, \qquad
	\sk \equiv 0, \qquad
	\rho =1, \\
	&A_2 =0, \qquad
	B_2 =0, \qquad 
	C_2  = 0,
	\end{align*}
	where 
	$A:=\max_i  (A_{1,i}+B_{1,i}L_i-L_i/(1+\omega))$,
	$C := \frac{1}{m} \summ C_{1,i} + 2A\Delta_f^*$
	and
	$\Delta_f^*:=f^*-\frac{1}{m}\summ f_i^*$.
	Then this corollary is proved by plugging these specific values for $A_1, A_2, B_1, B_2, C_1,C_2,D_1,\rho$ into our unified Theorem \ref{thm:main-pl-dec}. 
\end{proofof}	

\subsubsection{DC-LSVRG method under PL condition}
\begin{corollary}[DC-LSVRG under PL condition]\label{cor:dc-lsvrg-pl}
	Suppose that Assumption \ref{asp:lsmooth-diana}, \ref{asp:avgsmooth-fed} and \ref{asp:pl} hold.
	Let stepsize 
	$$\eta_{k} = \begin{cases}
	\eta   & \text{if~~}  k\leq \frac{K}{2}\\
	\frac{2\eta}{2+(k-\frac{K}{2})\mu\eta} &\text{if~~}  k>\frac{K}{2}
	\end{cases}, $$
	where
	$$\eta \leq \frac{1}{L+ L(1+\omega)(1+2\tau\rho^{-1})Am^{-1}\mu^{-1} + 2L(1+\omega)Bm^{-1}b^{-1}\rho^{-1}},
	$$
	then the number of iterations performed by DC-LSVRG (Algorithm \ref{alg:dc-lsvrg}) to find an $\epsilon$-solution of nonconvex federated problem \eqref{eq:prob-fed} with \eqref{prob-fed:finite}, i.e. a point $x^K$ such that $\E[f(x^K)-f^*] \leq \epsilon$, can be bounded by
	\begin{align}
	K = \max \left\{\left( 1+ \frac{1+\omega}{m}\left(\frac{(1+2\tau\rho^{-1})A}{\mu} +\frac{2B}{b\rho}\right)\right)\frac{2L}{\mu}\log \frac{2\fgap}{\epsilon},~ \frac{10(1+\omega)(1+2\tau\rho^{-1})LC}{m\mu^2 \epsilon}\right\},
	\end{align}
	where 
	$A:=\max_i (L_i-L_i/(1+\omega))$,
	$B: = \bar{L}^2 ((1-p)\eta^2\gamma^{-1} + \eta^2)$,
	$C:= 2A\Delta_f^*$,
	$\Delta_f^*:=f^*-\frac{1}{m}\summ f_i^*$,
	$\rho :=p+p\gamma-\gamma-\tau$,
	$\tau:= \frac{(1+\omega)\bar{L}^2 \eta^2}{mb}$
	and $\forall\gamma>0$.
\end{corollary}
\begin{proofof}{Corollary \ref{cor:dc-lsvrg-pl}}
	According to Lemma \ref{lem:dc-lsvrg}, we have proved that 
	$g^k$ (see Line \ref{line:gk-dc-lsvrg} of DC-LSVRG Algorithm \ref{alg:dc-lsvrg})  satisfies the unified Assumption \ref{asp:boge} with
	\begin{align*}
	&A_1 =\frac{(1+\omega)A}{m}, \qquad
	B_1 =1, \qquad 
	C_1  =  \frac{(1+\omega)C}{m}, \\ 
	&D_1 =\frac{1+\omega}{m}, \qquad
	\sk =\frac{\bar{L}^2}{b}\ns{\xk-\wk}, \qquad
	\rho =p+p\gamma-\gamma-\tau, \\
	&A_2 =\tau A, \qquad
	B_2 = \frac{\bar{L}^2 ((1-p)\eta^2\gamma^{-1} + \eta^2)}{b}, \qquad 
	C_2  = \tau C,
	\end{align*}
	where 
	$A:=\max_i (L_i-L_i/(1+\omega))$,
	$C:= 2A\Delta_f^*$,
	$\Delta_f^*:=f^*-\frac{1}{m}\summ f_i^*$,
	$\tau:= \frac{(1+\omega)\bar{L}^2 \eta^2}{mb}$
	and $\forall\gamma>0$.
	Then this corollary is proved by plugging these specific values for $A_1, A_2, B_1, B_2, C_1,C_2,D_1,\rho$ into our unified Theorem \ref{thm:main-pl-dec}. 
\end{proofof}	

\subsubsection{DC-SAGA method under PL condition}
\begin{corollary}[DC-SAGA under PL condition]\label{cor:dc-saga-pl}
	Suppose that Assumption \ref{asp:lsmooth-diana}, \ref{asp:avgsmooth-saga-fed} and \ref{asp:pl} hold.
	Let stepsize 
	$$\eta_{k} = \begin{cases}
	\eta   & \text{if~~}  k\leq \frac{K}{2}\\
	\frac{2\eta}{2+(k-\frac{K}{2})\mu\eta} &\text{if~~}  k>\frac{K}{2}
	\end{cases},$$ 	
	where
	$$\eta \leq \frac{1}{L+ L(1+\omega)(1+2\tau\rho^{-1})Am^{-1}\mu^{-1} + 2L(1+\omega)Bm^{-1}b^{-1}\rho^{-1}},
	$$
	then the number of iterations performed by DC-SAGA (Algorithm \ref{alg:dc-saga}) to find an $\epsilon$-solution of nonconvex federated problem \eqref{eq:prob-fed} with \eqref{prob-fed:finite}, i.e. a point $x^K$ such that $\E[f(x^K)-f^*] \leq \epsilon$, can be bounded by
	\begin{align}
	K = \max \left\{\left( 1+ \frac{1+\omega}{m}\left(\frac{(1+2\tau\rho^{-1})A}{\mu} +\frac{2B}{b\rho}\right)\right)\frac{2L}{\mu}\log \frac{2\fgap}{\epsilon},~ \frac{10(1+\omega)(1+2\tau\rho^{-1})LC}{m\mu^2 \epsilon}\right\},
	\end{align}
	where 
	$A:=\max_i (L_i-L_i/(1+\omega))$,
	$B: = \bar{L}^2 ((1-\frac{b}{n})\eta^2\gamma^{-1} + \eta^2)$,
	$C:= 2A\Delta_f^*$,
	$\Delta_f^*:=f^*-\frac{1}{m}\summ f_i^*$,
	$\rho :=\frac{b}{n}+\frac{b}{n}\gamma-\gamma-\tau$,
	$\tau:= \frac{(1+\omega)\bar{L}^2 \eta^2}{mb}$
	and $\forall\gamma>0$.
\end{corollary}
\begin{proofof}{Corollary \ref{cor:dc-saga-pl}}
	According to Lemma \ref{lem:dc-saga}, we have proved that 
	$g^k$ (see Line \ref{line:gk-dc-saga} of DC-SAGA Algorithm \ref{alg:dc-saga})  satisfies the unified Assumption \ref{asp:boge} with
	\begin{align*}
	&A_1 =\frac{(1+\omega)A}{m}, \qquad
	B_1 =1, \qquad 
	C_1  =  \frac{(1+\omega)C}{m}, \\ 
	&D_1 =\frac{1+\omega}{m}, \qquad
	\sk =\frac{1}{mn}\summ \sum_{j=1}^{n}\frac{\bar{L}^2}{b}\ns{\xk-w_{i,j}^k}, \qquad
	\rho =\frac{b}{n}+\frac{b}{n}\gamma-\gamma-\tau, \\
	&A_2 =\tau A, \qquad
	B_2 = \frac{\bar{L}^2 ((1-\frac{b}{n})\eta^2\gamma^{-1} + \eta^2)}{b}, \qquad 
	C_2  = \tau C,
	\end{align*}
	where 
	$A:=\max_i (L_i-L_i/(1+\omega))$,
	$C:= 2A\Delta_f^*$,
	$\Delta_f^*:=f^*-\frac{1}{m}\summ f_i^*$,
	$\tau:= \frac{(1+\omega)\bar{L}^2 \eta^2}{mb}$
	and $\forall\gamma>0$.
	Then this corollary is proved by plugging these specific values for $A_1, A_2, B_1, B_2, C_1,C_2,D_1,\rho$ into our unified Theorem \ref{thm:main-pl-dec}. 
\end{proofof}

\subsection{Better convergence results of DIANA-type methods under PL Condition}
\label{sec:result-diana-pl}
In this section, we provide the detailed convergence results under PL condition for some specific methods (i.e., DIANA-GD/SGD/LSVRG/SAGA) belonging to our DIANA framework (Algorithm \ref{alg:diana}) by using our unified Theorem \ref{thm:main-pl-dec} or \ref{thm:main-pl} under PL condition.

\subsubsection{DIANA-GD method under PL condition}
\begin{corollary}[DIANA-GD under PL condition]\label{cor:diana-gd-pl}
	Suppose that Assumption \ref{asp:lsmooth-diana} and \ref{asp:pl} hold. 
	Let stepsize
	$$
	\eta \leq  \frac{1}{L+2L(1+\omega)Bm^{-1}\rho^{-1}},
	$$
	then the number of iterations performed by DIANA-GD (Algorithm \ref{alg:diana-gd}) to find an $\epsilon$-solution of nonconvex federated problem \eqref{eq:prob-fed}, i.e. a point $x^K$ such that $\E[f(x^K)-f^*] \leq \epsilon$, can be bounded by
	\begin{align}
	K = \left( 1+ \frac{2(1+\omega)B}{m\rho}\right)\frac{L}{\mu}\log \frac{2\fgap}{\epsilon},
	\end{align}
	where $B:=\frac{\omega(1+\beta)L^2\eta^2}{1+\omega}$, 
	$\rho:=\min\{1-\tau,~ 2\alpha -(1-\alpha)\beta^{-1} -\alpha^2-\tau \}$,
	$\tau:=\alpha^2 \omega + \frac{(1+\omega)B}{m}$,
	and $\forall \beta>0$.
\end{corollary}

\begin{proofof}{Corollary \ref{cor:diana-gd-pl}}
	According to Lemma \ref{lem:diana-gd}, we have proved that 
	$g^k$ (see Line \ref{line:gk-diana-gd} of DIANA-GD Algorithm \ref{alg:diana-gd})  satisfies the unified Assumption \ref{asp:boge} with
	\begin{align*}
	&A_1 =0, \qquad
	B_1 =1, \qquad 
	C_1  =  0, \\ 
	&D_1 =\frac{1+\omega}{m}, \qquad
	\sk =  \frac{\omega}{(1+\omega)m}
	\summ \ns{\nabla f_i(\xk) -\hi}, \\
	&\rho =\min\{1-\tau,~ 2\alpha -(1-\alpha)\beta^{-1} -\alpha^2-\tau \},\\
	&A_2 =0, \qquad
	B_2 =B, \qquad 
	C_2  = 0,
	\end{align*}
	where 
	$B:=\frac{\omega(1+\beta)L^2\eta^2}{1+\omega}$,
	$\tau:=\alpha^2 \omega + \frac{(1+\omega)B}{m}$,
	and $\forall \beta>0$.
	Then this corollary is proved by plugging these specific values for $A_1, A_2, B_1, B_2, C_1,C_2,D_1,\rho$ into our unified Theorem \ref{thm:main-pl}. 
\end{proofof}	

\subsubsection{DIANA-SGD method under PL condition}
\begin{corollary}[DIANA-SGD under PL condition]\label{cor:diana-sgd-pl}
	Suppose that Assumption \ref{asp:lsmooth-diana} and \ref{asp:pl} hold, and the local gradient estimator $\widetilde{g}_i^k$ in Algorithm \ref{alg:diana-sgd} satisfies Assumption \ref{asp:es}. 
	Let stepsize 
	$$\eta_{k} = \begin{cases}
	\eta   & \text{if~~}  k\leq \frac{K}{2}\\
	\frac{2\eta}{2+(k-\frac{K}{2})\mu\eta} &\text{if~~}  k>\frac{K}{2}
	\end{cases},$$ 	
	where
	$$\eta \leq \frac{1}{L+ L(1+\omega)(1+2\tau\rho^{-1})Am^{-1}\mu^{-1} + 2L(1+\omega)Bm^{-1}\rho^{-1}},
	$$
	then the number of iterations performed by DIANA-SGD (Algorithm \ref{alg:diana-sgd}) to find an $\epsilon$-solution of nonconvex federated problem \eqref{eq:prob-fed} with \eqref{prob-fed:exp} or \eqref{prob-fed:finite}, i.e. a point $x^K$ such that $\E[f(x^K)-f^*] \leq \epsilon$, can be bounded by
	\begin{align}
	K = \max \left\{\left( 1+ \frac{1+\omega}{m}\left(\frac{(1+2\tau\rho^{-1})A}{\mu} +\frac{2B}{\rho}\right)\right)\frac{2L}{\mu}\log \frac{2\fgap}{\epsilon},~ \frac{10(1+\omega)(1+2\tau\rho^{-1})LC}{m\mu^2 \epsilon}\right\},
	\end{align}
	where 
	$A:=\max_i  (A_{1,i}+(B_{1,i}-1)L_i)$, 
	$B:=\frac{\omega(1+\beta)L^2\eta^2}{1+\omega}$, 
	$C := \frac{1}{m} \summ C_{1,i} + 2A\Delta_f^*$,
	$\Delta_f^*:=f^*-\frac{1}{m}\summ f_i^*$,
	$\rho :=\min\{1-\tau, 2\alpha -(1-\alpha)\beta^{-1} -\alpha^2-\tau \}$,
	$\tau:=\alpha^2 \omega + \frac{(1+\omega)B}{m}$,
	and $\forall \beta>0$.
\end{corollary}

\begin{proofof}{Corollary \ref{cor:diana-sgd-pl}}
	According to Lemma \ref{lem:diana-sgd}, we have proved that 
	$g^k$ (see Line \ref{line:gk-diana-sgd} of DIANA-SGD Algorithm \ref{alg:diana-sgd})  satisfies the unified Assumption \ref{asp:boge} with
	\begin{align*}
	&A_1 =\frac{(1+\omega)A}{m}, \qquad
	B_1 =1, \qquad 
	C_1  =  \frac{(1+\omega)C}{m}, \\ 
	&D_1 =\frac{1+\omega}{m}, \qquad
	\sk = \frac{\omega}{(1+\omega)m}
	\summ \ns{\nabla f_i(\xk) -\hi}, \\
	&\rho =\min\{1-\tau,~ 2\alpha -(1-\alpha)\beta^{-1} -\alpha^2-\tau \},\\
	&A_2 =\tau A, \qquad
	B_2 =B, \qquad 
	C_2  = \tau C,
	\end{align*}
	where 
	$A:=\max_i  (A_{1,i}+(B_{1,i}-1)L_i)$, 
	$B:=\frac{\omega(1+\beta)L^2\eta^2}{1+\omega}$,
	$C := \frac{1}{m} \summ C_{1,i} + 2A\Delta_f^*$,
	$\Delta_f^*:=f^*-\frac{1}{m}\summ f_i^*$,
	$\tau:=\alpha^2 \omega + \frac{(1+\omega)B}{m}$,
	and $\forall\beta>0$.
	Then this corollary is proved by plugging these specific values for $A_1, A_2, B_1, B_2, C_1,C_2,D_1,\rho$ into our unified Theorem \ref{thm:main-pl-dec}. 
\end{proofof}	

\subsubsection{DIANA-LSVRG method under PL condition}
\begin{corollary}[DIANA-LSVRG under PL condition]\label{cor:diana-lsvrg-pl}
	Suppose that Assumption \ref{asp:lsmooth-diana}, \ref{asp:avgsmooth-fed} and \ref{asp:pl} hold.
	Let stepsize 
	$$
	\eta \leq \frac{1}{L+2L(1+\omega)B'm^{-1}b^{-1}\rho^{-1}},
	$$
	then the number of iterations performed by DIANA-LSVRG (Algorithm \ref{alg:diana-lsvrg}) to find an $\epsilon$-solution of nonconvex federated problem \eqref{eq:prob-fed} with \eqref{prob-fed:finite}, i.e. a point $x^K$ such that $\E[f(x^K)-f^*] \leq \epsilon$, can be bounded by
	\begin{align}
	K = \left( 1+ \frac{2(1+\omega)B'}{mb\rho}\right)\frac{L}{\mu}\log \frac{2\fgap}{\epsilon},
	\end{align}
	where 
	$B' :=  (1-p)\bar{L}^2\eta^2\gamma^{-1} + Bb^{-1}$,
	$\rho :=\min\{p+p\gamma-\gamma-\tau,~ 2\alpha -(1-\alpha)\beta^{-1} -\alpha^2-\tau \}$,
	$\tau:=\alpha^2 \omega + \frac{(1+\omega)B}{m}$,
	$B:=\frac{\omega(1+\beta)L^2\eta^2}{1+\omega} + \bar{L}^2\eta^2 b^{-1}$,
	and $\forall \gamma, \beta>0$.
\end{corollary}

\begin{proofof}{Corollary \ref{cor:diana-lsvrg-pl}}
	According to Lemma \ref{lem:diana-lsvrg}, we have proved that 
	$g^k$ (see Line \ref{line:gk-diana-lsvrg} of DIANA-LSVRG Algorithm \ref{alg:diana-lsvrg})  satisfies the unified Assumption \ref{asp:boge} with
	\begin{align*}
	&A_1 =0, \qquad
	B_1 =1, \qquad 
	C_1  =  0, \\ 
	&D_1 =\frac{1+\omega}{m}, \qquad
	\sk = \frac{\bar{L}^2}{b}\ns{\xk-\wk} + \frac{\omega}{(1+\omega)m}
	\summ \ns{\nabla f_i(\xk) -\hi}, \\
	&\rho =\min\{p+p\gamma-\gamma-\tau,~ 2\alpha -(1-\alpha)\beta^{-1} -\alpha^2-\tau \},\\
	&A_2 =0 \qquad
	B_2 =(1-p)\bar{L}^2\eta^2\gamma^{-1} b^{-1} +B, \qquad 
	C_2  = 0,
	\end{align*}
	where 
	$B:=\frac{\omega(1+\beta)L^2\eta^2}{1+\omega} + \bar{L}^2\eta^2 b^{-1}$,
	$\tau:=\alpha^2 \omega + \frac{(1+\omega)B}{m}$,
	and $\forall \gamma, \beta>0$.
	Then this corollary is proved by plugging these specific values for $A_1, A_2, B_1, B_2, C_1,C_2,D_1,\rho$ into our unified Theorem \ref{thm:main-pl}. 
\end{proofof}

\subsubsection{DIANA-SAGA method under PL condition}
\begin{corollary}[DIANA-SAGA under PL condition]\label{cor:diana-saga-pl}
	Suppose that Assumption \ref{asp:lsmooth-diana}, \ref{asp:avgsmooth-saga-fed} and \ref{asp:pl} hold.
	Let stepsize 
	$$
	\eta \leq \frac{1}{L+2L(1+\omega)B'm^{-1}b^{-1}\rho^{-1}},
	$$
	then the number of iterations performed by DIANA-SAGA (Algorithm \ref{alg:diana-saga}) to find an $\epsilon$-solution of nonconvex federated problem \eqref{eq:prob-fed} with \eqref{prob-fed:finite}, i.e. a point $x^K$ such that $\E[f(x^K)-f^*] \leq \epsilon$, can be bounded by
	\begin{align}
	K = \left( 1+ \frac{2(1+\omega)B'}{mb\rho}\right)\frac{L}{\mu}\log \frac{2\fgap}{\epsilon},
	\end{align}
	where 
	$B' :=  (1-\frac{b}{n})\bar{L}^2\eta^2\gamma^{-1} + Bb^{-1}$,
	$\rho :=\min\{\frac{b}{n}+\frac{b}{n}\gamma-\gamma-\tau,~ 2\alpha -(1-\alpha)\beta^{-1} -\alpha^2-\tau \}$,
	$\tau:=\alpha^2 \omega + \frac{(1+\omega)B}{m}$,
	$B:=\frac{\omega(1+\beta)L^2\eta^2}{1+\omega} + \bar{L}^2\eta^2 b^{-1}$,
	and $\forall \gamma, \beta>0$.
\end{corollary}

\begin{proofof}{Corollary \ref{cor:diana-saga-pl}}
	According to Lemma \ref{lem:diana-saga}, we have proved that 
	$g^k$ (see Line \ref{line:gk-diana-saga} of DIANA-SAGA Algorithm \ref{alg:diana-saga})  satisfies the unified Assumption \ref{asp:boge} with
	\begin{align*}
	&A_1 =0, \qquad
	B_1 =1, \qquad 
	C_1  =  0, \\ 
	&D_1 =\frac{1+\omega}{m}, \qquad
	\sk = \frac{1}{mn}\summ \sum_{j=1}^{m}\frac{\bar{L}^2}{b}\ns{\xk-w_{i,j}^k} + \frac{\omega}{(1+\omega)m}
	\summ \ns{\nabla f_i(\xk) -\hi}, \\
	&\rho =\min\left\{\frac{b}{n}+\frac{b}{n}\gamma-\gamma-\tau,~ 2\alpha -(1-\alpha)\beta^{-1} -\alpha^2-\tau \right\},\\
	&A_2 =0 \qquad
	B_2 =(1-\frac{b}{n})\bar{L}^2\eta^2\gamma^{-1} b^{-1} +B, \qquad 
	C_2  = 0,
	\end{align*}
	where 
	$B:=\frac{\omega(1+\beta)L^2\eta^2}{1+\omega} + \bar{L}^2\eta^2 b^{-1}$,
	$\tau:=\alpha^2 \omega + \frac{(1+\omega)B}{m}$,
	and $\forall \gamma, \beta>0$.
	Then this corollary is proved by plugging these specific values for $A_1, A_2, B_1, B_2, C_1,C_2,D_1,\rho$ into our unified Theorem \ref{thm:main-pl}. 
\end{proofof}

\end{document}